\documentclass[11pt]{article}
\newtheorem{theorem}{Theorem}
\newtheorem{lemma}{Lemma}
\newtheorem{definition}{Definition}
\newtheorem{proof}{Proof}
\usepackage{hyperref}
\newtheorem{remark}{Remark}
\usepackage{amsmath}
\usepackage{graphicx}
\usepackage{amssymb}
\usepackage{graphicx}
\usepackage[dvipsnames,usenames]{color}
\usepackage{subfigure}
\usepackage{multirow}
\usepackage[extra]{tipa}
\usepackage{array}
\usepackage[affil-it]{authblk}
\title{On partially homogeneous nearest-neighbour random walks in the quarter plane and their application in the analysis of two-dimensional queues with limited state-dependency\footnote{Accepted in \textit{Queueing Systems}}}

\author{Ioannis Dimitriou \footnote{E-mail: idimit@math.upatras.gr\\Website: \href{https://thalis.math.upatras.gr/~idimit/}{https://thalis.math.upatras.gr/~idimit/}}}
\affil{\small Department of Mathematics, 
University of Patras, P.O.~Box 
26504, Patras, Greece.}
\begin{document}
\maketitle
\begin{abstract} 
This work deals with the stationary analysis of two-dimensional partially homogeneous nearest-neighbour random walks. Such type of random walks are characterized by the fact that the one-step transition probabilities are functions of the state-space. We show that its stationary behaviour is investigated by solving a finite system of linear equations, two matrix functional equations, and a functional equation with the aid of the theory of Riemann (-Hilbert) boundary value problems. This work is strongly motivated by emerging applications in flow level performance of wireless networks that give rise in queueing models with scalable service capacity, as well as in queue-based random access protocols, where the network's parameters are functions of the queue lengths. A simple numerical illustration, along with some details on the numerical implementation are also presented. 
\vspace{2mm}
\end{abstract}
\section{Introduction}
In this work, we focus on the stationary analysis of irreducible discrete time Markov chains in the quarter plane $\mathbb{Z}_{+}^{2}$ (where $\mathbb{Z}_{+}$ refers to the set of non-negative integers), whose one-step transition probabilities possess a partial homogeneity property. More precisely, we focus on nearest-neighbour two-dimensional random walks with one-step transition probabilities defined as follows: transitions from an interior point $(n_{1},n_{2})\in\{1,2,\ldots\}\times\{1,2,\ldots\}$ of the state space lead with probability $p_{i,j}(n_{1},n_{2})$ to a neighbouring point $(n_{1}+i,n_{2}+j)$, where $(i,j)\in\{-1,0,1\}\times\{-1,0,1\}$. 

Such a class of state-dependent two-dimensional random walks is instrumental in the analytical investigation of a large class of queueing networks with \textit{interacting queues}, as well as with \textit{scalable service capacity}, i.e., the system parameters are functions of the state of the network. However, the stationary analysis of general state-dependent two-dimensional random walks is still an open problem. 

In this paper, we focus on the partial-homogeneous nearest neighbour random walks in the quarter plane (PH-NNRWQP; see Figure \ref{fig10}), which obey the following property: There exist two positive integers, say $N_{1}>1$, $N_{2}>1$, such that the state space $S=\{(n_{1},n_{2});n_{1},n_{2}\geq0\}$ is split in four non-intersecting subsets, i.e., $S=S_{a}\cup S_{b}\cup S_{c}\cup S_{d}$, where:
\begin{equation}
\begin{array}{rl}
S_{a}=\{(n_{1},n_{2});n_{1}<N_{1},n_{2}<N_{2}\},&S_{b}=\{(n_{1},n_{2});n_{1}\geq N_{1},n_{2}<N_{2}\},\\
S_{c}=\{(n_{1},n_{2});n_{1}<N_{1},n_{2}\geq N_{2}\},&S_{d}=\{(n_{1},n_{2});n_{1}\geq N_{1},n_{2}\geq N_{2}\},
\end{array}\label{split}
\end{equation}
such that for $i,j=0,\pm 1$,
\begin{equation}
\begin{array}{rl}
p_{i,j}(n_{1},n_{2})=&\left\{\begin{array}{ll}
p_{i,j}(N_{1},n_{2}),&(n_{1},n_{2})\in S_{b},\\
p_{i,j}(n_{1},N_{2}),&(n_{1},n_{2})\in S_{c},\\
p_{i,j}(N_{1},N_{2}):=p_{i,j},&(n_{1},n_{2})\in S_{d}.
\end{array}\right.
\end{array}\label{trans}
\end{equation}
Our aim in this work is to provide an analytical approach to investigate its stationary behavior, and state the importance of the models described by such a class of RWQP in the modelling (among others) of emerging engineering applications in multiple access systems.
\begin{figure}[htp]
\centering
\includegraphics[scale=0.5]{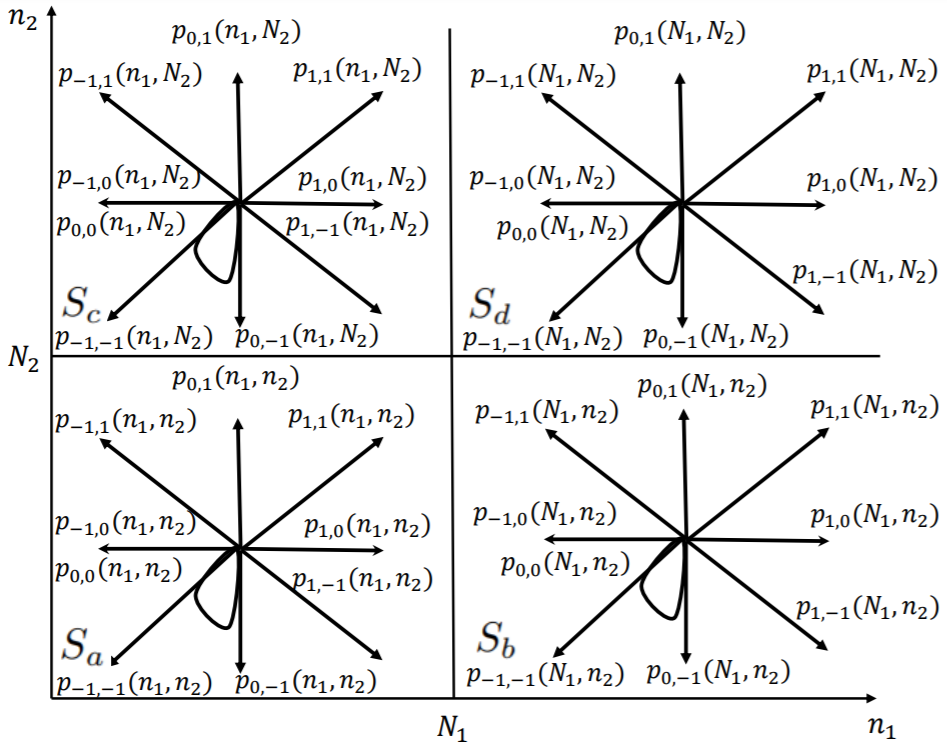}
\caption{The PH-NNRWQP.}
\label{fig10}
\end{figure}
\subsection{Related work}
Since the pioneered works \cite{fay2,maly3}, RWQP has been extensively studied as an important topic of applied probability. More precisely, there are strong links in $i$) queueing theory, e.g., the seminal books \cite{fay1,coh1} provided the foundations of a quite general theory on their stationary behaviour based on the theory of boundary value problems; see also the review in \cite{coh2} on the use of the generating function approach, as well as the works in \cite{flat,c2} that focused on the join the shortest queue problem. $ii$) Applications in finance are also found, e.g., in \cite{cont} where order book events were described
in terms of a two-dimensional Markovian queue. $iii)$ Combinatorics is another field of interest, e.g., \cite{adRL,BM1,rasc}, where the authors focused on methods of counting lattice walks in the quarter plane. $iv)$ Exact tail asymptotics of the stationary joint queue-length distribution of several two-dimensional queues based on a Tauberian-type theorem were studied in \cite{lizha1,lizha2,lizha3}. In \cite{miya}, an extensive overview on light tail asymptotics was given, where in \cite{guileu}, a large deviation approach was presented; see also \cite{oza} for related work in a general setting (non-exhaustive list).

The main body of the related literature is devoted to the analysis of \textit{semi-homogeneous} NNRWQP. With the term \textit{semi-homogeneous}, we mean that the transition probabilities are state-independent in so far as it concerns states belonging to the interior, i.e., $\{(n_{1},n_{2}),n_{1},n_{2}>0\}$, similarly for those of $\{(n_{1},n_{2}),n_{1}>0,n_{2}=0\}$, of $\{(n_{1},n_{2}),n_{1}=0,n_{2}>0\}$ and of $\{(0,0)\}$. Most of the research refers to the investigation of ergodicity conditions; e.g., \cite{rw,fay}. The stationary analysis of \textit{semi-homogeneous} NNRWQP is challenging, and is performed with the aid of the theory of boundary value problems \cite{fay1,coh1}. 

Explicit conditions for recurrence and transience were given in the seminal works in \cite{maly1,malmen}. For a detailed treatment see \cite{fay}. Ergodicity conditions for the \textit{partially homogeneous} case (see Fig. \ref{fig10}) were considered in \cite{malmen} under the assumption that the jumps of the RWQP are bounded. It was later considered in \cite{fayo} under a weaker restriction that the jumps of the RWQP have bounded second moments; see also \cite{fay}. A profound study concerning necessary and sufficient conditions for ergodicity of general RWQP that are continuous to the \textit{West}, to the \textit{South-West} and to the \textit{South}, is given in \cite[Part II]{rw}. We finally mention the book in \cite[Chapters 6-8]{boro}, where stability conditions for general spatially homogeneous multidimensional Markov chains were investigated. 

For a detailed methodological treatment of the stationary analysis of semi-homogeneous RWQP, the reader is referred to the seminal books \cite{fay1,coh1}. In \cite{fay1}, the analysis is concentrated mainly to nearest neighbour RWQPs, while in \cite{coh1} the authors provided a systematic study of RWQP that are continuous to the \textit{West}, to the \textit{South-West} and to the \textit{South}. 

An effective analytical approach can be applied for NNRWQP whose transitions to the \textit{North}, to the \textit{North-East} and to the \textit{East} are not allowed. In particular, it is shown that the probability generating function (pgf) of the stationary joint distribution of the random walk can be explicitly expressed in terms of meromorphic functions; see e.g., the work in \cite{Adan2016} where the shortest queue polling model was investigated, the works in \cite{c3,c2} that focused on the join the shortest queue (JSQ) problem, and the work in \cite{c1} that dealt with the asymmetric clocked buffered switch. 

The analysis becomes quite harder when space homogeneity property collapses. Such a situation arises in the two-dimensional JSQ problem, in which the quarter plane is separated into two homogeneous regions \cite{fay1,coh1}. In these studies, the analysis is reduced to the simultaneous solution of two boundary value problems. In \cite{king}, the author considered the symmetric shortest queue problem, and provided an analytic method to obtain its stationary distribution. An efficient method to obtain the stationary joint queue length distribution in the shortest queue problem (and several of its variants) was developed in \cite{Adan2016,ad1,ad2,ad3} (i.e., the compensation method), by transforming the original RWQP to a random walk with no transitions to the East, to the North-East and to North. Such an approach, provides an explicit characterization of the equilibrium probabilities as an infinite or finite series of product forms.

In \cite{fayo}, the authors considered a two-dimensional birth-death process (i.e., transitions were allowed only to the North, East, West and South) with partial homogeneity that separate the state space in four distinct regions. They showed that its stationary behavior is investigated by solving a boundary value problem and a system of linear equations. Stability conditions for such type of random walks were investigated in \cite{za,faysta}. 
\subsection{Our contribution}\label{contr}
\subsubsection{Fundamental contribution}
In this work, we present an analytical method for analysing the stationary behaviour of a partially homogeneous nearest-neighbour random walk in the quarter plane whose one step transition probabilities for $(n_{1},n_{2})\in S=S_{a}\cup S_{b}\cup S_{c}\cup S_{d}$, and $i,j=0,\pm 1$ are as in (\ref{trans}).
\begin{itemize}
\item We show that the determination of the steady-state distribution of a PH-NNRWQP can be reduced to the solution of $i)$ a finite system of linear equations, $ii)$ of two matrix functional equations, as well as $iii)$ of a non-homogeneous Riemann boundary value problem. 
\item Random walks in the quarter plane with limited state-dependency, can also be used to approximate general state-dependent two-dimensional random walks. This could be done by choosing arbitrarily two integers $N_{1}$, $N_{2}$ and truncating the state space as in \eqref{split}, and by setting the one-step transition probabilities in
subregions $S_{b}$, $S_{c}$, $S_{d}$ in \eqref{trans} equal to zero. Another, possibly more efficient, approach is to use \eqref{trans}, and construct a close approximation of a general state-dependent random walk in the quarter plane. Such an approximation is expected to be better compared with the finite one.
\item In subsection \ref{appl2}, we observe that the analysis presented in Section \ref{pre} can be also applied to the case of non-nearest neighbour two-dimensional random walks with bounded jumps. There, we consider a discrete-time generalized processor sharing two-queue system with limited state-dependency and impatient jobs. In such a model, we may have at most two simultaneous departures from each job type in a time slot. Our work shows that the theory of boundary value problems can be applied to the case where we can have a fixed number of simultaneous departures from either type (i.e., bounded batch departures). Note that such a situation has never reported in the related literature.
\end{itemize}

This work is strongly motivated by the applications of PH-NNRWQP in queueing systems, for which, both the server behaviour and the input streams are, in a non-trivial way, affected by the queue lengths at both queues. This work provides an \textit{analytical method} to study two-dimensional queues with \textit{limited} state-dependency, and show how this state-dependency can be used to improve their performance. To our best knowledge, there is no other unified analytical study that refers to the performance modelling of two-dimensional queueing models with general arrival/service scheduling function. 

It will be more apparent in the following that our modelling framework incorporates a wide range of service disciplines, and further generalizes them, by allowing combinations of service scheduling policies according to the splitting rule \eqref{split}. Such a flexibility is achieved by considering different scheduling policies in each subregion $S_{j}$, $j=a,b,c,d$. A thorough discussion on various scheduling policies based on our framework is given in Appendix \ref{discu}.
\subsubsection{Applications} In the following, we briefly present two fundamental queueing models described as PH-NNRWQP obeying the splitting scheme in \eqref{split}, and for which, there are no analytical results in the related literature.
\paragraph{Limited Discrete-time Generalized Processor Sharing queues (LDGPS)} Limited processor sharing (LPS) queues arise naturally when the server's capacity is equally shared to a limited number of jobs. Despite their numerous applications, there are very few studies even on the continuous time LPS queues, referring mainly to a single queue systems \cite{zwart,zwart1}. An extension to a layered queueing network was recently given in \cite{ave}; for other applications in communication systems see e.g., \cite{telek,guptamor}. A quite flexible scheduling discipline for systems that handle several classes of jobs is Generalized
Processor Sharing (GPS). Under GPS, each class is assigned a weight, and the available capacity of the server is proportionally divided amongst the backlogged classes. Our discrete time modelling framework can be seen as a practical implementation that solves the infinite divisibility assumption of continuous-time GPS.

The LDGPS system operates as follows: up to $N_1$ jobs of class 1 and up to $N_2$ jobs of class 2 are allowed to share the processor; the remaining jobs (if any) wait in their
respective queues. Instead of a weight that corresponds to the allocation of processor capacity, each class of jobs is assigned a probability that may depend in an
arbitrary way on the numbers of class 1 and class 2 jobs sharing the processor,
but not on the total number waiting in the system. In such a setting, if packets of both classes are backlogged, the oldest packet
of class 1 is served with probability $\beta_{1}(n_{1},n_{2})$ whereas with probability $\beta_{2}(n_{1},n_{2})=1-\beta_{1}(n_{1},n_{2})$ the oldest packet of class 2 is served. The LDGPS is a generalization of the model introduced recently in \cite{walr}.

Our modelling framework incorporates also the concept of state-dependent impatience, by allowing waiting jobs to depart from the network without service, with a probability that depends on the number of waiting jobs. Such a situation arises when we investigate the flow-level performance in wireless networks with users' mobility; e.g. \cite{bonald,sima,bor2}. Moreover, by introducing the concept of impatience we may have a bounded number of departures, resulting in a non-neighbour PH-RWQP; for more details see subsection \ref{appl2}. To our best knowledge, no analytical results are available in the literature for LDGPS systems with/without impatience.
\paragraph{Queue-based two-node random access (RA) networks}
The ever increasing need for massive uncoordinated access, due to the huge number of communicating devices
in 5G networks, has increased the interest
on RA protocols \cite{metis}. In these networks, there is no centralized access control, and the nodes operate autonomously by sharing the medium in a distributed fashion. Since any node decides individually to transmit, there exists a high level of interaction that affects the successful transmission probabilities, i.e., concurrent transmissions result in collisions. Such an interaction gives rise to quite challenging mathematical problems even for very simple networks \cite{ephr}. 

Our modelling framework can be used to investigate the delay analysis of slotted time queue-based two-node RA networks (see Fig. \ref{fig1w}), and may serve as a macroscopic model and as a building block to investigate large-scale networks. In such a network, users communicate in RA manner with a common destination node (node D) as shown in Figure \ref{fig1w}. 
\begin{figure}[htp]
\centering
\includegraphics[scale=0.3]{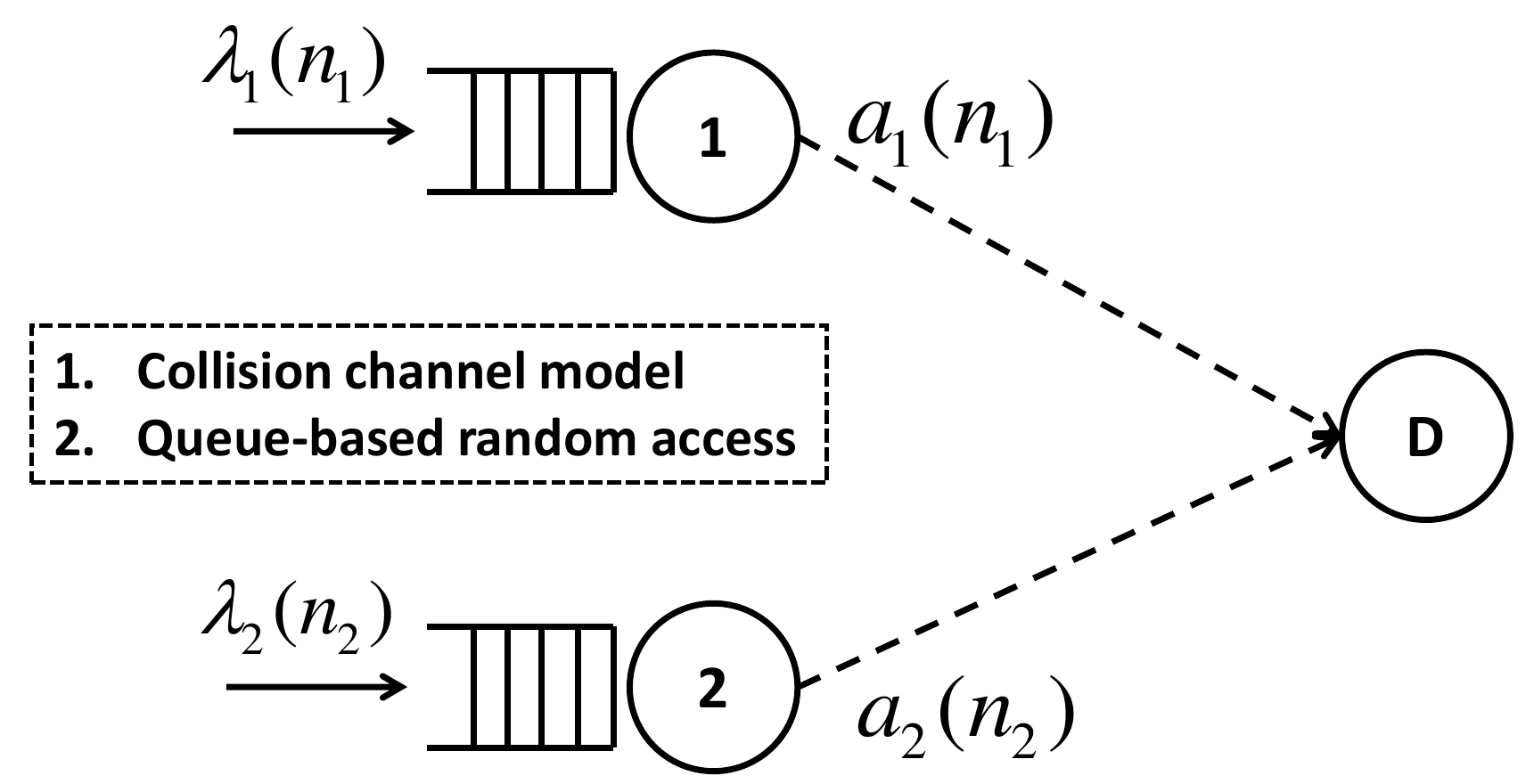}
\caption{The queue-based random access network.}
\label{fig1w}
\end{figure}
Note that packet generation (i.e., packet arrivals) and/or packet transmission (i.e., packet transmissions to the node D) parameters depend on the user state (i.e., backlogged packets at user's buffer). However, since user nodes have limited memory and capacity, they adapt their parameters based on an upper level of their buffer occupancy rather than on all the backlogged packets; i.e., a limited state dependent dynamic transmission/packet generation policy is employed. For more details see subsection \ref{appl1}.

The paper is organized as follows. In Section \ref{mod}, we present the mathematical model. Our main results are given in Section \ref{pre} by focusing on the stationary analysis of a general PH-NNRWQP. Section \ref{appl} is devoted to two novel queueing models that are analysed based on our framework, i.e., the \textit{two-queue LDGPS with impatience}; see subsection \ref{appl2}, and the \textit{two-node queue-based RA network}; see subsection \ref{appl1}. A simple numerical illustration along with some details on the numerical implementation are given in Section \ref{num}. 
\section{Model description}\label{mod}
We consider a partially homogeneous two-dimensional Markov chain $\mathbf{Q}=\{(Q_{1,n},Q_{2,n});n=0,1,\ldots\}$, with state space $S=\mathbb{Z}_{+}^{2}$. For the complete description of the structure of $\mathbf{Q}$ we need the following assumptions:
\begin{itemize}
\item There exist two positive constants, say $N_{1}$, $N_{2}$, such that $S=S_{a}\cup S_{b}\cup S_{c}\cup S_{d}$, where the non-intersecting sets $S_{j}$, $j=a,b,c,d$ are given in (\ref{split}).
\item For $j=a,b,c,d$, let the sequence of independent stochastic vectors $\{(\xi_{1n}^{(j)}(Q_{1,n},Q_{2,n}),\xi_{2n}^{(j)}(Q_{1,n},Q_{2,n})),n=0,1,\ldots\}$, with range space $\{-1,0,1\}\times\{-1,0,1\}$ (in subsection \ref{appl2} we relax this assumption). Their distribution depends on the state of $\mathbf{Q}$ according to the state-space splitting rule, as shown in Figure \ref{fig10}.
\item The family $\{(\xi_{1n}^{(d)}(Q_{1,n},Q_{2,n}),\xi_{2n}^{(d)}(Q_{1,n},Q_{2,n})),n=0,1,\ldots\}$ is a sequence of i.i.d. stochastic vectors. Moreover, $(\xi_{1n}^{(d)}(Q_{1,n},Q_{2,n}),\xi_{2n}^{(d)}(Q_{1,n},Q_{2,n}))\sim(\xi_{1n}^{(d)},\xi_{2n}^{(d)})$, for $(Q_{1,n},Q_{2,n})\in S_{d}$. 
\item The four families $\{(\xi_{1n}^{(j)}(Q_{1,n},Q_{2,n}),\xi_{2n}^{(j)}(Q_{1,n},Q_{2,n}))\}$, $j=a,b,c,d$ are independent, and \\$(\xi_{1n}^{(j)}(Q_{1,n},Q_{2,n}),\xi_{2n}^{(j)}(Q_{1,n},Q_{2,n}))\sim(\xi_{1}^{(j)}(Q_{1},Q_{2}),\xi_{2}^{(j)}(Q_{1},Q_{2}))$. 
\end{itemize}
Then, for $n=0,1,\ldots$, $k=1,2,$ and $(Q_{1,0},Q_{2,0})=(0,0)$
\begin{displaymath}
Q_{k,n+1}=[Q_{k,n}+\xi_{n}^{(j)}(Q_{1,n},Q_{2,n})]^{+},
\end{displaymath}
where $[a]^{+}=max(0,a)$. Conditions for ergodicity for such two-dimensional random walks, have been investigated in \cite[Theorem 3.1, p. 178]{faysta}, \cite[Theorem 4]{za}, as well as in \cite[Chapters 6-8]{boro}. More precisely, in \cite{faysta}, using the notion of induced Markov chains \cite{fay}, we realize that for $Q_{1,n}>N_{1}$ (resp. $Q_{2,n}>N_{2}$) the component $Q_{2,n}$ (resp. $Q_{1,n}$) evolves as a one-dimensional random walk (RW). Denote its corresponding stationary distribution by $\psi:=(\psi_{0},\psi_{1},...)$ (resp. $\varphi:=(\varphi_{0},\varphi_{1},...)$). Consider the mean drifts 
\begin{displaymath}
\begin{array}{rl}
E^{(n_{2})}_{x}:=&\mathbb{E}(Q_{1,n+1}-Q_{1,n}|\mathbf{Q}=(n_{1},n_{2})),\,n_{1}\geq N_{1},n_{2}\geq 0,\\
E^{(n_{1})}_{y}:=&\mathbb{E}(Q_{2,n+1}-Q_{2,n}|\mathbf{Q}=(n_{1},n_{2})),\,n_{1}\geq 0,n_{2}\geq N_{2}.
\end{array}
\end{displaymath}
Note that for $(n_{1},n_{2})\in S_{d}$,
\begin{displaymath}
\begin{array}{rl}
E^{(n_{2})}_{x}:=&E_{x}=p_{1,0}+p_{1,1}+p_{1,-1}-(p_{-1,1}+p_{-1,0}+p_{-1,-1}),\\
E^{(n_{1})}_{y}:=&E_{y}=p_{0,1}+p_{1,1}+p_{-1,1}-(p_{-1,-1}+p_{0,-1}+p_{1,-1}).
\end{array}
\end{displaymath}
Let also $\rho_{1}=\sum_{n_{2}=0}^{\infty}\psi_{n_{2}}E_{x}^{(n_{2})}$, $\rho_{2}=\sum_{n_{1}=0}^{\infty}\varphi_{n_{1}}E_{y}^{(n_{1})}$. Then, the following theorem provides the stability conditions of $\textbf{Q}$. 
\begin{theorem}\label{faystab}
(\cite[Theorem 3.1]{faysta}) \begin{itemize}
\item If $E_{x}<0$, $E_{y}<0$, then, $\textbf{Q}$ is ergodic if $a)$ $\rho_{1}<0$, and $\rho_{2}<0$, $b)$ is transient if $\rho_{1}>0$, or $\rho_{2}>0$.
\item If $E_{x}\geq 0$, $E_{y}<0$, then, $\textbf{Q}$ is ergodic if $a)$ $\rho_{1}<0$, $b)$ is transient if $\rho_{1}>0$. It is also transient if $E_{x}>0$ and $\rho_{1}=0$.
\item If $E_{x}< 0$, $E_{y}\geq 0$, then, $\textbf{Q}$ is ergodic if $a)$ $\rho_{2}<0$, $b)$ is transient if $\rho_{2}>0$. It is also transient if $E_{y}>0$ and $\rho_{2}=0$.
\item If $E_{x}\geq 0$, $E_{y}\geq 0$, then, $\textbf{Q}$ is transient.
\end{itemize}
\end{theorem}
\section{Analysis}\label{pre}
Assume here on that for the Markov chain $\textbf{Q}$ is ergodic and there exists a unique stationary probability distribution,
\begin{displaymath}\begin{array}{c}
\pi(n_{1},n_{2})=\lim_{n\to\infty}P(Q_{1,n}=n_{1},Q_{2,n}=n_{2}).
\end{array}
\end{displaymath}
Then, for $\underline{n}=(n_{1},n_{2})\in S$, the equilibrium equations read
\begin{equation}
\begin{array}{l}
\pi(\underline{n})=\pi(\underline{n})p_{0,0}(\underline{n})+\pi(\underline{n}+\underline{1}_{2})p_{0,-1}(\underline{n}+\underline{1}_{2})
+\pi(\underline{n}+\underline{1}_{1})p_{-1,0}(\underline{n}+\underline{1}_{1})\\+\pi(\underline{n}-\underline{1}_{1}+\underline{1}_{2})p_{1,-1}(\underline{n}-\underline{1}_{1}+\underline{1}_{2})\mathbf{1}_{\{n_{1}\geq1\}}\\
+\pi(\underline{n}+\underline{1}_{1}-\underline{1}_{2})p_{-1,1}(\underline{n}+\underline{1}_{1}-\underline{1}_{2})\mathbf{1}_{\{n_{2}\geq1\}}+\pi(\underline{n}-\underline{1}_{2})p_{0,1}(\underline{n}-\underline{1}_{2})\mathbf{1}_{\{n_{2}\geq1\}}\\
+\pi(\underline{n}+\underline{1}_{1}+\underline{1}_{2})p_{-1,-1}(\underline{n}+\underline{1}_{1}+\underline{1}_{2})+\pi(\underline{n}-\underline{1}_{1})p_{1,0}(\underline{n}-\underline{1}_{1})\mathbf{1}_{\{n_{1}\geq1\}}\\+\pi(\underline{n}-\underline{1}_{1}-\underline{1}_{2})p_{1,1}(\underline{n}-\underline{1}_{1}-\underline{1}_{2})\mathbf{1}_{\{n_{1},n_{2}\geq1\}},
\end{array}\label{e1}
\end{equation}
where $\sum_{n_{1}=0}^{\infty}\sum_{n_{2}=0}^{\infty}\pi(n_{1},n_{2})=1$, $\pi(n_{1},-1)=0=\pi(-1,n_{2})$, $\mathbf{1}_{\{A\}}$ is the indicator function of the event $A$, and $\underline{1}_{i}$, $i=1,2,$ is the $1\times 2$ row vector with 1 in the $i$-th place, and zero elsewhere. 

In order to help the reader to follow the approach, we briefly present its basic steps, by investigating each subregion separately.
\begin{enumerate}
\item[Step 1] Our primary aim is to express all the unknown equilibrium probabilities in terms of the probabilities associated with region $S_{a}$, and $\pi(N_{1},N_{2})$. There are $N_{1}\times N_{2}$ equilibrium equations that refer to $S_{a}$. These equations contain $(N_{1}+1)\times(N_{2}+1)$ unknown probabilities: $N_{1}\times N_{2}$ of them refer to states that belong to $S_{a}$ (i.e., $\pi(n_{1},n_{2})$, $n_{1}=0,1,\ldots,N_{1}-1$, $n_{1}=0,1,\ldots,N_{1}-1$), and $N_{1}+N_{2}+1$ refer to states that belong to the boundary of $S_{b}$ with $S_{a}$ (i.e., $\pi(N_{1},n_{2})$, $n_{2}=0,1,\ldots,N_{2}-1$), to the boundary of $S_{c}$ with $S_{a}$ (i.e., $\pi(n_{1},N_{2})$, $n_{1}=0,1,\ldots,N_{1}-1$), and to the boundary of $S_{d}$ with $S_{a}$ (i.e., $\pi(N_{1},N_{2})$).
\item[Step 2] For the equilibrium equations that refer to $S_{b}$ (i.e., $n_{1}=N_{1},N_{1}+1,\ldots$, $n_{2}=0,1,\ldots N_{2}-1$), and to $S_{c}$ ((i.e., $n_{2}=N_{2},N_{2}+1,\ldots$, $n_{1}=0,1,\ldots N_{1}-1$)), we apply the generating function technique. Such a procedure leads to two matrix equations (see \eqref{sys1}, \eqref{sys2} below). The solution of \eqref{sys1} provides the pgf of $\pi(n_{1},n_{2})$, $n_{1}=N_{1},N_{1}+1,\ldots$, $n_{2}=1,\ldots,N_{2}$ (i.e., the function $g_{n_{2}}(x)$; see \eqref{r2}) in terms of:
\begin{itemize}
\item[-] the pgf of $\pi(n_{1},0)$ (i.e., the function $g_{0}(x)$), $n_{1}=N_{1},N_{1}+1,\ldots$,
\item[-] the probabilities $\pi(N_{1}-1,n_{2})$, $\pi(N_{1},n_{2})$, $n_{2}=0,1,\ldots,N_{2}$. Note that these unknown probabilities are a subset of the unknown probabilities mentioned in step 1.
\end{itemize}
Similarly, the solution of \eqref{sys2} provides the pgf of $\pi(n_{1},n_{2})$, $n_{2}=N_{2},N_{2}+1,\ldots$, $n_{1}=1,\ldots,N_{1}$ (i.e., the function $h_{n_{1}}(y)$; see \eqref{r3}) in terms of:
\begin{itemize}
\item[-] the pgf of $\pi(0,n_{2})$,  (i.e., the function $h_{0}(y)$) $n_{2}=N_{2},N_{2}+1,\ldots$,
\item[-] the probabilities $\pi(n_{1}, N_{2}-1)$, $\pi(n_{1},N_{2})$, $n_{1}=0,1,\ldots,N_{1}$, which are also a subset of the unknown probabilities mentioned in step 1.
\end{itemize}
\item[Step 3] We apply the generating function approach for the equilibrium equations referring to subregion $S_{d}$. Then, after some heavy, but straightforward, algebra, we come up with a functional equation, the solution of which, provides the pgf of the joint equilibrium probabilities for the states in $S_{d}$.  However, this equation is written in terms of $g_{N_{2}}(x)$, $g_{N_{2}-1}(x)$, $h_{N_{1}}(x)$, $h_{N_{1}-1}(x)$. Substituting the results obtained in step 2, we come up with the functional equation \eqref{fun}, which now contains the two unknown functions $g_{0}(x)$, $h_{0}(y)$ mentioned in step 2. The theory of boundary value problems is the tool to obtain $g_{0}(x)$, $h_{0}(y)$, and as a consequence, to solve \eqref{fun}. Such a solution is obtained in terms of the unknown probabilities mentioned in step 1. Having obtained $g_{0}(x)$, $h_{0}(y)$, and using the normalizing condition, we obtain the additional $N_{1}+N_{2}+1$ equations for the corresponding unknown probabilities, in terms of the derivatives of $g_{0}(x)$, $h_{0}(y)$ with respect to $x$, $y$ at points $x=0$, $y=0$, respectively. 

Then, going back to step 1, we solve an $(N_{1}+1)\times(N_{2}+1)$ system of independent linear equations to obtain $\pi(n_{1},n_{2})$, $n_{1}=0,1,\ldots,N_{1}$, $n_{2}=0,1,\ldots,N_{2}$. With this result, the pgfs of the equilibrium probabilities in subregions $S_{b}$, $S_{c}$, $S_{d}$ are known.  
\item[Step 4] Having obtained the equilibrium probabilities associated with subregion $S_{a}$, and the pgfs that refer to the the equilibrium probabilities in subregions $S_{b}$, $S_{c}$, $S_{d}$, we can obtain various stationary metrics, such as the stationary expected values of $\mathbf{Q}$.
\end{enumerate}  
\subsection{Generating functions and the functional equation}\label{gen}
We proceed with the generating function approach by focusing at each sub-region of the state space separately.
\begin{enumerate}
\item \textbf{Region $S_{a}$:} Consider the equilibrium equations (\ref{e1}) corresponding to the region $S_{a}$. There are $N_{1}\times N_{2}$ equations ($n_{1}=0,1,...,N_{1}-1$, $n_{2}=0,1,...,N_{2}-1$) involving $(N_{1}+1)\times(N_{2}+1)$ unknown probabilities ($\pi(n_{1},n_{2})$, $n_{1}=0,1,...,N_{1}$, $n_{2}=0,1,...,N_{2}$). This leaves $N_{1}+N_{2}+1$ unknowns, and as mentioned above, we seek to find additional equations for these unknowns. The following procedure helps us to identify these equations. Then, along with the equilibrium equations that are associated with $S_{a}$, an independent linear system of $(N_{1}+1)\times(N_{2}+1)$ equations will be constructed. 
\item \textbf{Region $S_{b}$:} We now focus on the equations related with the region $S_{b}$ ($n_{1}=N_{1},N_{1}+1,...$, $n_{2}=0,1,...N_{2}-1$). Let
\begin{displaymath}
\begin{array}{c}
g_{n_{2}}(x)=\sum_{n_{1}=N_{1}}^{\infty}\pi(n_{1},n_{2})x^{n_{1}-N_{1}},\,n_{2}=0,1,...,\,|x|\leq1.
\end{array}
\end{displaymath}
Having in mind that $p_{i,j}(n_{1},n_{2})=p_{i,j}(N_{1},n_{2})$ for $n_{1}\geq N_{1}$, we obtain from (\ref{e1}) the following relations,
\begin{equation}
\begin{cases}\begin{array}{l}
f_{2}(N_{1},0,x)g_{0}(x)-f_{3}(N_{1},1,x)g_{1}(x)=b_{0}(x),\vspace{2mm}\\
-f_{1}(N_{1},n_{2}-1,x)g_{n_{2}-1}(x)+f_{2}(N_{1},n_{2},x)g_{n_{2}}(x)\\-f_{3}(N_{1},n_{2}+1,x)g_{n_{2}+1}(x)=b_{n_{2}}(x),\,n_{2}=1,2,...,
\end{array}\end{cases}
\label{r1}
\end{equation} 
where for $n_{2}=0,1,2,...$,
\begin{displaymath}
\begin{array}{rl}
f_{1}(N_{1},n_{2},x)=&x^{2}p_{1,1}(N_{1},n_{2})+xp_{0,1}(N_{1},n_{2})+p_{-1,1}(N_{1},n_{2}),\vspace{2mm}\\
f_{2}(N_{1},n_{2},x)=&x[1-p_{0,0}(N_{1},n_{2})]-x^{2}p_{1,0}(N_{1},n_{2})-p_{-1,0}(N_{1},n_{2}),\vspace{2mm}\\
f_{3}(N_{1},n_{2},x)=&p_{-1,-1}(N_{1},n_{2})+xp_{0,-1}(N_{1},n_{2})+x^{2}p_{1,-1}(N_{1},n_{2}),\end{array}
\end{displaymath}
\begin{displaymath}
\begin{array}{l}
b_{n_{2}}(x)=x[\pi(N_{1}-1,n_{2}-1)p_{1,1}(N_{1}-1,n_{2}-1)\\+\pi(N_{1}-1,n_{2}+1)p_{1,-1}(N_{1}-1,n_{2}+1)
+\pi(N_{1}-1,n_{2})p_{1,0}(N_{1}-1,n_{2})]\vspace{2mm}\\-[\pi(N_{1},n_{2}-1)p_{-1,1}(N_{1},n_{2}-1)+\pi(N_{1},n_{2})p_{-1,0}(N_{1},n_{2})\vspace{2mm}\\+\pi(N_{1},n_{2}+1)p_{-1,-1}(N_{1},n_{2}+1)].
\end{array}
\end{displaymath}

Relations (\ref{r1}) allow to explicitly express $g_{n_{2}}(x)$, $n_{2}=1,2,...$, in terms of $g_{0}(x)$ and $b_{0}(x),...,b_{n_{2}-1}(x)$. More precisely, we straightforwardly express $g_{n_{2}}(x)$ in terms of $g_{0}(x)$, and the unknown probabilities $\pi(N_{1}-1,n_{2})$, $\pi(N_{1},n_{2})$ for $n_{2}=0,1,\ldots,N$. Indeed, let
\begin{displaymath}
\begin{array}{rl}
\textbf{l}(x):=&(g_{1}(x),\ldots,g_{N_{2}}(x))^{\prime},\\
\textbf{b}(x):=&(b_{0}(x),\ldots,b_{N_{2}-1}(x))^{\prime},\\
\textbf{c}_{1}(x):=&(-f_{2}(N_{1},0,x),f_{1}(N_{1},0,x),0,\ldots,0)^{\prime},
\end{array}
\end{displaymath}
where $\mathbf{x}^{\prime}$ denotes the transpose of the vector $\mathbf{x}$. Then, the system of equations (\ref{r1}) is written in matrix form as
\begin{equation}
\begin{array}{c}
\textbf{K}(x)\textbf{l}(x)=\textbf{c}_{1}(x)g_{0}(x)+\textbf{b}(x),\,|x|\leq1,
\end{array}
\label{sys1}
\end{equation}
where $\textbf{K}(x):=(k_{i,j}(x))$ is a $N_{2}\times N_{2}$ matrix with elements
\begin{displaymath}
\begin{array}{rl}
k_{i,j}(x)=&\left\{\begin{array}{ll}
-f_{3}(N_{1},i,x),&i=j,\\
f_{2}(N_{1},i-1,x),&i=j+1,\\
-f_{1}(N_{1},j,x),&i=j+2.\\
\end{array}\right.
\end{array}
\end{displaymath}

Note that system (\ref{sys1}) is non-singular, since $\textbf{K}(x)$ is lower triangular, having determinant equal to $det(\textbf{K}(x))=(-1)^{N_{2}}\prod_{n_{2}=1}^{N_{2}}f_{3}(N_{1},n_{2},x)$ for $|x|\leq 1$ (Moreover, the roots of $f_{3}(N_{1},n_{2},x)=0$, $|x|\leq 1$ are the poles of $g_{n_{2}}(x)$, $n_{2}=1,\ldots,N_{2}$). Furthermore, note that $\textbf{K}(x)$ and $\textbf{c}_{1}(x)$ contain known elements (i.e., their elements are polynomials whose coefficients are the one-step transition probabilities). On the other hand, the elements of $\textbf{b}(x)$ contain the unknown probabilities $\pi(N_{1}-1,n_{2})$, $\pi(N_{1},n_{2})$, $n_{2}=0,1,\ldots,N_{2}$. Applying the Cramer's rule yields,
\begin{displaymath}
\begin{array}{c}
g_{n_{2}}(x)=\frac{det(\textbf{K}^{(n_{2})}(x))}{det(\textbf{K}(x))},\,n_{2}=1,\ldots,N_{2},
\end{array}
\end{displaymath}
where $\textbf{K}^{(n_{2})}(x)$, is obtained from the matrix $\textbf{K}(x)$ by replacing $n_{2}-$th column by the vector of free terms $\textbf{c}(x)g_{0}(x)+\textbf{b}(x)$. Moreover, $\textbf{K}^{(n_{2})}(x)=\textbf{K}_{1}^{(n_{2})}(x)g_{0}(x)+\textbf{K}_{2}^{(n_{2})}(x)$, where $\textbf{K}_{1}^{(n_{2})}(x)$, $\textbf{K}_{2}^{(n_{2})}(x)$, are obtained from $\textbf{K}^{(n_{2})}(x)$ by replacing its $n_{2}$-th column with the vector $\textbf{c}(x)$ and $\textbf{b}(x)$, respectively. Therefore, it is straightforward to show that for $n_{2}=1,2,...,N_{2},$
\begin{equation}
\begin{array}{rl}
g_{n_{2}}(x)=&\frac{det(\textbf{K}_{1}^{(n_{2})}(x))g_{0}(x)+det(\textbf{K}_{2}^{(n_{2})}(x))}{det(\textbf{K}(x))}\\
=&e_{n_{2}}(x)g_{0}(x)+t_{n_{2}}(x),\,|x|\leq1,
\end{array}\label{r2}
\end{equation}
where $e_{n_{2}}(x)$, $n_{2}=1,2,...,N_{2},$ are known functions of $x$ containing one-step transition probabilities, while $t_{n_{2}}(x)$, $n_{2}=1,2,...,N_{2},$ are also functions of $x$ containing the unknown probabilities mentioned above. 

Note that up to now, no other new unknown probabilities appear, except $\pi(N_{1}-1,n_{2})$, $\pi(N_{1},n_{2})$, $n_{2}=0,1,\ldots,N_{2}$.  
\item \textbf{Region $S_{c}$:} Clearly, region $S_{c}$ is a mirror image of $S_{b}$. Let,
\begin{displaymath}
\begin{array}{c}
h_{n_{1}}(y)=\sum_{n_{2}=N_{2}}^{\infty}\pi(n_{1},n_{2})y^{n_{2}-N_{2}},\,n_{1}=0,1,...,\,|y|\leq1.
\end{array}
\end{displaymath}
By repeating the procedure described above, we obtain 
\begin{equation}
\begin{cases}\begin{array}{l}
\tilde{f}_{2}(0,N_{2},y)h_{0}(y)-\tilde{f}_{3}(1,N_{2},y)h_{1}(y)=u_{0}(y),\vspace{2mm}\\
-\tilde{f}_{1}(n_{1}-1,N_{2},y)h_{n_{1}-1}(y)+\tilde{f}_{2}(n_{1},N_{2},y)h_{n_{1}}(y)\\-\tilde{f}_{3}(n_{1}+1,N_{2}y)h_{n_{1}+1}(y)=u_{n_{2}}(y),\,n_{1}=1,2,...,
\end{array}\end{cases}
\label{r11}
\end{equation} 
where for $n_{1}=0,1,2,...$,
\begin{displaymath}
\begin{array}{rl}
\tilde{f}_{1}(n_{1},N_{2},y)=&y^{2}p_{1,1}(n_{1},N_{2})+yp_{1,0}(n_{1},N_{2})+p_{1,-1}(n_{1},N_{2}),\vspace{2mm}\\
\tilde{f}_{2}(n_{1},N_{2},y)=&y[1-p_{0,0}(n_{1},N_{2})]-y^{2}p_{0,1}(n_{1},N_{2})-p_{0,-1}(n_{1},N_{2}),\vspace{2mm}\\
\tilde{f}_{3}(n_{1},N_{2},y)=&y^{2}p_{-1,1}(n_{1},N_{2})+yp_{-1,0}(n_{1},N_{2})+p_{-1,-1}(n_{1},N_{2}),
\end{array}
\end{displaymath}
\begin{displaymath}
\begin{array}{l}
u_{n_{1}}(y)=y[\pi(n_{1}-1,N_{2}-1)p_{1,1}(n_{1}-1,N_{2}-1)\\+\pi(n_{1}+1,N_{2}-1)p_{-1,1}(n_{1}+1,N_{2}-1)
+\pi(n_{1},N_{2}-1)p_{0,1}(n_{1},N_{2}-1)]\vspace{2mm}\\-[\pi(n_{1}-1,N_{2})p_{1,-1}(n_{1}-1,N_{2})+\pi(n_{1},N_{2})p_{0,-1}(n_{1},N_{2})\vspace{2mm}\\+\pi(n_{1}+1,N_{2})p_{-1,-1}(n_{1}+1,N_{2})].
\end{array}
\end{displaymath}
Similarly, we express $h_{n_{1}}(y)$ in terms of $h_{0}(y)$, and the unknown probabilities $\pi(n_{1},N_{2}-1)$, $\pi(n_{1},N_{2})$ for $n_{1}=0,1,\ldots,N_{1}$. Let
\begin{displaymath}
\begin{array}{rl}
\textbf{j}(y):=&(h_{1}(y),\ldots,h_{N_{1}}(y))^{\prime},\\
\textbf{u}(y):=&(u_{0}(y),\ldots,u_{N_{1}-1}(y))^{\prime},\\
\textbf{c}_{2}(y):=&(-\tilde{f}_{2}(0,N_{2},y),\tilde{f}_{1}(0,N_{2},y),0,\ldots,0)^{\prime}.
\end{array}
\end{displaymath}
Then, the system of equations (\ref{r11}) is written in matrix form as
\begin{equation}
\begin{array}{c}
\textbf{M}(y)\textbf{j}(y)=\textbf{c}_{2}(y)h_{0}(x)+\textbf{u}(y),\,|y|\leq1,
\end{array}
\label{sys2}
\end{equation}
where $\textbf{M}(y):=(m_{i,j}(y))$ is a $N_{1}\times N_{1}$ matrix with elements
\begin{displaymath}
\begin{array}{rl}
m_{i,j}(y)=&\left\{\begin{array}{ll}
-\tilde{f}_{3}(i,N_{2},y),&i=j,\\
\tilde{f}_{2}(j-1,N_{2},y),&j=i+1,\\
-\tilde{f}_{1}(j,N_{2},x),&j=i+2.\\
\end{array}\right.
\end{array}
\end{displaymath}
Note also that the elements of $\textbf{M}(y)$, and $\textbf{c}_{2}(y)$, are polynomials of the one-step transition probabilities, while the elements of $\textbf{u}(y)$ contain the unknown probabilities $\pi(n_{1},N_{2}-1)$, $\pi(n_{1},N_{2})$, for $n_{1}=0,1,\ldots,N_{1}$.

Since $det(\textbf{M}(y))=(-1)^{N_{1}}\prod_{n_{1}=1}^{N_{2}}\tilde{f}_{3}(n_{1},N_{2},y)\neq 0$, $|y|\leq 1$, by applying the Cramer's rule we have
\begin{displaymath}
\begin{array}{c}
h_{n_{1}}(y)=\frac{det(\textbf{M}^{(n_{1})}(y))}{det(\textbf{M}(y))},\,n_{1}=1,\ldots,N_{1},
\end{array}
\end{displaymath}
where $\textbf{M}^{(n_{1})}(y)$, is obtained from the matrix $\textbf{M}(y)$ by replacing the $n_{1}-$th column by the vector of free terms $\textbf{c}_{1}(y)h_{0}(y)+\textbf{u}(y)$. Moreover, $\textbf{M}^{(n_{1})}(y)=\textbf{M}_{1}^{(n_{1})}(y)h_{0}(y)+\textbf{M}_{2}^{(n_{1})}(y)$, where $\textbf{M}_{1}^{(n_{1})}(y)$, $\textbf{M}_{2}^{(n_{1})}(y)$ are obtained from $\textbf{M}^{(n_{1})}(y)$, by replacing its $n_{1}$-th column with the vector $\textbf{c}_{1}(y)$ and $\textbf{u}(y)$, respectively. Therefore, for $n_{1}=1,2,...,N_{1}$
\begin{equation}
\begin{array}{rl}
h_{n_{1}}(y)=&\frac{det(\textbf{M}_{1}^{(n_{1})}(y))h_{0}(y)+det(\textbf{M}_{2}^{(n_{1})}(y))}{det(\textbf{M}(y))}\\
=&\tilde{e}_{n_{1}}(y)h_{0}(y)+\tilde{t}_{n_{1}}(y),\,\,|y|\leq1,
\end{array}
\label{r3}
\end{equation}
where $\tilde{e}_{n_{1}}(y)$, $n_{1}=1,2,...,N_{1},$ are known functions of $y$ containing one-step transition probabilities, while $\tilde{t}_{n_{1}}(y)$, $n_{1}=1,2,...,N_{1},$ are also functions of $y$ containing the unknown probabilities mentioned above. 
Note that up to now, no other new unknown probabilities appear except $\pi(n_{1},N_{2}-1)$, $\pi(n_{1},N_{2})$ for $n_{1}=0,1,\ldots,N_{1}$.
\item \textbf{Region $S_{d}$:} Consider the region $S_{d}$, and let for $|x|\leq1$, $|y|\leq1$,
\begin{displaymath}
\begin{array}{rl}
g(x,y)=&\sum_{n_{1}=N_{2}}^{\infty}\sum_{n_{2}=N_{2}}^{\infty}\pi(n_{1},n_{2})x^{n_{1}-N_{1}}y^{n_{2}-N_{2}}\vspace{2mm}\\
=&\sum_{n_{2}=N_{2}}^{\infty}g_{n_{2}}(x)y^{n_{2}-N_{2}}=\sum_{n_{1}=N_{1}}^{\infty}h_{n_{1}}(y)x^{n_{1}-N_{1}}.
\end{array}
\end{displaymath}
Using the second equation in (\ref{r1}) and noting that $f_{i}(N_{1},n_{2},x):=f_{i}(N_{1},N_{2},x)$ for $n_{2}\geq N_{2}$, we have
\begin{displaymath}
\begin{array}{c}
-f_{1}(N_{1},N_{2}-1,x)g_{N_{2}-1}(x)+f_{2}(N_{1},N_{2},x)g_{N_{2}}(x)\\-f_{3}(N_{1},N_{2},x)g_{N_{2}+1}(x)=b_{N_{2}}(x),\vspace{2mm}\\
-f_{1}(N_{1},N_{2},x)g_{n_{2}-1}(x)+f_{2}(N_{1},N_{2},x)g_{n_{2}}(x)\\-f_{3}(N_{1},N_{2},x)g_{n_{2}+1}(x)=b_{n_{2}}(x),\,n_{2}\geq N_{2}+1.
\end{array}
\end{displaymath}
Multiplying with $y^{n_{2}-N_{2}}$, and summing for all $n_{2}=N_{2},N_{2}+1,\ldots$, we obtain after some algebra
\begin{equation}
\begin{array}{rl}
R(x,y)g(x,y)=&yf_{1}(N_{1},N_{2}-1,x)g_{N_{2}-1}(x)-f_{3}(N_{1},N_{2},x)g_{N_{2}}(x)\\
&+x\tilde{f}_{1}(N_{1}-1,N_{2},y)h_{N_{1}-1}(y)-\tilde{f}_{3}(N_{1},N_{2},y)h_{N_{1}}(y)\\
&+xyp_{1,1}(N_{1}-1,N_{2}-1)\pi(N_{1}-1,N_{2}-1)\\&-xp_{1,-1}(N_{1}-1,N_{2})\pi(N_{1}-1,N_{2})\\&-yp_{-1,1}(N_{1},N_{2}-1)\pi(N_{1},N_{2}-1)+p_{-1,-1}\pi(N_{1},N_{2}),
\end{array}\label{funnn}
\end{equation}
where $R(x,y)$ is given in \eqref{kern}. Substituting \eqref{r1} and \eqref{r2} in \eqref{funnn}, we finally obtain
\begin{equation}
R(x,y)g(x,y)=A(x,y)g_{0}(x)+B(x,y)h_{0}(y)+C(x,y),
\label{fun}
\end{equation}
where
\begin{equation}
R(x,y)=xy-\Psi(x,y),\label{kern}
\end{equation}
and
\begin{displaymath}
\begin{array}{rl}
\Psi(x,y)=&xyp_{0,0}+x^{2}yp_{1,0}+yp_{-1,0}+x^{2}y^{2}p_{1,1}+xy^{2}p_{0,1}\\&+y^{2}p_{-1,1}+xp_{0,-1}+x^{2}p_{1,-1}+p_{-1,-1},\\
A(x,y)=&yf_{1}(N_{1},N_{2}-1,x)e_{N_{2}-1}(x)-f_{3}(N_{1},N_{2},x)e_{N_{2}}(x),\vspace{2mm}\\
B(x,y)=&x\tilde{f}_{1}(N_{1}-1,N_{2},y)\tilde{e}_{N_{1}-1}(y)-\tilde{f}_{3}(N_{1},N_{2},y)\tilde{e}_{N_{1}}(y),\vspace{2mm}\\
C(x,y)=&K(\pi(N_{1}-1,N_{2}-1),\pi(N_{1}-1,N_{2}),\pi(N_{1},N_{2}-1),x,y)\\&+yf_{1}(N_{1},N_{2}-1,x)t_{N_{2}-1}(x)-f_{3}(N_{1},N_{2},x)t_{N_{2}}(x)\\&+x\tilde{f}_{1}(N_{1}-1,N_{2},y)\tilde{t}_{N_{1}-1}(y)-\tilde{f}_{3}(N_{1},N_{2},y)\tilde{t}_{N_{1}}(y),
\end{array}
\end{displaymath}
\begin{displaymath}
\begin{array}{l}
K(\pi(N_{1}-1,N_{2}-1),\pi(N_{1}-1,N_{2}),\pi(N_{1},N_{2}-1),\pi(N_{1},N_{2}),x,y)\\
=xy\pi(N_{1}-1,N_{2}-1)p_{1,1}(N_{1}-1,N_{2}-1)+p_{-1,-1}\pi(N_{1},N_{2})\\-y\pi(N_{1},N_{2}-1)p_{1,-1}(N_{1}-1,N_{2})
-x\pi(N_{1}-1,N_{2})p_{-1,1}(N_{1},N_{2}-1).
\end{array}
\end{displaymath}
\end{enumerate}
Note that $g(x,y)$ is regular in $y$ for $|y|<1$, and continuous in $y$ for $|y|\leq 1$, for any $x$ such that $|x|\leq 1$, and similarly with $x$, $y$ interchanged.

Our primary aim is to solve \eqref{fun}, and obtain $g(x,y)$. However, to accomplish this task, we first need to obtain the boundary functions $g_{0}(x)$, $h_{0}(y)$. Moreover, these unknown functions will be given in terms of the $N_{1}+N_{2}+1$ unknown probabilities: $\pi(N_{1},n_{2})$, $n_{2} =0, 1,...,N_{2} -1$, $\pi(n_{1},N_{2})$, $n_{1} =0, 1,...,N_{1} -1$ and $\pi(N_{1},N_{2})$. 
The theory of boundary value problems \cite{coh1,fay1} is the methodological tool to obtain $g_{0}(x)$, $h_{0}(y)$ (and as a consequence the function $g(x,y)$). Such an approach is quite lengthy and technical. The basic steps are briefly summarized below:
\begin{enumerate}
\item[Step 1:] Provide a thorough investigation of the zero-pairs of the kernel $R(x,y)$, i.e., the set $\mathcal{K}=\{(x,y):R(x,y)=0,|x|\leq 1,|y|\leq 1\}$. For such zero-pairs, due to the regularity properties of $g(x,y)$, the right hand side of \eqref{fun} also vanishes, and provides a relation between $g_{0}(x)$ and $h_{0}(y)$. In particular, we restrict the analysis of \eqref{fun} to a subset $\mathcal{S}_{1}\times\mathcal{S}_{2}\subset \mathcal{K}$ (where $\mathcal{S}_{1}\subset\{x:|x|\leq 1\}$, $\mathcal{S}_{2}\subset\{y:|y|\leq 1\}$), such that there exists a one-to-one map from $\mathcal{S}_{1}$ onto $\mathcal{S}_{2}$. Then, the regularity of $R(x,y)$ implies that by analytic continuation, all zero-pairs of the kernel can be constructed starting from $\mathcal{S}_{2}$. Thus, the construction of the suitable subset $\mathcal{S}_{1}\times\mathcal{S}_{2}$ is crucial. Under ergodicity conditions, it is shown that $\mathcal{S}_{1}$, $\mathcal{S}_{2}$ are both simple (i.e., no double points), smooth (i.e., with continuously varying tangents), and analytic contours, except possibly at $x=1$, and $y=1$, respectively (see Figure \ref{conto} and Appendix \ref{defi}). Thus, up to now, our problem is reduced to the determination of $g_{0}(x)$ and $h_{0}(y)$, both regular in $|x|<1$, $|y|<1$, respectively, that satisfy a relation that is derived by \eqref{fun} for $(x,y)\in \mathcal{S}_{1}\times\mathcal{S}_{2}$.
\item[Step 2:] Then, the above problem is transformed into a Riemann boundary value problem in a smooth closed contour $\mathcal{L}$. To accomplish this task, we need to firstly construct the contour $\mathcal{L}$ from the contours $\mathcal{S}_{1}$, $\mathcal{S}_{2}$; see Section \cite[II.3.3-II.3.7]{coh1}. Details on steps 1, 2, are provided in the subsection \ref{ker}. 
\item[Step 3:] Using the results derived steps 1, 2, and the \textit{corresponding boundary theorem} \cite[p. 179]{neh}, the problem of finding $g_{0}(x)$ and $h_{0}(y)$, both regular in $|x|<1$, $|y|<1$, is reduced to the solution of a Riemann boundary value problem in a smooth, closed contour $\mathcal{L}$, with a boundary condition being the relation that is derived by \eqref{fun} for $(x,y)\in \mathcal{S}_{1}\times\mathcal{S}_{2}$, which is now transformed on the smooth contour $\mathcal{L}$ (constructed at step 2). Details on this step are given in subsection \ref{solution}. By solving this Riemann boundary value problem, we obtain $g_{0}(x)$, $h_{0}(y)$, and thus $g(x,y)$ in terms of the $N_{1}+N_{2}+1$ unknown probabilities mentioned above. Details on deriving these unknown probabilities are given at the end of the subsection \ref{solution}.
\end{enumerate}
To recapitulate the state of affairs so far:
\begin{enumerate}
\item For the probabilities of states in the region $S_{a}$ we have equations (\ref{e1}), for  $n_{1}= 0, 1,..., N_{1}-1$, $n_{2} = 0, 1,..., N_{2}-1$.
\item For those in the region $S_{b}$, we have equations (\ref{r2}), $n_{2}=0, 1,...,N_{2} -1$.
\item For those in the region $S_{c}$, we have equations (\ref{r3}), $n_{1} = 0, 1,..., N_{1}-1$.  
\item For the region $S_{d}$, all unknown quantities are expressed in terms of
\begin{enumerate}
\item  $g_{0}(x)$, $h_{0}(y)$, 
\item the $N_{1}+N_{2}+1$ probabilities, $\pi(N_{1},n_{2})$, $n_{2} =0, 1,...,N_{2} -1$, and $\pi(n_{1},N_{2})$, $n_{1} =0, 1,...,N_{1} -1$ and $\pi(N_{1},N_{2})$.  
\end{enumerate} 
\end{enumerate}
\subsection{Kernel analysis}\label{ker}
Assume here on that $\Psi(0,0)>0$; i.e., $p_{-1,-1}>0$. As already stated, the kernel analysis is the first and crucial step towards the solution of the functional equation \eqref{fun}. A central role in the analysis is played by the zero-pairs of the kernel, i.e., the set $\mathcal{K}=\{(x,y):R(x,y)=0,|x|\leq1,|y|\leq 1\}$, the existence of which, can be shown by the Rouch\'e's theorem \cite{marku}. For such zero-pairs, due to the regularity properties of $g(x,y)$, the right-hand side of \eqref{fun} vanishes too. This will provide us with a relation between the functions $g_{0}(x)$ and $h_{0}(y)$.  

As stated in step 1, we look for a suitable representation of zero-pairs, which in our case can be found by assuming a \textit{symmetrical} representation (see also \cite[Section II.2.2]{rw}), i.e., by considering the kernel for $s\in\mathbb{C}$ (where $\mathbb{C}$ is the complex plane), such that $|s|=1$ and 
\begin{equation*}
x=gs,\,\,\,y=gs^{-1},\,\,|g|\leq1,\label{dgh}
\end{equation*}
(where $gs=g\times s$, $gs^{-1}=g\times s^{-1}$). Then,
\begin{equation}
\begin{array}{rl}
R(gs,gs^{-1})=&g^{2}-\Psi(gs,gs^{-1})=g^{2}-\mathbb{E}(g^{\xi_{1}^{(3)}+\xi_{2}^{(3)}}s^{\xi_{1}^{(3)}-\xi_{2}^{(3)}})\\
=&g^{2}-\sum_{i=-1}^{1}\sum_{j=-1}^{1}p_{i,j}g^{i+j+2}s^{i-j}.
\end{array}\label{zzz}
\end{equation}
It is readily seen that $\Psi(gs,gs^{-1})$ is for $g=1$ regular at $s=1$, and for $s=1$, regular at $g=1$ (i.e., all moments exist and are finite), whereas, after rearranging the terms in \eqref{zzz} we obtain 
\begin{equation}
\begin{array}{l}
R(gs,gs^{-1})=0\Leftrightarrow g^{2}=\frac{p_{-1,-1}+p_{-1,0}gs^{-1}+p_{0,-1}gs}{1-\sum_{\substack{ i=-1\\i+j\geq0}}^{1}\sum_{j=-1}^{1}p_{i,j}g^{i+j}s^{i-j}}.
\end{array}\label{fty}
\end{equation}
Note that for $|g|\leq1$, $|s|=1$, the denominator in (\ref{fty}) never vanishes. Indeed, 
\begin{displaymath}
\begin{array}{rl}
|\sum_{\substack{ i=-1\\i+j\geq0}}^{1}\sum_{j=-1}^{1}p_{i,j}g^{i+j}s^{i-j}|\leq& p_{0,0}+p_{0,1}+p_{1,-1}+p_{1,0}+p_{1,1}+p_{-1,1}\\=&1-(p_{-1,-1}+p_{-1,0}+p_{0,-1})<1.
\end{array}
\end{displaymath}
\begin{theorem}\label{th1}\begin{enumerate}
\item If $E_{x}+E_{y}<0$, and for fixed $s\in\mathbb{C}$ such that $|s|=1$, the kernel $R(gs,gs^{-1})$, $|s|=1$ has in $|g|\leq1$ exactly two zeros each with multiplicity one, which are both real for $s=\pm 1$.
\item If $g(s)$ is a zero, so is $-g(-s)$.
\end{enumerate}
\end{theorem}
\begin{proof}
See Appendix \ref{nnrw}.\hfill$\square$
\end{proof}

Define,
\begin{displaymath}
\mathcal{S}_{1}:=\{x:x=g(s)s,|s|=1\},\,\,\mathcal{S}_{2}:=\{y:y=g(s)s^{-1},|s|=1\},
\end{displaymath}
where $g(s)$ is the zero for which $g(s)\to 1$ as $s\to 1$, $|s|=1$. Let $s=e^{i\phi}$, i.e., 
\begin{equation}
x(\phi)=g(\phi)e^{i\phi},\ y(\phi)=g(\phi)e^{-i\phi},\,\phi\in[0,2\pi].
\label{dgh}
\end{equation}
Then, the relation (\ref{dgh}) defines a one-to-one mapping $f$ of $\mathcal{S}_{2}$ onto $\mathcal{S}_{1}$, i.e., $x(\phi)=f(y(\phi))$ or $y(\phi)=f^{-1}(x(\phi))$, $\phi\in[0,2\pi]$. More precisely, $x\in\mathcal{S}_{1}\Leftrightarrow \bar{x}\in\mathcal{S}_{2}$, where $\bar{x}$ is the complex conjugate of $x$. It is readily verified that $g(s)=g(\bar{s})$ for $|s|=1$, so that $\mathcal{S}_{1}$ and $\mathcal{S}_{2}$ are congruent contours, and as $s$ traverses the unit circle counter clockwise, $x$ traverses $\mathcal{S}_{1}$ counter clockwise whereas $y$ traverses $\mathcal{S}_{2}$ clockwise.

\begin{remark} Note that instead of \eqref{dgh}, other possible candidates for zero-pairs of the kernel can also be found. For example one can take $x$, such that $|x|=1$, and for such a fixed $x$, to look for $y=Y(x)$ such that $|y|<1$. However, such an approach may sometimes leads to a self-intersecting contour $\mathcal{S}_{2}=\{y:y=Y(x),|x|=1\}$, and along with $\mathcal{S}_{1}=\{x:|x|=1\}$, they do not satisfy the assumptions required in step 1 (i.e., we must have simple contours). \end{remark}

In the following, we have to show that $\mathcal{S}_{1}$, $\mathcal{S}_{2}$ are simple and smooth, i.e., they are closed, non-self intersecting curves with a continuously varying tangent; see Figures \ref{conto}, \ref{conto1}.

\begin{figure}[ht!]
\centering
\includegraphics[scale=0.6]{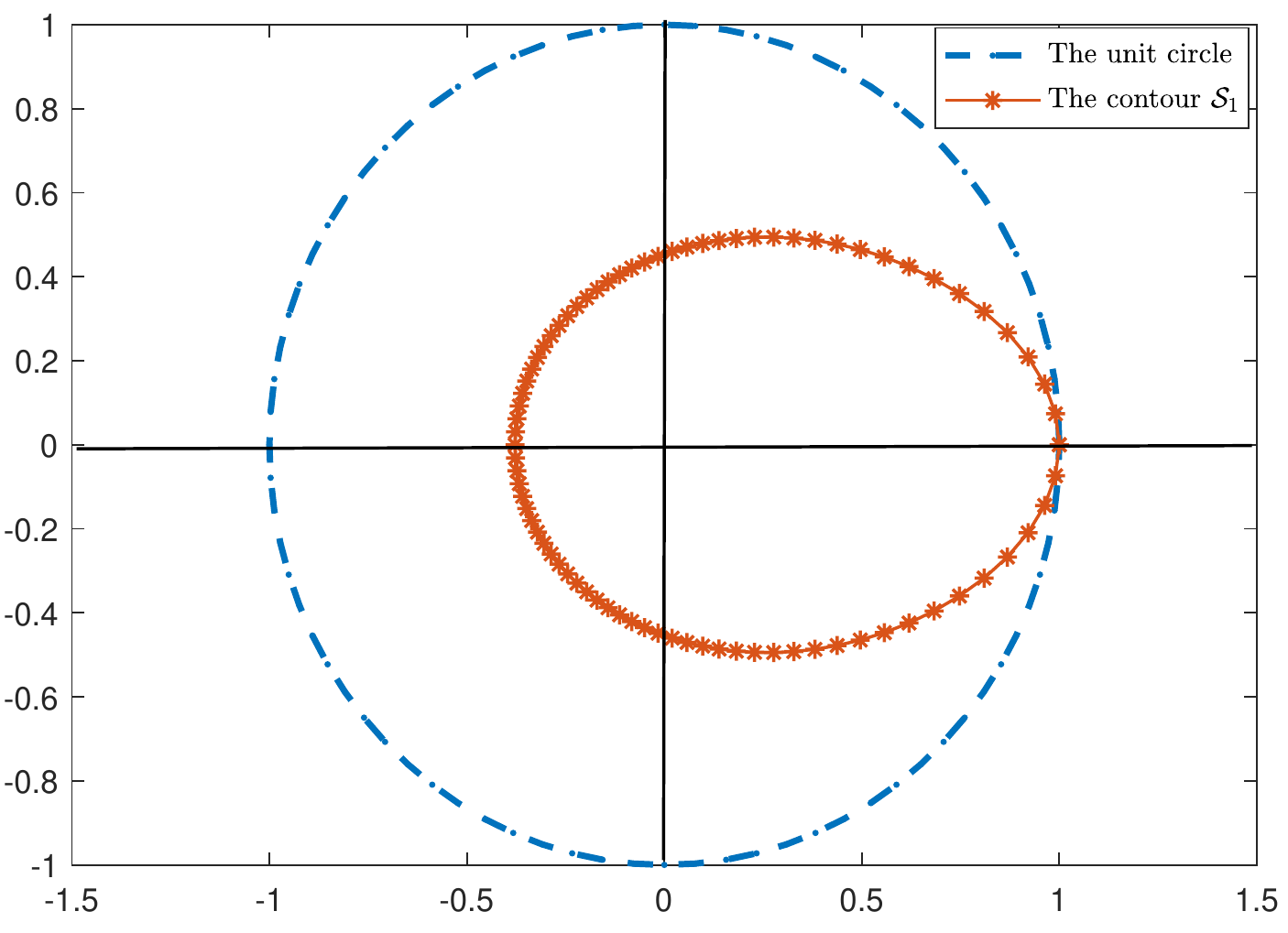}
\caption{The contour $\mathcal{S}_{1}$ for the symmetrical case ($\mathcal{S}_{1}$ coincides with $\mathcal{S}_{2}$) where $p_{0,1}=p_{1,0}$, $p_{-1,0}=p_{0,-1}$, $p_{1,-1}=p_{-1,1}$, and the unit circle.}\label{conto}
\end{figure}
\begin{figure}[ht!]
\centering
\includegraphics[scale=0.7]{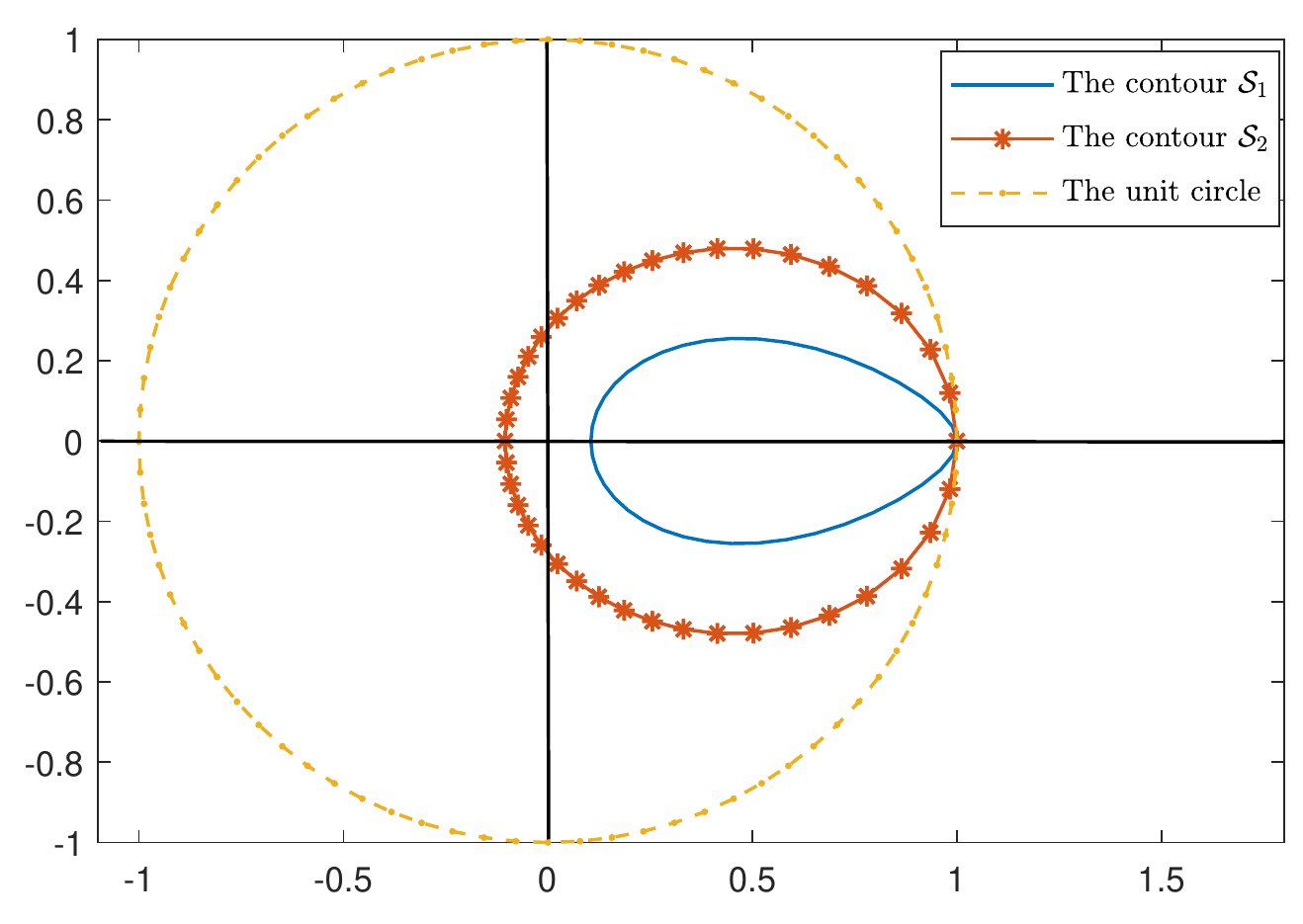}
\caption{The contours $\mathcal{S}_{1}$, $\mathcal{S}_{2}$ for the asymmetrical case for the case $p_{-1,-1}=0$, and the unit circle.}\label{conto11}
\end{figure}
\begin{figure}[ht!]
\centering
\includegraphics[scale=0.73]{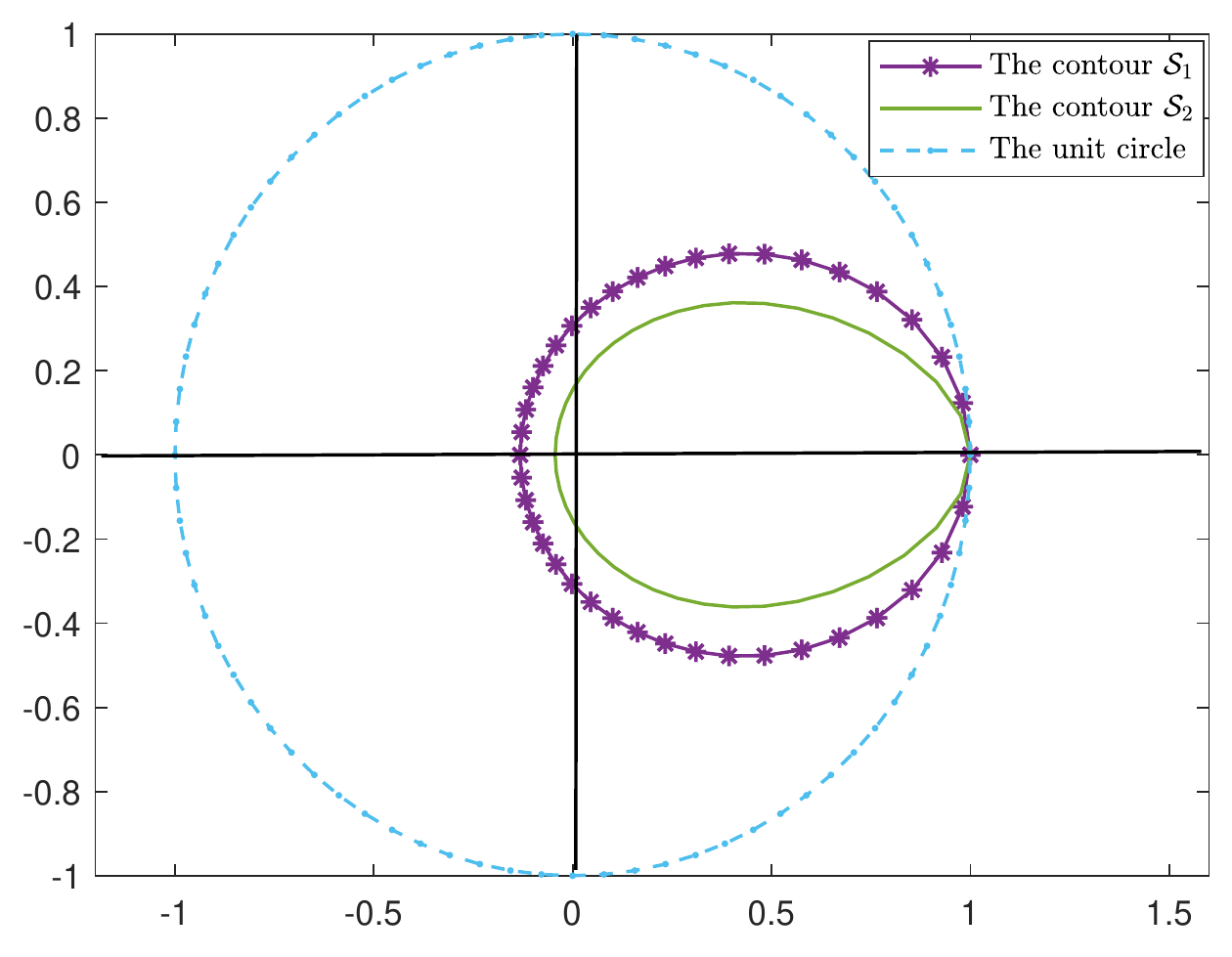}
\caption{The contours $\mathcal{S}_{1}$, $\mathcal{S}_{2}$ for the asymmetrical case for the case $p_{-1,-1}>0$ and the unit circle.}\label{conto1}
\end{figure}
Simple calculations show that for $|g|\leq 1$ satisfying (\ref{fty}) we have
{\small{\begin{equation}
\begin{array}{l}
\frac{s}{g(s)}\frac{d}{ds}g(s)\\
=\frac{s(p_{0,-1}s^{2}-p_{-1,0})+g(s)[2(s^{4}p_{1,-1}-p_{-1,1})-p_{0,1}g(s)s]}{g(s)\left(2[(1-p_{0,0})s^{2}-p_{-1,1}-s^{4}p_{1,-1}-2g^{2}(s)s^{2}]-3g^{2}(s)s[s^{2}p_{1,0}+p_{0,1}]\right)-s(s^{2}p_{0,-1}+p_{-1,0})}.
\end{array}\label{dcv}
\end{equation}}}
Thus, if the denominator in \eqref{dcv} never vanishes for $|g|\leq 1$, $|s|=1$, then, both zeros of (\ref{fty}) have multiplicity one, and each of these zeros is an analytic function of $s$ on the unit circle $|s|=1$.
\begin{theorem}\label{yui}
If $E_{x}+E_{y}<0$, and $\mathbb{E}((2-\xi_{1}^{(3)}-\xi_{2}^{(3)})g^{\xi_{1}^{(3)}+\xi_{2}^{(3)}}s^{\xi_{1}^{(3)}-\xi_{2}^{(3)}})\neq0$ (i.e., the denominator in \eqref{dcv}), then $\mathcal{S}_{1}$, $\mathcal{S}_{2}$ are both smooth, and analytic contours, except possibly at $s=1$. Moreover, $x=0\in \mathcal{S}_{1}^{+}$, $y=0\in\mathcal{S}_{2}^{+}$, where $\mathcal{S}_{j}^{+}$ denotes the interior domain bounded by $\mathcal{S}_{j}$, $j=1,2$; see also Figures \ref{conto}, \ref{conto1}.
\end{theorem}
\begin{proof}
The results presented so far are necessary for the proof of the theorem, which follows the lines in \cite[Lemma II.3.2.2.2, pp. 157-158]{coh1}, and further details are omitted.\hfill$\square$
\end{proof}

Theorem \ref{yui} implies (See also Theorem 1.1 in \cite{coh2}) that we need to solve the following boundary value problem: Find a unique simple contour $\mathcal{L}$ in the $z$-plane with $z=0\in\mathcal{L}^{+}$, $z=1\in\mathcal{L}$, $z=\infty\in\mathcal{L}^{-},$ and functions 
\begin{displaymath}
x(z):\mathcal{L}^{+}\cup\mathcal{L}\to\mathcal{S}_{1}^{+}\cup\mathcal{S}_{1},\,\,y(z):\mathcal{L}^{-}\cup\mathcal{L}\to\mathcal{S}_{2}^{+}\cup\mathcal{S}_{2},
\end{displaymath}
such that
\begin{enumerate}
\item $z=0$ is a simple zero of $x(.)$, and $z=\infty$ is a simple zero of $y(.)$. Moreover, $0<d:=\lim_{|z|\to\infty}zy(z)<\infty$,
\item $x(z):\mathcal{L}^{+}\to\mathcal{S}_{1}^{+}$ is regular and univalent for $z\in \mathcal{L}^{+}$,
\item $y(z):\mathcal{L}^{-}\to\mathcal{S}_{2}^{+}$ is regular and univalent for $z\in \mathcal{L}^{-}$,
\item $x(z)=f(y(z))$, $z\in\mathcal{L}^{+}$, where $f(.)$ is the one-to-one mapping defined by the relations in \eqref{dgh}.
\item $x^{+}(z)$, $y^{-}(z)$, $z\in\mathcal{L}$ is a zero pair of the kernel $R(x,y)=0$, with $x^{+}(z)\in\mathcal{S}_{1}$, $y^{-}(z)\in\mathcal{S}_{2}$, where for $z\in\mathcal{L}$, $x^{+}(z)=\lim_{t\to z,t\in\mathcal{L}^{+}}x(t)$, $y^{-}(z)=\lim_{t\to z,t\in\mathcal{L}^{-}}y(t)$.
\end{enumerate}

To summarize the results obtained in this section, we have: $i)$ $\Psi(0,0)>0$, $ii)$ $E_{x}<0$, $E_{y}<0$, $iii)$ $\mathcal{S}_{1}$, $\mathcal{S}_{2}$ are simple, and analytic contours with $x=0\in \mathcal{S}_{1}^{+}$, $y=0\in \mathcal{S}_{2}^{+}$, and that $iv)$ $R(x,y)$ is for $x=1$ (resp. $y=1$) a regular function at $y=1$ (resp. $x=1$). With these results, it is assured (see \cite[Theorem II.3.3.3.1]{coh1}) that there exist functions $x(z)$, $y(z)$, and a Jordan contour (see Appendix \ref{defi}) $\mathcal{L}$ satisfying the assertions $1,...,5$. Moreover, $\mathcal{L}$ is an analytic contour. The following lemma describes the solution of the boundary value problem described by the assertions $1,...,5$ above.
\begin{lemma}\label{kopa}
The boundary value problem described by the assertions $1,...,5$ has a unique solution and for $z\in\mathcal{L}$, there exists a real function $\lambda(\cdot)$ such that $\lambda(1)=0$,
\begin{displaymath}
x^{+}(z)=g(e^{i\lambda(z)})e^{i\lambda(z)},\,\,\,\,y^{-}(z)=g(e^{i\lambda(z)})e^{-i\lambda(z)},
\end{displaymath}
and
\begin{align}
x(z)=&ze^{\frac{1}{2i\pi}\int_{\zeta\in\mathcal{L}}\log[g(e^{i\lambda(\zeta)})][\frac{\zeta+z}{\zeta-z}-\frac{\zeta+1}{\zeta-1}]\frac{d\zeta}{\zeta}},\,z\in\mathcal{L}^{+},\label{ml1}\\
y(z)=&z^{-1}e^{\frac{-1}{2i\pi}\int_{\zeta\in\mathcal{L}}\log[g(e^{i\lambda(\zeta)})][\frac{\zeta+z}{\zeta-z}-\frac{\zeta+1}{\zeta-1}]\frac{d\zeta}{\zeta}},\,z\in\mathcal{L}^{-},\label{ml2}\\
e^{i\lambda(z)}=&ze^{\frac{1}{2i\pi}\int_{\zeta\in\mathcal{L}}\log[g(e^{i\lambda(\zeta)})][\frac{\zeta+z}{\zeta-z}-\frac{\zeta+1}{\zeta-1}]\frac{d\zeta}{\zeta}},\,z\in\mathcal{L}.\label{ml3}
\end{align}
\end{lemma}
\begin{proof}
The proof follows the lines in \cite[Section II.3.6]{coh1}.\hfill$\square$
\end{proof}  
\begin{remark}
Note that \eqref{ml3} follows by applying the Plemelj-Sokhotski formulas in \eqref{ml1}, \eqref{ml2}. Since $x^{+}(z)$ maps the contour $\mathcal{L}$ one-to-one onto the contour $\mathcal{S}_{1}$, and since $x(z)$ given in \eqref{ml1}, is regular in $\mathcal{L}^{+}\cup\mathcal{L}$, it follows from the principle of corresponding boundaries (see Appendix \ref{defi}) that $x(z)$ maps $\mathcal{L}^{+}$ \textit{conformally} onto $\mathcal{S}_{1}^{+}$. Similarly for $y(z)$. 
\end{remark}
\begin{remark}
Equation (\ref{ml3}) is the key to obtain $\lambda(.)$ and $\mathcal{L}$, and is rewritten as:
\begin{equation}\begin{array}{c}
i\lambda(z)-\log[z]=\frac{1}{2i\pi}\int_{\zeta\in\mathcal{L}}\log[g(e^{i\lambda(\zeta)})][\frac{\zeta+z}{\zeta-z}-\frac{\zeta+1}{\zeta-1}]\frac{d\zeta}{\zeta},\,z\in\mathcal{L},
\end{array}
\label{soo}
\end{equation} 
with $\mathcal{L}=\{z:z=\rho(\phi)e^{i\phi},0\leq\phi\leq 2\pi\}$, $\lambda(z)$, $z\in\mathcal{L}$. Note that 
\eqref{soo} represents a Theodorsen's integral equation (replace $\zeta$ by $e^{i\omega}$, $z$ by $e^{i(\phi-\phi_{0})}$ and $\lambda(z)$ by $\theta(\phi)$ with $\omega,\phi,\phi_{0}$ real and $\theta(\phi_{0})=0$).

In order to obtain $\mathcal{L}=\{z\in\mathbb{C}:z=\rho(\phi)e^{i\phi},\phi\in[0,2\pi]\}$ and $\lambda(z)$, $z\in\mathcal{L}$, or equivalently, $\theta(\phi)=\lambda(\rho(\phi)e^{i\phi})$, $\phi\in[0,2\pi]$, we need to solve the integral equation \eqref{soo}, the solution of which provides the two real functions $\rho(\phi)$, and $\theta(\phi)$. This is done numerically by separating the real and imaginary part of its right hand-side of (\ref{soo}), which leads to two singular integral equations in the two unknowns functions $\rho(\phi):=\rho(\theta(\phi))=g(e^{i\theta(\phi)})$, $\phi\in[0,2\pi]$, $\lambda(z)$, $z\in\mathcal{L}$. For a numerical treatment of (\ref{soo}), which may be regarded as a generalization of the Theodorsen's integral equation, see \cite[Section IV.2.3]{coh1}.\end{remark}
\subsection{Solution of the functional equation \eqref{fun}}\label{solution}
Subsection \ref{ker} provides the necessary information for the solution of the functional equation \eqref{fun}. Having investigated the properties of the zero-pairs of the kernel equation we are able to obtain a relation between the unknown functions $g_{0}(x)$, $h_{0}(y)$, and establish a way to reduce the problem of their determination to the solution of a Riemann boundary value problem on the smooth contour $\mathcal{L}$, as shown below.

Since for $z\in\mathcal{L}$, $(x^{+}(z),y^{-}(z))$ is a zero-pair of the kernel, such that $|x^{+}(z)|<1$, $|y^{-}(z)|<1$, we have for $z\in\mathcal{L}$
\begin{equation}
A(x^{+}(z),y^{-}(z))g_{0}(x^{+}(z))+B(x^{+}(z),y^{-}(z))h_{0}(y^{-}(z))+C(x^{+}(z),y^{-}(z))=0,
\end{equation}
or equivalently,
\begin{equation}
\tilde{g}_{0}(z)=G(z)\tilde{h}_{0}(z)+\tilde{c}(z),\,\,z\in\mathcal{L},
\label{pp}\end{equation}
where $\tilde{g}_{0}(z):=g_{0}(x^{+}(z))$, $\tilde{h}_{0}(z):=h_{0}(y^{-}(z))$ and
\begin{displaymath}
\begin{array}{rl}
G(z):=-\frac{B(x^{+}(z),y^{-}(z))}{A(x^{+}(z),y^{-}(z))},\,\,\,\tilde{c}(z):=-\frac{C(x^{+}(z),y^{-}(z))}{A(x^{+}(z),y^{-}(z))}.
\end{array}
\end{displaymath}

To proceed, we have to ensure that $G(.)$ and $\tilde{c}(.)$ satisfy the H\"older condition, and $G(.)$ never vanishes. However, the general form of $G(.)$ cannot exclude the possibility both of vanishing, and on taking infinite values at some points of $\mathcal{L}$. More importantly, the poles of $G(.)$ (if any) that are located in the region bounded by $\mathcal{L}$ and the unit circle will be also poles of $\tilde{g}_{0}(.)$. Let $a_{k}$, $k=1,\ldots,m$ be the poles of $G(.)$ (i.e., the zeros of $A(x^{+}(z),y^{-}(z))$), with multiplicity $u_{k}$, and let also $b_{s}$, $s=1,\ldots,l$ the zeros of $G(.)$ with multiplicity $p_{s}$. Denote
\begin{displaymath}
\begin{array}{ll}
\widehat{g}_{0}(z):=\prod_{k=1}^{m}(z-a_{k})^{u_{k}}\tilde{g}_{0}(z),&\widehat{h}_{0}(z):=\prod_{s=1}^{l}(z-b_{s})^{p_{s}}\tilde{h}_{0}(z),\\\widehat{G}(z):=\frac{\prod_{k=1}^{m}(z-a_{k})^{u_{k}}}{\prod_{s=1}^{l}(z-b_{s})^{p_{s}}}G(z),&\widehat{c}(z):=\prod_{k=1}^{m}(z-a_{k})^{u_{k}}\tilde{c}(z).
\end{array}
\end{displaymath}
Then, (\ref{pp}) reads for $z\in\mathcal{L}$
\begin{equation}
\widehat{g}_{0}(z)=\widehat{G}(z)\widehat{h}_{0}(z)+\widehat{c}(z),
\label{pqp}
\end{equation}
and is the boundary condition of a non-homogeneous Riemann boundary value problem \cite{ga} on $\mathcal{L}$: Determine two functions $\widehat{g}_{0}(z)$, $\widehat{h}_{0}(z)$ such that
\begin{enumerate}
\item $\widehat{g}_{0}(z)$ is regular for $z\in\mathcal{L}^{+}$, and continuous in $\mathcal{L}^{+}\cup\mathcal{L}$,
\item $\widehat{h}_{0}(z)$ is regular for $z\in\mathcal{L}^{-}$, and continuous in $\mathcal{L}^{-}\cup\mathcal{L}$,
\item For $z\in\mathcal{L}$, the boundary condition \eqref{pqp} holds.
\end{enumerate} 

If the index (see Appendix \ref{defi}) $\chi:=index[\widehat{G}(z)]_{\mathcal{L}}\geq 0$, then
\begin{equation}
\begin{array}{rl}
g_{0}(z):=&\prod_{k=1}^{m}(z-a_{k})^{-u_{k}}e^{\Gamma(z)}[\Phi(z)+P_{\chi}(z)],\,z\in\mathcal{L}^{+},\vspace{2mm}\\
h_{0}(z):=&\prod_{s=1}^{l}(z-b_{s})^{-p_{s}}e^{\Gamma(z)}[\Phi(z)+P_{\chi}(z)],\,z\in\mathcal{L}^{-},
\end{array}
\label{solo}
\end{equation}
where $P_{\chi}(\cdot)$ an arbitrary polynomial of degree $\chi\geq 0$, and
\begin{displaymath}
\begin{array}{rl}
\Gamma(z):=&\frac{1}{2i\pi}\int_{t\in\mathcal{L}}\log[t^{-\chi}G(t)]\frac{dt}{t-z},\,z\notin\mathcal{L},\vspace{2mm}\\
\Phi(z):=&\frac{1}{2i\pi}\int_{t\in\mathcal{L}}\widehat{c}(t)e^{-\Gamma^{+}(t)}\frac{dt}{t-z},\,z\notin\mathcal{L}.
\end{array}
\end{displaymath}
If $\chi<0$, then
\begin{equation}
\begin{array}{rl}
g_{0}(z):=&\prod_{k=1}^{m}(z-a_{k})^{-u_{k}}e^{\Gamma(z)}\Phi(z),\,z\in\mathcal{L}^{+},\vspace{2mm}\\
h_{0}(z):=&\prod_{s=1}^{l}(z-b_{s})^{-p_{s}}e^{\Gamma(z)}\Phi(z),\,z\in\mathcal{L}^{-},
\end{array}
\label{solo1}
\end{equation}
but now the following $-\chi-1$ conditions must be satisfied
\begin{equation}
\int_{t\in\mathcal{L}}t^{r-1}\widehat{c}(t)e^{-\Gamma^{+}(t)}dt=0,\,r=1,2,...,-\chi-1.
\label{index}
\end{equation}
Having obtained $g_{0}(z)$, $h_{0}(z)$, we are able to obtain $g(x,y)$ in \eqref{fun}.

The following steps summarize the way we can fully determine the stationary distribution: 
\begin{enumerate}
\item The $N_{1}\times N_{2}$ equations for $S_{a}$ involve $(N_{1}+1)\times(N_{2}+1)$ unknowns: $\pi(n_{1},n_{2})$, $n_{1}=0,1,\ldots,N_{1}-1$, $n_{1}=0,1,\ldots,N_{1}-1$, that refer to the states of $S_{a}$, and $N_{1}+N_{2}+1$ refer to the states that belong to the boundary of $S_{b}$ with $S_{a}$ (i.e., $\pi(N_{1},n_{2})$, $n_{2}=0,1,\ldots,N_{2}-1$), to the boundary of $S_{c}$ with $S_{a}$ (i.e., $\pi(n_{1},N_{2})$, $n_{1}=0,1,\ldots,N_{1}-1$), and to the boundary of $S_{d}$ with $S_{a}$ (i.e., $\pi(N_{1},N_{2})$). Thus, we further need $N_{1}+N_{2}+1$ equations that involve the unknowns that are associated with the states that belong $a)$ to the boundary of $S_{b}$ with $S_{a}$, $b)$ to the boundary of $S_{c}$ with $S_{a}$, and $c)$ to the boundary of $S_{d}$ with $S_{a}$ (see point 3 below).
\item We have expressed the generating functions of equilibrium probabilities that are associated with the states of subregions $S_{b}$, $S_{c}$, in terms of $g_{0}(x)$, $h_{0}(y)$. The generating function of the equilibrium probabilities that are associated with the states of subregion $S_{d}$ is obtained in terms of the solution of a Riemann boundary value problem. That solution, obtains $g_{0}(x)$, $h_{0}(y)$ in terms of $A(x,y)$, $B(x,y)$ and $C(x,y)$. The first two are known, and the third one contains the $N_{1}+N_{2}+1$ unknown probabilities that we mentioned in point 1. Thus, we need some additional equations. These additional equations are derived from the (integral)  expressions of $g_{0}(x)$, $h_{0}(y)$ as follows at steps 3 and 4.
\item Use (\ref{r2}), (\ref{r3}), to express the unknown probabilities in terms of the values of $g_{0}(x)$, $h_{0}(y)$ at point 0, i.e., $\pi(N_{1},0)=g_{0}(0)$, $\pi(0,N_{2})=h_{0}(0)$ and
\begin{equation}
\begin{array}{rl}
\pi(N_{1},n_{2})=&\lim_{x\to 0}[e_{n_{2}}(x)g_{0}(x)+t_{n_{2}}(x)],\,n_{2}=1,...,N_{2}-1,\vspace{2mm}\\
\pi(n_{1},N_{2})=&\lim_{y\to 0}[\tilde{e}_{n_{1}}(y)h_{0}(y)+\tilde{t}_{n_{1}}(y)],\,n_{1}=1,...,N_{1}-1.
\end{array}
\end{equation}
This procedure will provide $N_{1}+N_{2}$ equations. Note that the functions $e_{n_{2}}(x)$, $\tilde{e}_{n_{1}}(y)$, $t_{n_{2}}(x)$, $\tilde{t}_{n_{1}}(y)$ are rational functions, and taking the limit at point 0, we may have to apply several times the l'Hospital rule. This means that we need further to obtain the derivatives of $g_{0}(x)$, $h_{0}(y)$, i.e., the derivatives of integral expressions, which have to be evaluated numerically.
\item The normalization equation yields the last one:
\begin{equation*}
\begin{array}{c}
1=\sum_{n_{1}=0}^{N_{1}-1}\sum_{n_{2}=0}^{N_{2}-1}\pi(n_{1},n_{2})+\sum_{n_{1}=0}^{N_{1}-1}h_{n_{1}}(1)+\sum_{n_{2}=0}^{N_{2}-1}g_{n_{2}}(1)+g(1,1).
\end{array}
\label{norq}
\end{equation*}
\end{enumerate}
\subsection{Performance metrics}\label{perm}
Having obtained the equilibrium probabilities associated with subregion $S_{a}$, and the pgfs that refer to the equilibrium probabilities in subregions $S_{i}$, $i=b,c,d$, we are able to obtain useful performance metrics. More precisely,
\begin{equation}
\begin{array}{rl}
E(Q_{1})=&\sum_{n_{1}=1}^{N_{1}-1}\sum_{n_{2}=0}^{N_{2}-1}n_{1}\pi(n_{1},n_{2})+\sum_{n_{2}=0}^{N_{2}}\frac{d}{dx}g_{n_{2}}(x)|_{x=1}\\&+\sum_{n_{1}=1}^{N_{1}}n_{1}h_{n_{1}}(1)+\frac{\partial}{\partial x}g(x,1)|_{x=1},\vspace{2mm}\\
E(Q_{2})=&\sum_{n_{2}=1}^{N_{2}-1}\sum_{n_{1}=0}^{N_{1}-1}n_{2}\pi(n_{1},n_{2})+\sum_{n_{1}=0}^{N_{1}}\frac{d}{dy}h_{n_{1}}(y)|_{y=1}\\&+\sum_{n_{2}=1}^{N_{2}}n_{2}g_{n_{2}}(1)+\frac{\partial}{\partial y}g(1,y)|_{y=1}.
\end{array}\label{dddd}
\end{equation}
Here on, denote by $s_{k}^{(j)}(1,1)$ the $j-$th order derivative of a function $s(x,y)$ with respect to $x$ (resp. $y$) when $k=1$ (resp. $k=2$) at point $(x,y)=(1,1)$, $j\geq 1$. Similarly, denote by $v^{(j)}(1)$ the $j-$th order derivative of a function $v(x)$ with respect to $x$ at point $x=1$. Note that
\begin{displaymath}
\begin{array}{rl}
\frac{d}{dx}g_{n_{2}}(x)|_{x=1}:=g_{n_{2}}^{(1)}(1)=&r_{n_{2}}(1)+e_{n_{2}}(1)g_{0}^{(1)}(1),\,n_{2}=1,\ldots,N_{2},\vspace{2mm}\\
\frac{d}{dy}h_{n_{1}}(y)|_{y=1}:=h_{n_{2}}^{(1)}(1)=&\tilde{r}_{n_{1}}(1)+\tilde{e}_{n_{1}}(1)h_{0}^{(1)}(1),\,n_{1}=1,\ldots,N_{1},
\end{array}
\end{displaymath}
where \begin{displaymath}
\begin{array}{rl}
r_{n_{2}}(x)=&g_{0}(1)e_{n_{2}}^{(1)}(x)+t_{n_{2}}^{(1)}(x),\\
\tilde{r}_{n_{1}}(y)=&h_{0}(1)\tilde{e}^{(1)}_{n_{1}}(y)+\tilde{t}_{n_{1}}^{(1)}(y).\end{array}
\end{displaymath}
Moreover, setting $y=1$ in \eqref{fun} yields,
\begin{displaymath}
\begin{array}{rl}
g(x,1)=&\frac{A(x,1)}{x-\Psi(x,1)}g_{0}(x)+\frac{B(x,1)}{x-\Psi(x,1)}h_{0}(1)+\frac{C(x,1)}{x-\Psi(x,1)}.
\end{array}
\end{displaymath}
Differentiating with respect to $x$, letting $x\to 1$, applying L'Hospital's rule and having in mind that $E_{x}<0$ yields after some algebra
\begin{displaymath}
\begin{array}{rl}
g_{1}^{(1)}(1,1)=&\frac{A_{1}^{(2)}(1,1)g_{0}(1)+B_{1}^{(2)}(1,1)h_{0}(1)+C_{1}^{(2)}(1,1)}{-2E_{x}}+\frac{A_{1}^{(1)}(1,1)}{-E_{x}}g_{0}^{(1)}(1).
\end{array}
\end{displaymath}
Substituting back in \eqref{dddd} yields
\begin{equation}
\begin{array}{rl}
E(Q_{1})=\widehat{S}+g_{0}^{(1)}(1)[\frac{A_{1}^{(1)}(1,1)}{-E_{x}}+\sum_{n_{2}=0}^{N_{2}}e_{n_{2}}(1)],
\end{array}\label{ex1}
\end{equation}
where,
\begin{displaymath}
\begin{array}{rl}
\widehat{S}=&\sum_{n_{1}=1}^{N_{1}-1}\sum_{n_{2}=0}^{N_{2}-1}n_{1}\pi(n_{1},n_{2})+g_{0}(1)[\sum_{n_{2}=0}^{N_{2}}e_{n_{2}}^{(1)}(1)+\frac{A_{2}^{(1)}(1,1)}{-2E_{x}}]\\&+h_{0}(1)[\sum_{n_{1}=0}^{N_{1}}n_{1}\tilde{e}_{n_{1}}(1)+\frac{B_{1}^{(2)}(1,1)}{-2E_{x}}]+\sum_{n_{2}=0}^{N_{2}}t_{n_{2}}^{(1)}(1)\\&+\sum_{n_{1}=0}^{N_{1}}n_{1}\tilde{t}_{n_{2}}(1)+\frac{C_{1}^{(2)}(1,1)}{-2E_{x}}.
\end{array}
\end{displaymath}
A similar expression can be derived for $E(Q_{2})$. In deriving \eqref{ex1}, it is crucial to derive $g_{0}^{(1)}(1)$, $g_{0}(1)$, $h_{0}(1)$, which are given in terms of a solution of a non-homogeneous Riemann boundary value problem given in \eqref{solo1}. These expressions are given as complex integrals, and have to be numerically evaluated. 
\begin{remark}
The analysis for $\Psi(0,0)=0$, i.e., $p_{-1,-1}=0$, is slightly different; see \cite[Sections II.3.10-II.3.12]{coh1}. The essential differences between these two cases refer mainly to the kernel analysis (Theorems \ref{th1}, \ref{yui}), as well as on the derivation of integral expressions \eqref{ml1}-\eqref{ml3}. See subsection \ref{subb} for more details.

For the case $p_{-1,-1}=0$, we can also apply a different approach based on \cite{fay1}. This approach is simpler compared with the one considered in Section \ref{pre}, and can be also applied in case $p_{-1,-1}>0$. For more details, see Appendix \ref{fayo}. 
\end{remark}
\subsection{The case where $\Psi(0,0)=0$}\label{subb}
In the following, we focus only on the solution of (\ref{fun}) when $p_{-1,-1}=0$. The assumption $p_{-1,-1}=0$ leads to some essential differences in the analysis we follow. The following lemma provides information about the zeros of the kernel.
\begin{lemma}\label{lemaw}\begin{enumerate}
\item If $E_{x}+E_{y}<0$, and for fixed $s\in\mathbb{C}$ such that $|s|=1$, the kernel $R(gs,gs^{-1})$ has in $|g|\leq1$ exactly two zeros, of which one is identically zero. The other one satisfies
\begin{equation}
\begin{array}{c}
g(s)=\frac{p_{-1,0}s^{-1}+p_{0,-1}s}{1-\sum_{\substack{ i=-1\\i+j\geq0}}^{1}\sum_{j=-1}^{1}p_{i,j}g(s)^{i+j}s^{i-j}}.
\end{array}
\label{cas}
\end{equation}
\item For $|s|=1$, if $g(s)$ is a zero, so is $-g(-s)$.
\end{enumerate}
\end{lemma}
\begin{proof}
See Appendix \ref{ape}
\end{proof}
Next, we focus on some technical issues that are related with the curves $\mathcal{S}_{1}$, $\mathcal{S}_{2}$ and are due to the fact that $\Psi(0,0)=0$ (see the discussion below Lemma \ref{lemas} and Remark \ref{rema}).
\begin{lemma}\label{lemas}Let the curves 
\begin{displaymath}
\mathcal{S}_{1}:=\{x:x=g(s)s,|s|=1\},\,\,\mathcal{S}_{2}:=\{y:y=g(s)s^{-1},|s|=1\},
\end{displaymath}
Then, if $s$ traverses the unit circle $|s|=1$ once, each of the curves is traversed twice in the opposite direction relative to each other. 
\end{lemma}
\begin{proof}

This result is immediately deduced by using \eqref{cas}: Taking the logarithm for $|s|=1$ we have
\begin{displaymath}
\begin{array}{c}
\log\{g(s)\}=\log\{p_{-1,0}+p_{0,-1}s^{2}\}-\log\{s\}-\log\{1-\sum_{\substack{ i=-1\\i+j\geq0}}^{1}\sum_{j=-1}^{1}p_{i,j}g^{i+j}s^{i-j}\}.
\end{array}
\end{displaymath}
Since the argument of the last logarithm never vanishes (see above Theorem \ref{th1}), and $p_{-1,0}+p_{0,-1}s^{2}$, and $s$, are both analytic on $|s|\leq 1$, we have
\begin{displaymath}
\begin{array}{rl}
ind_{|s|=1}g(s)s=0,&ind_{|s|=1}g(s)s^{-1}=-2,\,\text{for }p_{0,-1}>p_{-1,0},\\
ind_{|s|=1}g(s)s=1,&ind_{|s|=1}g(s)s^{-1}=-1,\,\text{for }p_{0,-1}=p_{-1,0},\\
ind_{|s|=1}g(s)s=2,&ind_{|s|=1}g(s)s^{-1}=0,\,\text{for }p_{0,-1}<p_{-1,0},
\end{array}
\end{displaymath}
thus, for $t=se^{i\pi}$, $|s|=1$, $g(s)s=g(t)t$, and $g(s)s^{-1}=g(t)t^{-1}$. \hfill$\square$
\end{proof}

Now we have to show that $\mathcal{S}_{1}$, $\mathcal{S}_{2}$ are both simple and smooth. Note that in case $\Psi(0,0)>0$, Theorem \ref{yui} guaranties the conditions under which $\mathcal{S}_{1}$, $\mathcal{S}_{2}$ are both simple and smooth. In case $\Psi(0,0)=0$, it is not possible to prove for general values of the system parameters that the contours $\mathcal{S}_{1}$ and $\mathcal{S}_{2}$ are simply connected. This is due to the variety of the positions of $x=0$, $y=0$ with respect to $\mathcal{S}_{1}$ and $\mathcal{S}_{2}$, respectively. Moreover, this problem leads to a different formulation of functions $x(z)$, $y(z)$, $z\in \mathcal{L}$ compared to the case where $\Psi(0,0)>0$. Then, we have the following cases:
\begin{enumerate}
\item If $p_{0,-1}>p_{-1,0}\Rightarrow$ $x=0\in\mathcal{S}_{1}^{+}$, and $y=0\in\mathcal{S}_{2}^{-}$.
\item If $p_{0,-1}=p_{-1,0}\Rightarrow$ $x=0\in\mathcal{S}_{1}$, and $y=0\in\mathcal{S}_{2}$.
\item If $p_{0,-1}<p_{-1,0}\Rightarrow$ $x=0\in\mathcal{S}_{1}^{-}$, and $y=0\in\mathcal{S}_{2}^{+}$.
\end{enumerate}

For each of the three cases mentioned above, we shall consider the problem formulated by the statements $1.,...,5.$ in subsection \ref{ker} and determine the functions $x(z)$, $y(z)$, $z\in \mathcal{L}$. So we restrict to the case where $\mathcal{S}_{1}$, $\mathcal{S}_{2}$ are both Jordan contours, and assume without loss of generality that $p_{0,-1}<p_{-1,0}$. The following lemma summarizes the solution of the corresponding boundary value problem described by the statements $1,...,5$ in subsection \ref{ker} (but now for $\Psi(0,0)=0$, $p_{0,-1}<p_{-1,0}$), and provides information about $x(\cdot)$, $y(\cdot)$. 
 
\begin{lemma}\label{olpa}
For $\Psi(0,0)=0$, the solution of the problem formulated by the statements $1,...,5$ (see subsection \ref{ker}) is unique and:
\begin{displaymath}
x^{+}(z)=g(e^{\frac{i\lambda(z)}{2}})e^{\frac{i\lambda(z)}{2}},\,\,\,\,y^{-}(z)=g(e^{\frac{i\lambda(z)}{2}})e^{-\frac{i\lambda(z)}{2}},\,z\in\mathcal{L},
\end{displaymath}
with $|x(z)|<1$, $z\in\mathcal{L}^{+}\cup\mathcal{L}$, $|y(z)|<1$, $z\in\mathcal{L}^{-}\cup\mathcal{L}$. Then,
\begin{displaymath}
\begin{array}{rl}
x(z)=&e^{\frac{1}{2\pi i}\int_{\zeta\in\mathcal{L}}[log[\sqrt{\zeta}g(e^{\frac{1}{2}i\lambda(\zeta)})](\frac{\zeta+z}{\zeta-z}-\frac{\zeta+1}{\zeta-1})\frac{d\zeta}{\zeta}]},\,z\in\mathcal{L}^{+},\\
y(z)=&\frac{1}{z}e^{-\frac{1}{2\pi i}\int_{\zeta\in\mathcal{L}}[log[\sqrt{\zeta}g(e^{\frac{1}{2}i\lambda(\zeta)})](\frac{\zeta+z}{\zeta-z}-\frac{\zeta+1}{\zeta-1})\frac{d\zeta}{\zeta}]},\,z\in\mathcal{L}^{-}.
\end{array}
\end{displaymath}
The relation for the determination of $\mathcal{L}$ and $\lambda(z)$, $z\in\mathcal{L}$ (with $\lambda(1)=0$) is
\begin{equation}
e^{i\lambda(z)}=ze^{\frac{2}{2i\pi}\int_{\zeta\in\mathcal{L}}[log[\sqrt{\zeta}g(e^{\frac{1}{2}i\lambda(\zeta)})](\frac{\zeta+z}{\zeta-z}-\frac{\zeta+1}{\zeta-1})\frac{d\zeta}{\zeta}]},\,z\in \mathcal{L}.
\label{int}
\end{equation}
\end{lemma}
\begin{proof}
For the proof see \cite[Sections II.3.10-II.3.11]{coh1}.
\end{proof}
The rest of the analysis is similar to the one given in subsection \ref{solution}.
\begin{remark}\label{rema}
 The expressions given in Lemma \ref{olpa} are slightly different with those given in case $\Psi(0,0)>0$ (see Lemma \ref{kopa}), due to the fact that for the present case, the contours $\mathcal{S}_{1}$, $\mathcal{S}_{2}$ are traversed twice if $s$ traverse the unit circle once. Thus, if $z$ traverse $\mathcal{L}$ once, then $\lambda(z)$ changes with $2\pi$.
\end{remark}
\begin{remark}
Note that by focusing on the symmetric case, i.e., $\Psi(x,y)=\Psi(y,x)$ ($p_{1,0}=p_{0,1}$, $p_{-1,0}=p_{0,-1}$, $p_{1,-1}=p_{-1,1}$), it is easily verified that $\frac{d}{dz}\Psi(z,0)|_{z=0}=\frac{d}{dz}\Psi(0,z)|_{z=0}$, thus $x=0\in\mathcal{S}_{1}$, $y=0\in\mathcal{S}_{2}$. In this case $\mathcal{L}$ is a circle with radius $1/2$ and center $1/2$; see \cite[Sect. II.3.12]{coh1}.
\end{remark}
\section{Applications}\label{appl}
In the following, we present two novel queueing models, which are described under our modelling framework. These models have applications in the flow-level performance of wireless networks with user mobility, and in queue-based RA networks. Moreover, in subsection \ref{appl2}, we observe that our framework can be extended to partially homogeneous \textbf{non-neighbour} RWQP.
\subsection{The two-queue LDGPS system with impatience}\label{appl2}
The LDGPS system with impatience serves as a basic model to investigate the impact of inter-cell mobility (i.e., the users may move out from the current cell during their communication) on data traffic performance
in dense networks with small cells \cite{sima,bonald}. In this setting, impatience models the user mobility. 

However, we have to note that impatience in PS systems has also been used to model random deadlines. In such a case, an arriving job must complete its service in a given time period, otherwise, it abandons the queue and therefore does not complete service. A typical example is the timeout of a TCP flow through the Internet. There, impatience refers to the expiration of a random deadline and subsequent reneging of the flow \cite{boyer}; see also \cite{gromoll,guilzwart}. 

In \cite{bonald}, the authors studied several models for characterizing the capacity and evaluating the flow-level performance of wireless networks carrying
elastic data transfers; see also \cite{bor2,bor1,guu} for similar studies. This is the \textit{Proportional Fair} (PF) sharing policy, which has been also widely used in commercial systems \cite{vis}. PF policy is an opportunistic scheduling scheme, which is often implemented in base stations and where the service capacity is no longer fixed, but increases as the number of users increases. Such a scenario gives rise to queueing models with scalable service capacity. More precisely, the flow-level performance of the PF policy can be evaluated in terms of PS models where the service rate varies with the total number of users. The
state-dependent service rate accounts for the fact that the throughput gains achieved
by channel-aware scheduling increase with the degree of multi-user diversity \cite{bonald2}. 

Thus, consider a single cell model with two classes of users served by a single base station (BS). We assume that the time is slotted. Class-$k$ users submit file transfer requests according to a Bernoulli process with parameter $\lambda_{k}(n_{1},n_{2})$, provided that there are $n_{k}$, $k=1,2$ backlogged packets (Note that BS provider may share information among users in a real time fashion regarding each user traffic in BS).
Class $k$-users have i.i.d. transmission requests that are geometrically distributed with parameter $\mu_{k}(n_{k})$ ($\bar{\mu}_{k}(n_{k})=1-\mu_{k}(n_{k})$).
 We further assume that the BS can simultaneously support at most $N_{1}$, $N_{2}$ user requests. There are plenty of factors that affect the maximum number of connected users in a cell. It will depend on the vendor implementation, memory and processing allocated, numbers of active bearers, etc. In an LTE enodeB (BS) it is estimated that in practice up to 100 devices (users) can be in RRC-Connected state \cite{sauta}, i.e., they are served by the BS. 

Moreover, the BS is employing the following service scheme: Up to $N_{1}$ jobs of type 1 and up to $N_{2}$  of type 2 file transfer requests are allowed to share the BS, and the BS serves a type 1 (resp. 2) request with probability $\beta_{1}(n_{1},n_{2})$ (resp. $\beta_{2}(n_{1},n_{2})$) with $\beta_{1}(n_{1},n_{2})+\beta_{2}(n_{1},n_{2})=1$, and where for $k=1,2$, 
\begin{equation*}
\begin{array}{rl}
\beta_{k}(n_{1},n_{2})=&\left\{\begin{array}{ll}
\beta_{k}(N_{1},n_{2}),&(n_{1},n_{2})\in S_{b},\\
\beta_{k}(n_{1},N_{2}),&(n_{1},n_{2})\in S_{c},\\
\beta_{k}(N_{1},N_{2})=\beta_{k},&(n_{1},n_{2})\in S_{d},
\end{array}\right.
\end{array}\label{bbb}
\end{equation*}
and $\beta_{1}(n_{1},0)=1$, $n_{1}>0$, $\beta_{2}(0,n_{2})=1$, $n_{1}>0$. The simplest scheduling function for such a system may has the form: 
\begin{equation*}
\begin{array}{c}
\beta_{k}(n_{1},n_{2})=\frac{n_{k}}{n_{1}+n_{2}},\,k=1,2,\,(n_{1},n_{2})\in S_{a},\vspace{2mm}\\
\beta_{1}(n_{1},n_{2})=\frac{n_{1}}{n_{1}+N_{2}},\,\beta_{2}(n_{1},n_{2})=\frac{N_{2}}{n_{1}+N_{2}},\,(n_{1},n_{2})\in S_{b},\vspace{2mm}
\\
\beta_{1}(n_{1},n_{2})=\frac{N_{1}}{n_{1}+N_{2}},\,\beta_{2}(n_{1},n_{2})=\frac{n_{2}}{N_{1}+n_{2}},\,(n_{1},n_{2})\in S_{c},\vspace{2mm}\\
\beta_{k}(n_{1},n_{2})=\frac{N_{k}}{N_{1}+N_{2}},\,k=1,2,\,(n_{1},n_{2})\in S_{d}.
\end{array}\label{scf}
\end{equation*}  

The \textit{user mobility} is modelled as follows: a user of type $k$ departs from the cell after a geometrically distributed time period with parameter $\theta_{k}(n_{k})$, $k=1,2$ (it remains in the cell with probability $\bar{\theta}_{k}(n_{k})=1-\theta_{k}(n_{k})$).

Due to the slotted-time framework, we have to appropriately schedule the way the events occur (i.e., departures due to service completions, departures due to user mobility, and arrivals). 
We assume that departures, either due to service completions, or due to user mobility, are scheduled at the beginning of a slot, while file transfer requests of either type at the end of the slot. In case a user's departure due to mobility occurs simultaneously with a service completion, we assume that the former one occurs first. However, since all the system parameters depend on the user state, a possible departure due to a user mobility will affect the service completion that follows (e.g., given the state of the system $\underline{n}$, in case there is a departure due to mobility of a user of type 1 (w.p. $\theta_{1}(n_{1})\bar{\theta}_{2}(n_{2})$), and the server chooses a type 2 user to serve, then, this will happen with probability $\beta_{2}(\underline{n}-\underline{1}_{1})$). This means that the service scheduling function will take into account the possible departure due to mobility. On the other hand, the file transfer requests take into account only the state of the network at the beginning of the slot (this assumption does not affect at all the employed analysis). 

Then, although the one-step transition probabilities have also the same form as in \eqref{trans}, they are more complicated. In Appendix \ref{onestep} we have derived the one-step transition probabilities for such a case. Moreover, due to the fact that we can have more than one departure during a slot, the resulting two-dimensional random walk is \textit{no longer of nearest-neighbour character}. 

In particular, the stochastic vectors $\{(\xi_{1n}^{(j)}(Q_{1,n},Q_{2,n}),\xi_{2n}^{(j)}(Q_{1,n},Q_{2,n})),n\geq0\}$ have range space $\{-2,-1,0,1\}\times\{-2,-1,0,1\}$. Note that the method introduced in the seminal books \cite{coh1,fay1} considered cases where at most one job of either type can depart. 
In our case we can see that the theory of boundary value problems \cite{coh1} can also be adapted to consider the case of at most two departures per job type. Thus, it seems that the analysis can also be adapted to the case where we can have a fixed number of departures (i.e., \textit{bounded batch departures}) from either type. To our best knowledge such a scenario has never been reported in the related literature, and reveals the flexibility of the analysis method. 

By applying the analysis in subsection \ref{gen}, the corresponding $g_{n_{2}}(x)$, $n_{2}=1,2,\ldots,N_{2}$, and $h_{n_{1}}(x)$, $n_{1}=1,2,\ldots,N_{1}$ have the same form as in \eqref{r2}, \eqref{r3}, i.e., they are functions of $g_{0}(x)$ and $h_{0}(y)$, respectively, where now the matrices $\mathbf{K}(x)=(k_{i,j}(x))$ (of order $N_{2}\times N_{2}$), $\mathbf{M}(y)=(m_{i,j}(y))$ (of order $N_{1}\times N_{1}$) are given by,
{\small{\begin{displaymath}
\begin{array}{rl}
k_{i,j}(x)=\left\{\begin{array}{ll}
-f_{4}(N_{1},i+1,x),&i=j-1,\\
-f_{3}(N_{1},i,x),&i=j,\\
f_{2}(N_{1},i-1,x),&i=j+1,\\
-f_{1}(N_{1},j,x),&i=j+2,\\
\end{array}\right.&\,m_{i,j}(y)=\left\{\begin{array}{ll}
-\tilde{f}_{4}(j+1,N_{2},y),&j=i-1,\\
-\tilde{f}_{3}(i,N_{2},y),&i=j,\\
\tilde{f}_{2}(j-1,N_{2},y),&j=i+1,\\
-\tilde{f}_{1}(j,N_{2},x),&j=i+2,\\
\end{array}\right.
\end{array}
\end{displaymath}}}
and the functional equation for the subregion $S_{d}$ has the form
\begin{displaymath}
[xy-\Psi(x,y)]g(x,y)=A(x,y)g_{0}(x)+B(x,y)h_{0}(y)+C(x,y), 
\end{displaymath}
where now,
{\small{\begin{displaymath}
\begin{array}{rl}
\Psi(x,y)=&\frac{y[y^{2}p_{-2,1}+yp_{-2,0}+p_{-2,-1}]+x[x^{2}p_{1,-2}+xp_{0,-2}+p_{-1,-2}]}{xy(1-p_{0,0}-xp_{1,0}-yp_{0,1}-xyp_{1,1}-y(p_{-1,0}+yp_{-1,1})-x(p_{0,-1}+xp_{1,-1})-p_{-1,-1}},\vspace{2mm}\\
A(x,y)=&\frac{y^{2}f_{1}(N_{1},N_{2}-1,x)e_{N_{2}-1}(x)-(yf_{3}(N_{1},N_{2},x)+f_{4}(N_{1},N_{2},x))e_{N_{2}}(x)-yf_{4}(N_{1},N_{2},x)e_{N_{2}+1}(x)}{xy(1-p_{0,0}-xp_{1,0}-yp_{0,1}-xyp_{1,1}-y(p_{-1,0}+yp_{-1,1})-x(p_{0,-1}+xp_{1,-1})-p_{-1,-1}},\vspace{2mm}\\
B(x,y)=&\frac{x^{2}\tilde{f}_{1}(N_{1}-1,N_{2},x)\tilde{e}_{N_{1}-1}(y)-(x\tilde{f}_{3}(N_{1},N_{2},y)+\tilde{f}_{4}(N_{1},N_{2},y))\tilde{e}_{N_{1}}(y)-y\tilde{f}_{4}(N_{1},N_{2},y)\tilde{e}_{N_{1}+1}(y)}{xy(1-p_{0,0}-xp_{1,0}-yp_{0,1}-xyp_{1,1}-y(p_{-1,0}+yp_{-1,1})-x(p_{0,-1}+xp_{1,-1})-p_{-1,-1}},\\
\end{array}
\end{displaymath}}}
and
{\small{\begin{displaymath}
\begin{array}{rl}
f_{1}(N_{1},n_{2},x)=&x^{3}p_{1,1}(N_{1},n_{2})+x^{2}p_{0,1}(N_{1},n_{2})+xp_{-1,1}(N_{1},n_{2})+p_{-2,0}(N_{1},n_{2},x),\vspace{2mm}\\
f_{2}(N_{1},n_{2},x)=&x^{2}[1-p_{0,0}(N_{1},n_{2})]-x^{3}p_{1,0}(N_{1},n_{2})-xp_{-1,0}(N_{1},n_{2})-p_{-2,0}(N_{1},n_{2},x),\vspace{2mm}\\
f_{3}(N_{1},n_{2},x)=&p_{-2,-1}(N_{1},n_{2})+xp_{-1,-1}(N_{1},n_{2})+x^{2}p_{0,-1}(N_{1},n_{2})+x^{3}p_{1,-1}(N_{1},n_{2}),\vspace{2mm}\\
f_{4}(N_{1},n_{2},x)=&x[p_{-1,-2}(N_{1},n_{2})+xp_{0,-2}(N_{1},n_{2})+x^{2}p_{1,-2}(N_{1},n_{2})+x^{2}p_{1,-1}(N_{1},n_{2})].
\end{array}
\end{displaymath}}}
Similar pattern follows for the polynomials $\tilde{f}_{k}(n_{1},N_{2},y)$, $k=1,2,3,4$. 

Note that although we have a similar behaviour as in the general model described in Section \ref{mod}, the form of $\Psi(x,y)$ as well as of the coefficients $A(x,y)$, $B(x,y)$, $C(x,y)$, have considerably changed, whereas we further have additional polynomial terms (i.e., the terms $f_{4}(N_{1},n_{2},x)$, $\tilde{f}_{4}(n_{1},N_{2},y)$).

However, the form of the resulting functional equation, and of $\Psi(x,y)$, ensures that the analysis performed in Section \ref{pre} can be used even in the current model. The overall conclusion is that the analysis in \cite{coh1} seems to be adapted to the case where we have \textit{bounded batch departures from either queue}. To our best knowledge, this result has never reported so far.
\subsection{A queue-based random access (RA) network}\label{appl1}
In the following, we focus on the modelling a queue-based RA network. In RA networks the main interest relies on stability and delay analysis. 

Delay analysis of (non queue-based) random access networks is quite challenging, and this is the reason of the quite limited number of analytical results in the related literature, e.g., \cite{dimpaptwc}. The vast majority in the related literature relies on the investigation of the stability conditions. We mentioned the works in \cite{Rao_TIT1988,LuoAE1999} that focused on interacting queues in ALOHA-type networks, and more recently, the work in \cite{PappasTWC2015} that dealt with cooperative wireless networks.

To retain the optimal throughput
performance of centralized policies in the distributed fashion, queue-based protocols are recently used \cite{gha,stol2}. Under such protocols, the transmission parameters of a node depend on the number of its backlogged packets; see in \cite{bouman1,castiel,borst2018} for progress regarding the stability conditions, and the heavy-traffic behaviour in such networks. However, it has been noticed that queue-based policies tend to result in long queues, and poor delay performance in order to ensure maximum stability \cite{bouman1,gha}; see also \cite{gu,shn2,stol2}. In these works, the actual queue lengths of the flows in each node's closest neighbourhood are used to determine the nodes' channel access
probabilities. However, they did not focus on the stationary behaviour.

Consider an ALOHA-type wireless network with two users communicating with a common destination node; see Figure \ref{fig1w}. Each user is equipped with an infinite capacity buffer for storing arriving and backlogged packets. The packet arrival processes are assumed to be independent from user to user and the channel is slotted in time. A time slot corresponds to the transmission time (i.e., the service time) of a packet. 

In such a shared access network, it is assumed a minimum exchanging information of one bit between the nodes, which allows them to be aware of the state of each other. The knowledge of the state of the network by a node provides the additional flexibility towards self-aware and dynamically adapted networks. To the best of our knowledge, this variation of RA has not been reported in the literature.

However, with the current technology it is rather difficult for a node to be completely aware of the state of a neighbour node. Thus, the transmission/packet generation parameters at each user node may depend \textbf{only} on their state (i.e., buffer occupancy); see \cite{stol2,bor1,bor2,gha}. However, due to the fact that user nodes have limited memory and capacity, the employed dynamic scheduling policy of user node $k$ is based only on a subset of its backlogged packets, say $N_{k}$, and not on the entire number of buffered packets.

Let $Q_{k,n}$, $k=1,2,$ be the number of stored packets at the buffer of user $k$, at the beginning of the $n$th slot. Then $\mathbf{Q}=\{(Q_{1,n},Q_{2,n}),n=0,1,...\}$ is a discrete time Markov chain with state space $S=\{\underline{n}=(n_{1},n_{2});n_{k}\geq0,k=1,2\}$.

At the beginning of each slot, each user adapts its transmission parameter based on the number of its backlogged packets. More precisely, if user node $k$, $k=1,2,$ stores $n_{k}$ packets at the beginning of a slot, it transmits a packet to the destination node with probability $a_{k}(n_{k})=1-r^{n_{k}}_{k}$ (with prob. $\bar{a}_{k}(n_{k})=r^{n_{k}}_{k}$ remains silent), where $0<r_{k}<1$. If both user nodes transmit at the same slot there is a collision, and both packets have to be retransmitted in a later slot. The term $r_{k}$, $k=1,2$, is called the service parameter of queue $k$. Thus, in such a scheme, the larger the queue length, the larger the
probability that the user will transmit a packet in a slot. Therefore, this \textit{probabilistic} service discipline is consistent with the \textit{weighted longest queue first discipline} \cite{leepatent}, since it prioritizes the most loaded user node. However, it is far more general since it does not oblige the least loaded user node to remain silent (i.e., it does not give absolute priority to the longest queue). Thus, it incorporates the benefits of RA scheduling, and at the same time it prioritizes the longest queue, which is recognized for its high performance in practice.

Note that if a user node transmits with probability $a_{k}(n_{k})=r_{k}^{n_{k}}$, $k=1,2,$ then, our service discipline is consistent with a RA scheme, which prioritizes the least loaded queue, i.e., a \textit{probabilistic shortest queue first policy}. 

Packet generation (i.e., arrivals) at user node $k$, $k=1,2$ are assumed to be i.i.d. random variables from slot to slot. Let $A_{k,n}(n_{k})$ be the number of packets that arrive at the buffer of user node $k$ at $(n,n+1]$ given that at the beginning of the $n$th slot the state of user node $k$ is $n_{k}$. We assume Bernoulli arrivals with the average number of arrivals being $\mathbb{E}(A_{k,n}(n_{k}))=\lambda_{k}(n_{k})<\infty$ packets per slot. It is natural to assume that $\lambda(n_{k})=\lambda_{k}^{n_{k}}$,  $0<\lambda_{k}<1$, $k=1,2$, i.e., the larger the queue length is, the smaller the packet generation probability in a slot. In such a situation, we have assumed that the packet generation rate at each user depends on its buffer occupancy level. 

We consider a \textit{limited} queue-based transmission  and packet generation (LQBT-G) protocol. Note that the LQBT-G policy arises quite natural in wireless systems. More precisely, it is well known that user nodes have \textit{limited memory, capacity and hardware resources}. Thus, to incorporate a dynamic transmission/packet generation policy based on the buffer occupancy level, it is normal to assume that the user's node adapts its transmission/packet generation parameters based on a subset of its backlogged packets, i.e., the node's $k$ parameters are configured by an upper level at its buffer occupancy, say $N_{k}$.

Let $N_{1}$, $N_{2}$ be the configuration levels at user's 1, 2 queue, respectively. Then, for $k=1,2,$
\begin{displaymath}
a_{k}(n_{k})=\begin{cases}
\begin{array}{rl}
1-r_{k}^{n_{k}},&if\,n_{k}<N_{k},\\
1-r_{k}^{N_{k}},&if\,n_{k}\geq N_{k},\\
\end{array}\end{cases}\,\,\lambda_{k}(n_{k})=\begin{cases}
\begin{array}{rl}
\lambda_{k}^{n_{k}},&if\,n_{k}<N_{k},\\
\lambda_{k}^{N_{k}},&if\,n_{k}\geq N_{k}.
\end{array}\end{cases}
\end{displaymath}
 The one step transition probabilities from $\underline{n}=(n_{1},n_{2})$ to $(n_{1}+i,n_{2},+j)$, say $p_{i,j}(\underline{n})$, where $\underline{n}\in S$, $i,j=-1,0,1$, are given by:
\begin{displaymath}
\begin{array}{rl}
p_{1,0}(\underline{n})=&(\bar{a}_{1}(n_{1})\bar{a}_{2}(n_{2})+a_{1}(n_{1})a_{2}(n_{2}))d_{1,0}(\underline{n})+\bar{a}_{1}(\underline{n})a_{2}(\underline{n})d_{1,1}(\underline{n}),\\
p_{0,1}(\underline{n})=&(\bar{a}_{1}(n_{1})\bar{a}_{2}(n_{2})+a_{1}(n_{1})a_{2}(n_{2}))d_{0,1}(\underline{n})+\bar{a}_{2}(\underline{n})a_{1}(\underline{n})d_{1,1}(\underline{n}),\\
p_{1,1}(\underline{n})=&(\bar{a}_{1}(n_{1})\bar{a}_{2}(n_{2})+a_{1}(n_{1})a_{2}(n_{2}))d_{1,1}(\underline{n}),\\
p_{-1,1}(\underline{n})=&a_{1}(n_{1})\bar{a}_{2}(n_{2})d_{0,1}(\underline{n}),\\
p_{1,-1}(\underline{n})=&a_{2}(n_{2})\bar{a}_{1}(n_{1})d_{1,0}(\underline{n}),\\
p_{-1,0}(\underline{n})=&a_{1}(n_{1})\bar{a}_{2}(n_{2})d_{0,0}(\underline{n}),\\
p_{0,-1}(\underline{n})=&a_{2}(n_{2})\bar{a}_{1}(n_{2})d_{0,0}(\underline{n}),\\
p_{0,0}(\underline{n})=&(\bar{a}_{1}(n_{1})\bar{a}_{2}(n_{2})+a_{1}(n_{1})a_{2}(n_{2}))d_{0,0}(\underline{n})+\bar{a}_{1}(n_{1})a_{2}(n_{2})d_{0,1}(\underline{n})\\&+\bar{a}_{2}(n_{2})a_{1}(n_{1})d_{1,0}(\underline{n}),
\end{array}
\end{displaymath}
where 
\begin{equation}
d_{i,j}(\underline{n})=\left\{\begin{array}{rl}
\lambda_{1}(n_{1})\bar{\lambda}_{2}(n_{2}),&i=1,j=0,\\
\lambda_{2}(n_{2})\bar{\lambda}_{1}(n_{1}),&i=0,j=1,\\
\lambda_{1}(n_{1})\lambda_{2}(n_{2}),&i=1,j=1,\\
\bar{\lambda}_{1}(n_{1})\bar{\lambda}_{2}(n_{2}),&i=0,j=0.
\end{array}\right.
\label{hjj}
\end{equation}
and $\bar{\lambda}_{k}(n_{k})=1-\lambda_{k}(n_{k})$, $k=1,2$. Note that by setting the configuring parameters $N_{1}$, $N_{2}$ we split the state space $S$ as given in \eqref{split}. For $(\underline{n})\in S_{d}$, $a_{k}(n_{k})=a_{k}(N_{k}):=a_{k}$, $\lambda_{k}(n_{k})=\lambda_{k}(N_{k})$, $k=1,2,$ and then
\begin{equation}
\begin{array}{rl}
R(x,y)=&xy-D(x,y)[xy+a_{1}\bar{a}_{2}y(1-x)+\bar{a}_{1}a_{2}x(1-y)],\\
D(x,y)=&(\bar{\lambda}_{1}+\lambda_{1}x)(\bar{\lambda}_{2}+\lambda_{2}y).
\end{array}
\end{equation}
The rest of the analysis follows the lines in Appendix  \ref{fayo}.

We now focus on the stability conditions using Theorem \ref{faystab}, and the notion of induced Markov chains (see Theorem \ref{ergth}). Remind that for $Q_{1,n}>N_{1}$ (resp. $Q_{2,n}>N_{2}$) the component $Q_{2,n}$ (resp. $Q_{1,n}$) evolves as a one-dimensional RW, and let its corresponding stationary distribution $\psi:=(\psi_{0},\psi_{1},...)$ (resp. $\varphi:=(\varphi_{0},\varphi_{1},...)$); see also Appendix \ref{apc}. Then, the mean drifts
\begin{displaymath}
\begin{array}{rl}
E_{x}^{(n_{2})}=&\lambda_{1}(N_{1},n_{2})-a_{1}(N_{1},n_{2})\bar{a}_{2}(N_{1},n_{2}),\,\forall n_{1}\geq N_{1},\\
E_{y}^{(n_{1})}=&\lambda_{2}(n_{1},N_{2})-a_{2}(n_{1},N_{2})\bar{a}_{1}(n_{1},N_{2}),\,\forall n_{2}\geq N_{2}.
\end{array}
\end{displaymath}
Since $a_{k}(\underline{n}):=a_{k}$, $\lambda_{k}(\underline{n})=\lambda_{k}$, $k=1,2,$ for $\underline{n}\in S_{d}$, we have that $E_{x}=\lambda_{1}-a_{1}\bar{a}_{2}$, $n_{2}\geq N_{2}$, $E_{y}=\lambda_{2}-a_{2}\bar{a}_{1}$, $n_{1}\geq N_{1}$.

The proof of Theorem \ref{ergth} is identical to those in \cite[Theorem 3.1, p. 178]{faysta}; see also \cite{malmen,za}. We also mention \cite[Chapters 6-8, Theorem 28.1, p. 322]{boro} for a thorough investigation of the stability conditions of general spatially homogeneous multidimensional Markov chains.
\begin{theorem}\label{ergth}
\begin{enumerate}
\item If $\lambda_{1}<a_{1}\bar{a}_{2}$, $\lambda_{2}<a_{2}\bar{a}_{1}$, $\mathbf{Q}$ is 
\begin{enumerate}
\item ergodic if
\begin{equation}
\begin{array}{rl}
\lambda_{1}(1-\sum_{k=0}^{N_{2}-1}\psi_{k})<&a_{1}\bar{a}_{2}(1-\sum_{k=0}^{N_{2}-1}\psi_{k})-\sum_{k=0}^{N_{2}-1}\gamma_{k}\psi_{k},\,and\vspace{2mm}\\
\lambda_{2}(1-\sum_{k=0}^{N_{1}-1}\varphi_{k})<&a_{2}\bar{a}_{1}(1-\sum_{k=0}^{N_{1}-1}\varphi_{k})-\sum_{k=0}^{N_{1}-1}\delta_{k}\varphi_{k}.
\end{array}
\end{equation}
\item transient if
\begin{equation}
\begin{array}{rl}
\lambda_{1}(1-\sum_{k=0}^{N_{2}-1}\psi_{k})>&a_{1}\bar{a}_{2}(1-\sum_{k=0}^{N_{2}-1}\psi_{k})-\sum_{k=0}^{N_{2}-1}\gamma_{k}\psi_{k},\,or\vspace{2mm}\\
\lambda_{2}(1-\sum_{k=0}^{N_{1}-1}\varphi_{k})>&a_{2}\bar{a}_{1}(1-\sum_{k=0}^{N_{1}-1}\varphi_{k})-\sum_{k=0}^{N_{1}-1}\delta_{k}\varphi_{k}.
\end{array}
\end{equation}
\end{enumerate}
\item If $\lambda_{1}\geq a_{1}\bar{a}_{2}$, $\lambda_{2}<a_{2}\bar{a}_{1}$, $\mathbf{Q}$ is
\begin{enumerate}
\item ergodic if
\begin{equation*}
\begin{array}{rl}
\lambda_{1}(1-\sum_{k=0}^{N_{2}-1}\psi_{k})<&a_{1}\bar{a}_{2}(1-\sum_{k=0}^{N_{2}-1}\psi_{k})-\sum_{k=0}^{N_{2}-1}\gamma_{k}\psi_{k}.
\end{array}
\end{equation*}
\item transient if 
\begin{equation*}
\begin{array}{rl}
\lambda_{1}(1-\sum_{k=0}^{N_{2}-1}\psi_{k})>&a_{1}\bar{a}_{2}(1-\sum_{k=0}^{N_{2}-1}\psi_{k})-\sum_{k=0}^{N_{2}-1}\gamma_{k}\psi_{k},
\end{array}
\end{equation*}
or when $\lambda_{1}> a_{1}\bar{a}_{2}$ and $\lambda_{1}(1-\sum_{k=0}^{N_{2}-1}\psi_{k})=a_{1}\bar{a}_{2}(1-\sum_{k=0}^{N_{2}-1}\psi_{k})-\sum_{k=0}^{N_{2}-1}\gamma_{k}\psi_{k}$.
\end{enumerate}

\item If $\lambda_{1}< a_{1}\bar{a}_{2}$, $\lambda_{2}\geq a_{2}\bar{a}_{1}$, $\mathbf{Q}$ is
\begin{enumerate}
\item  ergodic if 
\begin{equation*}
\begin{array}{rl}
\lambda_{2}(1-\sum_{k=0}^{N_{1}-1}\varphi_{k})<&a_{2}\bar{a}_{1}(1-\sum_{k=0}^{N_{1}-1}\varphi_{k})-\sum_{k=0}^{N_{1}-1}\delta_{k}\varphi_{k}.
\end{array}
\end{equation*}
\item transient if
\begin{equation*}
\begin{array}{rl}
\lambda_{2}(1-\sum_{k=0}^{N_{1}-1}\varphi_{k})>&a_{2}\bar{a}_{1}(1-\sum_{k=0}^{N_{1}-1}\varphi_{k})-\sum_{k=0}^{N_{1}-1}\delta_{k}\varphi_{k},
\end{array}
\end{equation*}
or when $\lambda_{2}>a_{2}\bar{a}_{1}$ and $\lambda_{2}(1-\sum_{k=0}^{N_{1}-1}\varphi_{k})=a_{2}\bar{a}_{1}(1-\sum_{k=0}^{N_{1}-1}\varphi_{k})-\sum_{k=0}^{N_{1}-1}\delta_{k}\varphi_{k}$.
\end{enumerate}
\item If $\lambda_{1}\geq a_{1}\bar{a}_{2}$, $\lambda_{2}\geq a_{2}\bar{a}_{1}$, $\mathbf{Q}$ is transient.
\end{enumerate}
\end{theorem}\vspace{-0.2in}
\section{Numerical example}\label{num}
For the numerical illustration, we focus on a queue-based RA network of two users with collisions, which is described by a PH-NNRWQP with no transitions to the South-West, and follow Appendix \ref{fayo} to handle the solution of functional equation \eqref{fun}. 
\paragraph{Queueing analysis:} The main steps are briefly described in the following:
\begin{enumerate}
\item For ease of computations we focus on a very simple example where $N_{1}=N_{2}=2$. We firstly determine the boundary functions $g_{0}(x)$ and $h_{0}(y)$ in terms of the unknown probabilities $\pi(N_{1},n_{2})$, $n_{2}=0,1,\ldots,N_{2}-1$, $\pi(n_{1},N_{2})$, $n_{1}=0,1,\ldots,N_{1}-1$, and $\pi(N_{1},N_{2})$. 

Both Riemann-Hilbert problems (for $g_{0}(x)$ and for $h_{0}(y)$) had for varying values of system parameters strictly negative indices. In most of the cases the value of the index equals $-2$. It seems that this holds (i.e., $\chi=-N$) for ergodic two-dimensional partially homogeneous random walks with $N_{1}=N_{2}=N$, but this conjecture requires further justification. Moreover, as shown in subsection \ref{rhp}, in case of a negative index, for the existence of a unique solution of the Riemann-Hilbert problems, additional conditions must be satisfied; see \eqref{index1}. In our case, these conditions produced additional equations on the unknown probabilities, but they were direct consequences of the equilibrium equations that are associated with $S_{a}$.
\item After obtaining $g_{0}(x)$ and $h_{0}(y)$, we can use them to produce $N_{1}+N_{2}$ equations by appropriate differentiation; see at the end of subsection \ref{rhp}:
\begin{displaymath}
\begin{array}{rl}
n_{2}!\pi(N_{1},n_{2})=\frac{d^{n_{2}}}{dx^{n_{2}}}[e_{n_{2}}(x)g_{0}(x)+t_{n_{2}}(x)]|_{x=0},\,n_{2}=1,...,N_{2},\vspace{2mm}\\
n_{1}!\pi(n_{1},N_{2})=\frac{d^{n_{1}}}{dy^{n_{1}}}[\tilde{e}_{n_{1}}(y)h_{0}(y)+\tilde{t}_{n_{1}}(y)|_{y=0},\,n_{1}=1,...,N_{1}.
\end{array}
\end{displaymath}
Note that each coefficient in these equations requires the evaluation of complex integrals of type (\ref{sool1}). In order to numerically evaluate them, we have firstly to construct the conformal mappings. Note that in most of the cases we are not able to obtain them explicitly. However, an efficient numerical approach was developed in \cite{coh1}, Sec. IV.1.1; see below for more details on the spesific example. Alternatively, since contours are close to ellipses, we can use the nearly circular approximation, \cite{neh}. The function $U(x)$ that appears in the boundary condition (\ref{bou}), involves determinants of matrices whose elements are polynomials.

Along with these equations, we use the normalizing condition (\ref{nor}).
\item Using these $N_{1}+N_{2}+1$ equations along with the $N_{1}\times N_{2}$ equations that refer to the states in region $S_{a}$, we build an $(N_{1}+1)\times(N_{2}+1)$ system of equations. Its solution provides $\pi(N_{1},n_{2})$, $n_{2}=0,1,\ldots,N_{2}-1$, $\pi(n_{1},N_{2})$, $n_{1}=0,1,\ldots,N_{1}-1$, and $\pi(N_{1},N_{2})$. Therefore, $g_{0}(x)$ and for $h_{0}(y)$ is known, and so it is $g(x,y)$ as well as $g_{n_{2}}(x)$, $n_{2}=1,\ldots,N_{2}$, $h_{n_{1}}(y)$, $n_{1}=1,\ldots,N_{1}$.


\end{enumerate} 

As already stated, we focus on the simple case $N_{1}=N_{2}=2$, and let for $||n||=n_{1}+n_{2}$, 
$a_{k}(\underline{n}):=a(\underline{n})=a\frac{n_{k}}{||n||}$, $\lambda_{k}(\underline{n}):=\lambda(\underline{n})=\lambda 2^{-||n||}$, with $a(\underline{n})=a$, and $\lambda(\underline{n})=\lambda$, for $\underline{n}\in S_{d}$. Note that the transmission probability $a_{k}(\underline{n})$ of user $k$ is an increasing function of the number of connecting users up to the level $N_{k}$ (i.e., $a_{k}(\underline{n})=a\frac{n_{k}}{||n||}$), and for $n_{k}>N_{k}$ remains fixed and equal to $a$. Thus, the nodes dynamically increase the transmission probability as demand increases, but due to their limited capabilities, there is an upper limit equal to $a$. Such a setting may also dictated by energy consumption constraints, since nodes are usually battery operated. Thus, the less buffer occupancy, the smaller the transmission probability, which results in less energy consumption. 

The calculation of basic metrics, such as $E(Q_{1})$ (resp. $E(Q_{2})$), requires the evaluation of integrals (\ref{zx}), (\ref{sool1}), and the values of $\gamma(1)$, $\gamma^{\prime}(1)$. Note that $\gamma(1)$ is the unique solution of $\gamma_{0}(\eta)=1$ in $[0,1]$, and it is determined numerically using (\ref{zx}) and the Newton-Raphson method. Moreover, $\gamma^{\prime}(1)=(\frac{d}{dz}\gamma_{0}(z)|_{z=\eta})^{-1}$, which also can be determined numerically using the trapezium rule. 

We now outline how these integrals can be computed: Firstly, we rewrite the integral (\ref{sool1}) by substituting $t=e^{i\phi}$ (similar work must be done also for the expression of $h_{0}(y)$) as well as for the derivatives of $g_{0}(x)$, $h_{0}(y)$:
\begin{equation*}
\begin{array}{rl}
g_{0}(x)=\prod_{i=1}^{k}(x-\xi_{i})^{-1}e^{i\sigma(\gamma(x))}(\gamma(x))^{\chi}[iK+\frac{1}{2\pi}\int_{0}^{2\pi}e^{\omega_{1}(e^{i\phi})}\delta(e^{i\phi})\frac{e^{i\phi}+\gamma(x)}{e^{i\phi}-\gamma(x)}d\phi].
\end{array}
\end{equation*}

Then, we split the interval $[0,2\pi]$ into $J$ parts of length $2\pi/J$. For the $J$ points given by their angles $\left\{\phi_{0},...,\phi_{J-1}\right\}$, we solve the Theodorsen's integral equation (\ref{zx}),
to obtain iteratively the points $\left\{\tilde{\psi}(\phi_{1}),...,\tilde{\psi}(\phi_{J-1})\right\}$ from:
\begin{displaymath}
\begin{array}{rl}
\tilde{\psi}_{0}(\phi_{k})=\phi_{k},&\,
\tilde{\psi}_{n+1}(\phi_{k})=\phi_{k}-\frac{1}{2\pi}\int_{0}^{2\pi}\log\left\{\frac{\delta(\tilde{\psi}_{n}(\omega))}{\cos(\tilde{\psi}_{n}(\omega))}\right\}\cot[\frac{1}{2}(\omega-\phi_{k})]d\omega,
\end{array}
\end{displaymath}
where $\delta(\tilde{\psi}_{n}(\omega))$ is determined by $\delta(\tilde{\psi}_{n}(\omega))-\cos(\psi_{n}(\omega))\sqrt{m(\tilde{\delta}(\psi_{n}(\omega)))}=0$, using the Newton-Raphson method with stopping criterion: \begin{displaymath}
\begin{array}{c}
\max_{k\in\left\{0,1,...,K-1\right\}}\left|\tilde{\psi}_{n+1}(\phi_{k})-\tilde{\psi}_{n}(\phi_{k})\right|<10^{-6}.
\end{array}
\end{displaymath}
Having obtained $\tilde{\psi}(\phi_{k})$ numerically, the values of $\gamma_{0}(z)$, $\left|z\right|\leq 1$ are given by
\begin{displaymath}
\begin{array}{c}
\gamma_{0}(e^{i\phi_{k}})=e^{i\tilde{\psi}(\phi_{k})}\frac{\delta(\tilde{\psi}(\phi_{k}))}{\cos(\tilde{\psi}(\phi_{k}))}=\delta(\tilde{\psi}(\phi_{k}))[1+i \tan(\tilde{\psi}(\phi_{k}))],\,k=0,1,...,J-1.
\end{array}
\end{displaymath}
We proceed with the determination of $\gamma(1)$, $\gamma^{\prime}(1)$. Clearly, $\gamma(1)=\eta$ means $\gamma_{0}(\eta)=1$. Thus, $\gamma(1)$ is the unique solution of $\gamma_{0}(\eta)=1$ in $[0,1]$, and can be obtained using (\ref{zx}) and the Newton-Raphson method. Furthermore, 
\begin{displaymath}
\begin{array}{l}
\gamma^{\prime}(1)=(\frac{d}{dz}\gamma_{0}(z)|_{z=\eta})^{-1}=\{\frac{1}{\gamma(1)}+\frac{1}{2\pi}\int_{0}^{2\pi}\frac{\log\{\rho(\tilde{\psi}(\omega))\}2e^{i\omega}}{(e^{i\omega}-\gamma(1))^{2}}d\omega\}^{-1},
\end{array}
\end{displaymath}
is numerically determined using the trapezium rule. For our numerical example we divided the interval $[0,2\pi)$ into $J = 1000$ equal parts and apply
the trapezium numerical integration approach described above. Then, we derive the required elements to obtain $E(Q_{1})$ in \eqref{ex1} (similarly for the $E(Q_{2})$).

In particular, in Figure \ref{fig2}, the total expected number of buffered packets $E(Q_{1}+Q_{2})$ is presented as a function of $\lambda$, $a$. Definitely, by increasing $a$, the delay in queue can be handled as long as $\lambda$ remains in small values. However, we can see there is no significant benefit if at the same time we increased $a$. This is because by increasing $a$, we also increase the possibility of a collision, which results in unsuccessful transmissions. However, by increasing $\lambda$, we observe the increase on the total expected number of buffered packets, as expected.

\begin{figure}
\centering
\includegraphics[scale=0.65]{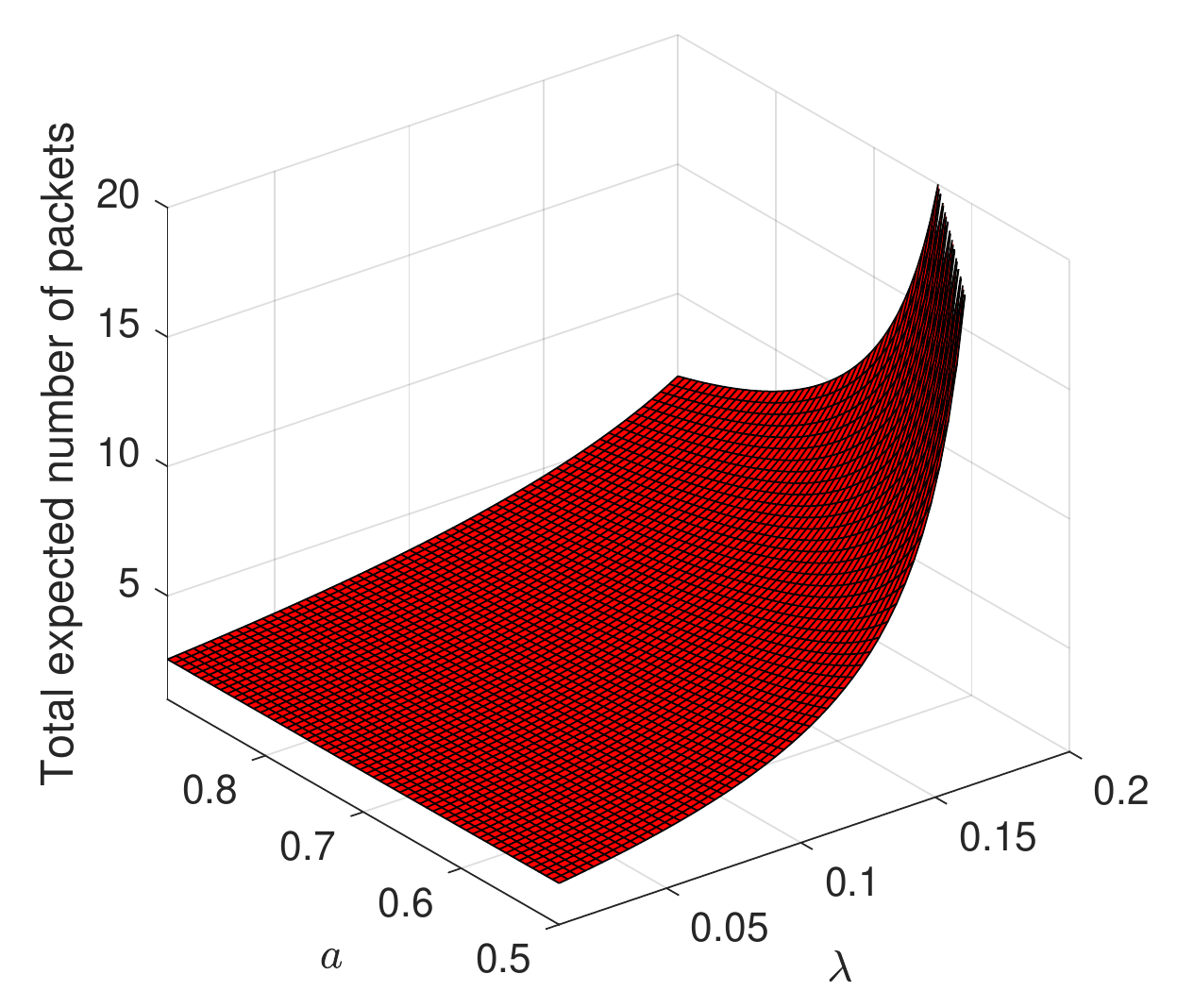}
\caption{Effect of $\lambda$, $a$ on the  total expected number of buffered packets.}\label{fig2}
\end{figure}

Moreover, in Figure \ref{fig3} we observe how the probability of an empty system varies as we vary the values of $N_{1}$, $N_{2}$ and $\lambda$. The overall conclusion is that when we increase the values of $N_{1}$, $N_{2}$, the transmission probabilities increase, and thus we expect an increase of the probability of an empty system.
\begin{figure}
\centering
\includegraphics[scale=0.5]{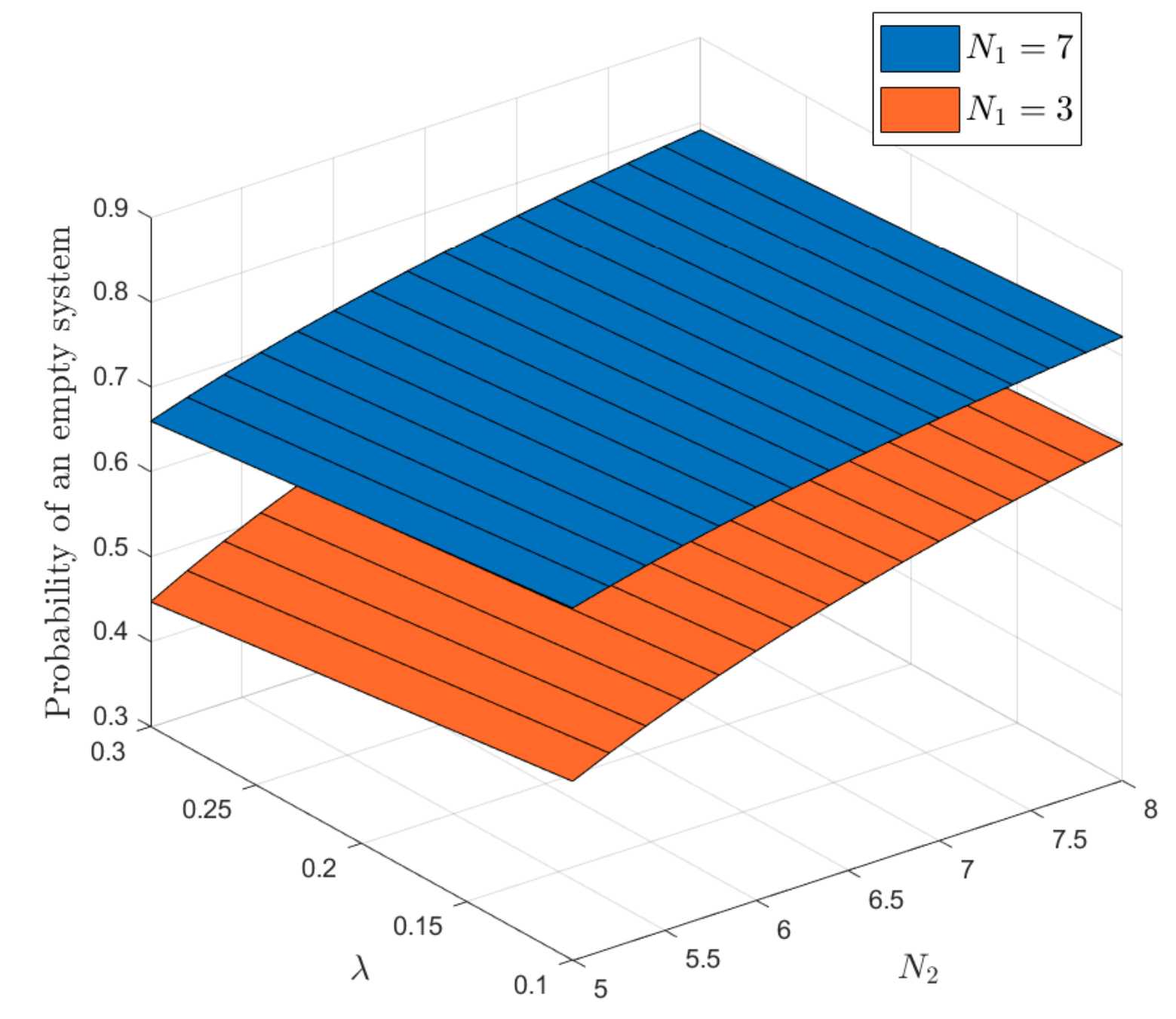}
\caption{The probability of an empty system for $a=0.6$.}\label{fig3}
\end{figure}
\paragraph{Stability condition:} Although our main concern in this work relies on the investigation of the stationary behaviour, it would be interesting to observe how the stability region is affected as we change the values of the system  parameters. Set $N_{1}=N_{2}=2$, $\lambda_{2}(\underline{n})=0.5^{n_{2}}$, $n_{1}=0,1$, $n_{2}=0,1,2$, and $a_{1}(\underline{n})=0.8\frac{n_{1}}{n_{1}+n_{2}}$, $a_{2}(\underline{n})=0.6\frac{n_{2}}{n_{1}+n_{2}}$. As a stability region, we define the set of arrival rates $(\lambda_{1}(2,2),\lambda_{2}(2,2))$ for which the queues are stable. Here, by setting the transmission probabilities at each queue as above, we are interesting on how the packet generation rate at queue 1 when $Q_{1}\leq N_{1}$, i.e., in the region $S_{a}$, affects the set of arrival rates $(\lambda_{1}(2,2),\lambda_{2}(2,2))$. Let also
\begin{equation*}
\begin{array}{rl}
h_{1}:=&\lambda_{1}(2,2)(1-\sum_{k=0}^{1}\psi_{k})-a_{1}(2,2)\bar{a}_{2}(2,2)(1-\sum_{k=0}^{1}\psi_{k})+\sum_{k=0}^{1}\gamma_{k}\psi_{k},\\
h_{2}:=&\lambda_{2}(2,2)(1-\sum_{k=0}^{1}\phi_{k})-a_{2}(2,2)\bar{a}_{1}(2,2)(1-\sum_{k=0}^{1}\phi_{k})+\sum_{k=0}^{1}\delta_{k}\phi_{k}.
\end{array}
\end{equation*}
Recall that $a_{1}(\underline{n})=a_{1}(2,2)$, $a_{2}(\underline{n})=a_{2}(2,2)$ for $\underline{n}\in S_{d}$.

In Figure \ref{stabi} we observe how the stability region is affected by varying $\lambda_{1}(\underline{n})$. In particular, when $\lambda_{1}(\underline{n})=0.2^{n_{1}}$, for $n_{1}=0,1,2$, $n_{2}=0,1$, the stability region is given by the domain $ABCO$, where $O:=(0,0)$. Note that the smaller value of $\lambda_{1}(\underline{n})$ with respect to $\lambda_{2}(\underline{n})$ allows $\lambda_{1}(2,2)$ to take relatively large values with respect to $\lambda_{2}(2,2)$. When $\lambda_{1}(\underline{n})=0.8^{n_{1}}$, the stability region becomes the domain $DEFO$, which seems to be more fair for $\lambda_{1}(2,2)$, $\lambda_{2}(2,2)$.

Another interesting observation is that contrary to the standard ALOHA-type network with collisions, the stability region of our queue-based version is a convex set as shown in Figure \ref{stabi}. We mention here that convexity of the stability region of ALOHA-type RA networks is accomplished only for the case where we allow simultaneous successful transmissions (i.e, the multi-packet reception policy \cite{dimpaptwc}). This is a quite interesting result, since non-convexity of stability region in ALOHA type networks is due to the presence of collisions. It seems that by employing the queue-based access scheme helps to overcome this problem. Clearly, such a conjecture for any queue-based RA network of interacting queues requires further justification.
\begin{figure}
\centering
\includegraphics[scale=0.5]{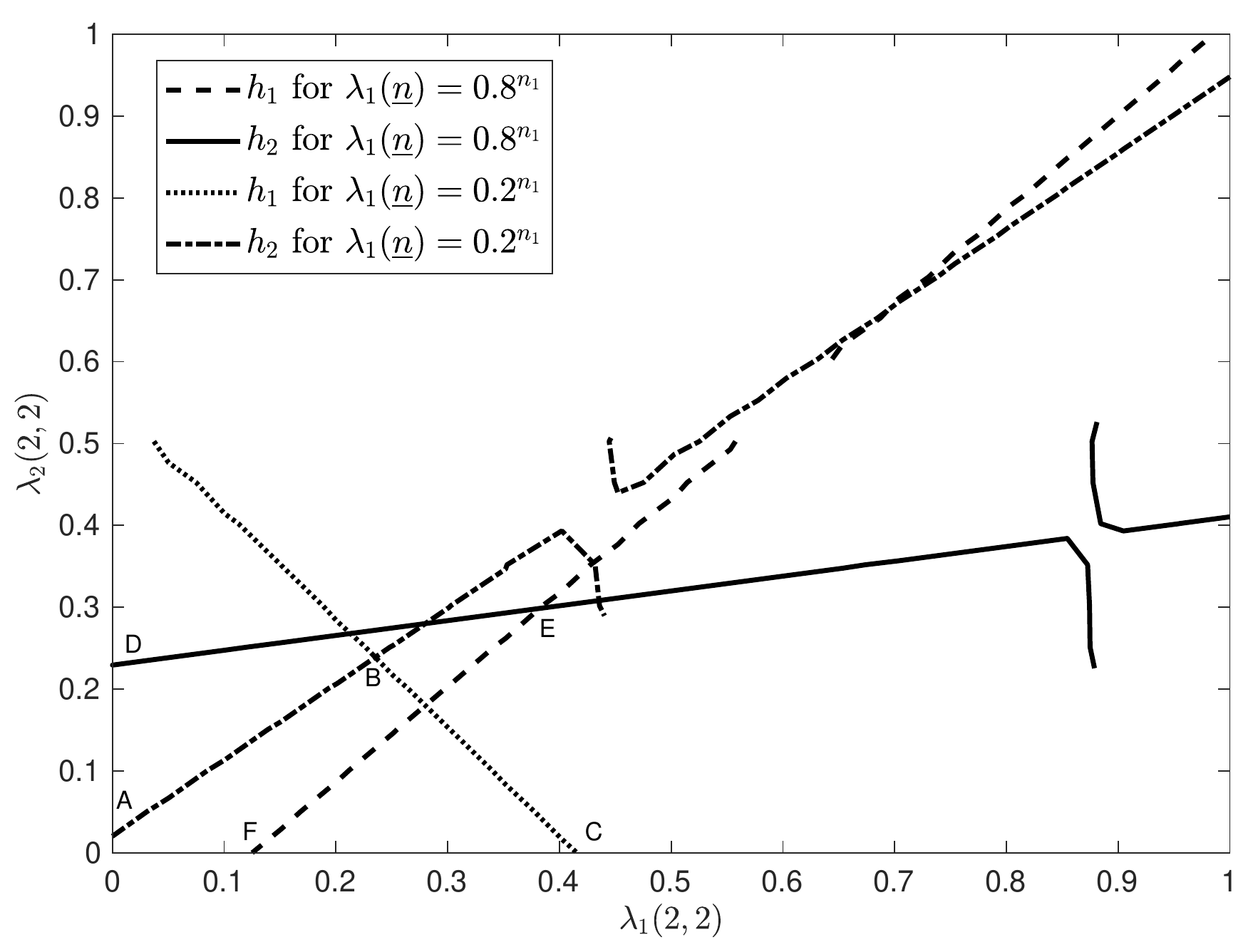}
\caption{The effect of $\lambda_{1}(\underline{n})$ on the stability region.}\label{stabi}
\end{figure}\vspace{-0.2in}
\section{Conclusion}
In this work, we provided an analytical approach to analyse the stationary behaviour of a partially homogeneous nearest-neighbour random walk in the quarter plane. We show that its stationary distribution is investigated by: $1)$ solving a finite system of linear equations, $2)$ solving two matrix functional equations in terms of $g_{0}(x)$, $h_{0}(y)$, and $3)$ solving a functional equation using the theory of Riemann (-Hilbert) boundary value problem.

This class of random walks can be used to model discrete-time two-dimensional queues with limited state-dependency. More precisely, it is flexible enough to model a wide range of general arrival/service scheduling policies. Furthermore, it provides a general framework to analyse other novel queueing models, such as the LDGPS system with/without impatience, and a two-node queue-based RA network. By theoretical point of view, the approach can be adapted to cases of bounded batch departures from each queue, a result that has never reported in the related literature so far. In a future work, we plan to further investigate the numerical implementation of the approach as well as to compare it with other well known numerical oriented approaches. A challenging future task is to extend this approach to the case of Markov-modulated PH-RWQP.  
\appendix
\section{Overview of various scheduling policies based on our framework}\label{discu}
In the following, we provide an overview of several other scheduling disciplines that are based on our modelling framework, and can be studied by our approach. 
\begin{enumerate}
\item Queue based traffic control in telecommunication networks \cite{choikimlee} (e.g., Asynchronous Transfer Mode networks (ATM), or Broadband Integrated Services Digital Network (B-ISDN)). The authors used a two-dimensional queueing system with two \textit{finite capacity} queues. Under the employed scheduling scheme, the next customer to be served is of type $k$ with probability $a_{k}(i,j)$, $k = 1, 2$, when there are $i$ type-1 and $j$ type-2 jobs, with $(i,j)\in\{0,1,\ldots,K_{1}\}\times\{0,1,\ldots,K_{2}\}$ and $a_{1}(i,j)+a_{2}(i,j)=1$ for all $i$ and $j$.
\item Head of the line (HOL) priority scheduling policy, e.g., \cite{takagi}. In such a scenario, we have e.g., $\beta_{1}(n_{1},n_{2})=1$, for $n_{1}>0$. i.e., type 1 jobs have HOL priority over type 2 jobs. Our modelling framework is even more general and incorporates a \textit{mixed} HOL policy, under which priority is assigned to a class based on the inhomogeneous regions, e.g., $\beta_{1}(n_{1},n_{2})=1$, for $(n_{1},n_{2})\in S_{a}$ and $\beta_{2}(n_{1},n_{2})=1$, for $(n_{1},n_{2})\in S_{d}$. 
\item Shortest queue first policy (SQF), e.g., \cite{guillemin1}: The next job to be served is chosen by the least loaded queue. Such a scheduling discipline was recently used to improve the quality of Internet access on high-speed communication links \cite{carofaglioetal}. In such a scenario:
\begin{displaymath}
\beta_{1}(n_{1},n_{2})=\left\{\begin{array}{ll}
1,&n_{1}\leq n_{2},\\
0,&n_{1}>n_{2},
\end{array}\right.\,\,
\beta_{2}(n_{1},n_{2})=\left\{\begin{array}{ll}
1,&n_{2}\leq n_{1},\\
0,&n_{2}>n_{1}.
\end{array}\right.
\end{displaymath}
\item Longest queue first policy (LQF), e.g., \cite{cohenlong}: Also known as the greedy maximal
scheduling policy, it is recognized for its high performance in practice and chooses the next customer to be served by the longest queue. In such a scenario:
\begin{displaymath}
\beta_{1}(n_{1},n_{2})=\left\{\begin{array}{ll}
1,&n_{1}> n_{2},\\
0,&n_{1}\leq n_{2},
\end{array}\right.\,\,
\beta_{2}(n_{1},n_{2})=\left\{\begin{array}{ll}
1,&n_{2}> n_{1},\\
0,&n_{2}\leq n_{1}.
\end{array}\right.
\end{displaymath}
\item Bernoulli scheduling policy \cite{dslee,feng}: Upon a service completion, a type $k$ customer is chosen with probability $p_{k}$. In such a scenario, $\beta_{k}(n_{1},n_{2})=p_{k}$, $k=1,2$ with $p_{1}+p_{2}=1$.
\item Queue-length-threshold (QLT) policy \cite{knessl,boxmadown}: Such a policy is applied for traffic control in telecommunication systems, such as ATM networks. We set a constant $L$ in queue 1, and the next job to be served is chosen according to the following rule:
\begin{displaymath}
\beta_{1}(n_{1},n_{2})=\left\{\begin{array}{ll}
1,&n_{1}> L,\\
0,&n_{1}\leq L,
\end{array}\right.\,\,
\beta_{2}(n_{1},n_{2})=\left\{\begin{array}{ll}
1,&n_{1}\leq L,\\
0,&n_{1}>L.
\end{array}\right.
\end{displaymath}
For $L=0$, we give HOL priority to queue 1. For $L=\infty$, then we have the 1-limited discipline, e.g., \cite{coh1}.
\item QLT scheduling policy with Bernoulli schedule \cite{choilee}: Refers to a generalization of the QLT policy described above. In such a scenario:
\begin{displaymath}
\beta_{1}(n_{1},n_{2})=\left\{\begin{array}{ll}
1,&n_{1}> L,\\
p,&n_{1}\leq L,
\end{array}\right.\,\,
\beta_{2}(n_{1},n_{2})=\left\{\begin{array}{ll}
1,&n_{1}\leq L,\\
q:=1-p,&n_{1}>L.
\end{array}\right.
\end{displaymath}
\item Discrete-time GPS (DGPS) \cite{walr}: If both queues are non-empty, the server chooses a job from queue 1 (resp. queue 2) with probability $\beta$ (resp. $1-\beta$). If one of the queues is empty, a job from the other queue is chosen with probability 1:
\begin{displaymath}
\beta_{1}(n_{1},n_{2})=\left\{\begin{array}{ll}
\beta,&n_{1},n_{2}> 0,\\
1,&n_{1}>0,n_{2}=0,
\end{array}\right.\,\,
\beta_{2}(n_{1},n_{2})=\left\{\begin{array}{ll}
1-\beta,&n_{1},n_{2}>0,\\
1,&n_{1}=0,n_{2}>0.
\end{array}\right.
\end{displaymath}
\item Modified DGPS: If both queues are non-empty, the server chooses a job from queue 1 (resp. queue 2) with probability $\beta_{1}^{n_{2}}$ (resp. $\beta_{2}^{n_{1}}$), with $\beta_{1}^{n_{2}}+\beta_{2}^{n_{1}}=1$, i.e., the probability of choosing the next customer of either type is a decreasing function of the number of backlogged jobs of the other type. In such a scenario:
\begin{displaymath}
\beta_{1}(n_{1},n_{2})=\left\{\begin{array}{ll}
\beta_{1}^{n_{2}},&n_{1},n_{2}> 0,\\
1,&n_{1}>0,n_{2}=0,
\end{array}\right.\,\,
\beta_{2}(n_{1},n_{2})=\left\{\begin{array}{ll}
\beta_{2}^{n_{2}},&n_{1},n_{2}>0,\\
1,&n_{1}=0,n_{2}>0.
\end{array}\right.
\end{displaymath}
\end{enumerate}

Our modelling framework describes the \textit{limited} versions of the all these service disciplines based on the splitting rule \eqref{split}, and further generalize them, by allowing:
\begin{itemize}
\item an even more general form of service scheduling function $\beta_{k}(n_{1},n_{2})$.
\item the arrival and service rates to depend also on the scheme described in \eqref{split}, \eqref{trans}. 
\item combination of scheduling policies according to the splitting rule \eqref{split}. This means that in each subregion $S_{j}$, $j=a,b,c,d,$ we are able to consider different scheduling policies. Such a flexibility allows to provide a quite general service scheduling scheme.
\end{itemize} 
%
\section{An alternative approach}\label{fayo}
In the following, we focus on the case $\Psi(0,0)=0$, i.e., $p_{-1,-1}=0$, and on the solution of (\ref{fun}), which is now reduced in terms of a solution of a Riemann-Hilbert boundary value problem, by following the method in \cite{fay1}. The basic steps are briefly summarized below:
\begin{enumerate}
\item[Step 1] We provide a thorough investigation of the kernel equation $R(x,y)=0$, by identifying its roots as well as their properties (see subsection \ref{kerp}).
\item[Step 2]The functional equation (\ref{fun}), is then used to show that $g_{0}(x)$, $h_{0}(y)$ satisfy certain boundary value problems of Riemann-Hilbert type \cite{fay}, with boundary conditions on \textit{closed contours}. Information on these curves are obtained from Step 1 (see subsection \ref{kerp}, and Lemmas \ref{lemm1}, \ref{lemp1}, \ref{SQ}). Clearly, the functions $g_{0}(x)$, $h_{0}(y)$ are analytic inside the unit disc, but they might have poles in the region bounded by the unit disc and these closed curves. Thus, we have to analytically continue (Appendix \ref{defi}) $g_{0}(x)$, $h_{0}(y)$ in the whole interiors of the contours; see \cite[Chapter 3]{fay}. Similar observations are made also for the functions $g_{n_{2}}(\cdot)$, $h_{n_{1}}(\cdot)$. Equation \eqref{bou} is the boundary condition that satisfies the function $g_{0}(.)$ on such a closed contour (see subsection \ref{rhp}).
\item[Step 3] Then, we transform (through \textit{conformal mappings} \cite{ga}; see Appendix \ref{defi}) that problem from the closed contour defined in Step 1, into boundary value problem of Riemann-Hilbert type on the unit disc; see (\ref{zx}). Such a conversion is motivated by the fact that the latter problems are more usual and by far more treated in the literature \cite{coh1}.
\item[Step 4] We solve the new problem by providing an integral expression of the unknown boundary function $g_{0}(.)$; see (\ref{sool1}) in subsection \ref{rhp}. Similarly, we obtain the other unknown function $h_{0}(y)$. Then, substituting in \eqref{fun} we finally obtain $g(x,y)$ in terms of the unknown probabilities: $\pi(N_{1},n_{2})$, $n_{2} =0, 1,...,N_{2} -1$, $\pi(n_{1},N_{2})$, $n_{1} =0, 1,...,N_{1} -1$ and $\pi(N_{1},N_{2})$. A way to obtain these unknowns is given at end of subsection \ref{rhp}.
\end{enumerate}
\subsection{Kernel analysis}\label{kerp}
The kernel $R(x,y)$ is a quadratic polynomial with respect to $x$, $y$. Indeed,
\begin{displaymath}
R(x,y)=\widehat{a}(x)y^2+\widehat{b}(x)y+\widehat{c}(x)=a(y)x^{2}+b(y)x+c(y),
\end{displaymath}
where for $p_{i,j}:=p_{i,j}(N_{1},N_{2})=p_{i,j}(n_{1},n_{2})$ for $(n_{1},n_{2})\in S_{d}$,
\begin{displaymath}
\begin{array}{lcl}
\widehat{a}(x)=-(xp_{0,1}+x^{2}p_{1,1}+p_{-1,1}),&&a(y)=-(yp_{1,0}+y^{2}p_{1,1}+p_{1,-1}),\\
\widehat{b}(x)=x(1-p_{0,0})-p_{-1,0}-x^{2}p_{1,0},&&b(y)=y(1-p_{0,0})-p_{0,-1}-y^{2}p_{0,1},\\
\widehat{c}(x)=-(x^{2}p_{1,-1}+xp_{0,-1}),&&c(y)=-(y^{2}p_{-1,1}+yp_{-1,0}).
\end{array}
\end{displaymath}

In the following, we provide some technical lemmas that are necessary for the formulation of a Riemann-Hilbert boundary value problem.
\begin{lemma}\label{lemm1}
For $|y|=1$, $y\neq1$, the kernel equation $R(x,y)=0$ has exactly one root $x=X_{0}(y)$ such that $|X_{0}(y)|<1$. For $E_{x}=p_{1,0}+p_{1,1}+p_{1,-1}-p_{-1,1}-p_{-1,0}<0$ (i.e., the one-step mean jumps (the drift) from any interior point of $S_{d}$ with respect to $x-$axis), $X_{0}(1)=1$. Similarly, we can prove that $R(x,y)=0$ has exactly one root $y=Y_{0}(x)$, such that $|Y_{0}(x)|\leq1$, for $|x|=1$.
\end{lemma}
\begin{proof}
For $|y|=1$, $y\neq1$, the kernel equation $R(x,y)=0$, or equivalently $xy=\Psi(x,y)$ has exactly one root $x=X_{0}(y)$ such that $|X_{0}(y)|<1$. This is immediately proven by realizing that $|\Psi(x,y)|<1=|xy|$ and applying Rouch\'e's theorem. For $y=1$, $R(x,1)=0$ implies $(1-x)[x(p_{1,0}+p_{1,1}+p_{1,-1})-(p_{-1,1}+p_{-1,0})]$. Thus, in case $E_{x}:=p_{1,0}+p_{1,1}+p_{1,-1}-p_{-1,1}-p_{-1,0}<0$, $X_{0}(1)=1$. Similarly, we can prove that $R(x,y)=0$ has exactly one root $y=Y_{0}(x)$, $|Y_{0}(x)|\leq1$, for $|x|=1$. For an alternative derivation see \cite[Lemma 5.3.1]{fay}.\hfill$\square$
\end{proof}


The next step is to identify the location of the branch points of the two-valued function $Y_{\pm}(x)=\frac{-\widehat{b}(x)\pm\sqrt{D_{Y}(x)}}{2\widehat{a}(x)}$ (resp. $X_{\pm}(y)=\frac{-b(y)\pm\sqrt{D_{X}(y)}}{2a(y)}$) defined by $R(x,Y(x))=0$ (resp. $R(X(y),y)=0$), where $D_{Y}(x)=\widehat{b}(x)^{2}-4\widehat{a}(x)\widehat{c}(x)$ (resp. $D_{X}(y)=b(y)^{2}-4a(y)c(y)$). The branch points of $Y_{\pm}(x)$ (resp. $X_{\pm}(y)$) are the roots of $D_{Y}(x)=0$ (resp. $D_{X}(y)=0$).
\begin{lemma}\label{lemp1}
The algebraic function $Y(x)$, defined by $R(x,Y(x)) = 0$, has four real branch points, say $x_{1},x_{2},x_{3},x_{4}$, such that $x_{1},x_{2}$ are positive an lie inside the unit disc, and $x_{3},x_{4}$ lie outside the unit disc. In particular, for $x_{3}$, $x_{4}$, the following classification holds:
\begin{enumerate}
\item If $p_{1,0}>2\sqrt{p_{1,1}p_{1,-1}}$, then $x_{3}$, $x_{4}$ are positive,
\item If $p_{1,0}=2\sqrt{p_{1,1}p_{1,-1}}$, then one is positive and the other is infinite,
\item If $p_{1,0}<2\sqrt{p_{1,1}p_{1,-1}}$, then one is positive and the other is negative.
\end{enumerate}

Moreover, $D_{Y}(x)<0$, $x\in(x_{1},x_{2})\cup(x_{3},x_{4})$. 
Similar arguments hold for $X(y)$.
\end{lemma}
\begin{proof}
See Lemma 2.3.8, pp. 27-28, \cite{fay1}.\hfill$\square$
\end{proof}

To ensure the continuity of the two-valued function $Y(x)$ (resp. $X(y)$), we consider the following cut planes: $\doubletilde{C}_{x}=\mathbb{C}_{x}-([x_{1},x_{2}]\cup[x_{3},x_{4}])$, $\doubletilde{C}_{y}=\mathbb{C}_{y}-([y_{1},y_{2}]\cup[y_{3},y_{4}])$, where $\mathbb{C}_{x}$, $\mathbb{C}_{y}$ are the complex planes of $x$, $y$, respectively. For $x\in\doubletilde{C}_{x}$ (resp. $x\in\doubletilde{C}_{x}$), denote by $Y_{0}(x)$ (resp. $X_{0}(y)$) the zero of $R(x,Y(x))$ (resp. $R(X(y),y)$) with the smallest modulus, and by $Y_{1}(x)$ (resp. $X_{1}(y)$) the other one.
Define also the image contours, $\mathcal{L}=Y_{0}([\overrightarrow{\underleftarrow{x_{1},x_{2}}}])=\bar{Y}_{1}([\overleftarrow{\underrightarrow{x_{1},x_{2}}}])$, $\mathcal{M}=X_{0}([\overrightarrow{\underleftarrow{y_{1},y_{2}}}])=\bar{X}([\overleftarrow{\underrightarrow{y_{1},y_{2}}}])$, where $[\overrightarrow{\underleftarrow{u,v}}]$ stands for the contour (see Appendix \ref{defi}) traversed from $u$ to $v$ along the upper edge of the slit $[u,v]$ and then back to $u$ along the lower edge of the slit. 
The following lemma provides an exact characterization for the contours $\mathcal{L}$, $\mathcal{M}$, respectively.
\begin{lemma}\label{SQ}\begin{enumerate}\item For $y\in[y_{1},y_{2}]$, the algebraic function $X(y)$ lies on a closed contour $\mathcal{M}$, which is symmetric with respect to the real line and written as a function of $Re(x)$, i.e.,
\begin{displaymath}
\begin{array}{l}
|x|^{2}=m(Re(x)),\,|x|^{2}\leq\frac{c(y_{2})}{a(y_{2})}.
\end{array}
\end{displaymath}
Set $\beta_{0}:=\sqrt{\frac{c(y_{2})}{a(y_{2})}}$, $\beta_{1}=-\sqrt{\frac{c(y_{1})}{a(y_{1})}}$ the extreme right and left points of $\mathcal{M}$, respectively.
\item For $x\in[x_{1},x_{2}]$, the algebraic function $Y(x)$ lies on a closed contour $\mathcal{L}$, which is symmetric with respect to the real line and written as a function of $Re(y)$ as,
\begin{displaymath}
\begin{array}{l}
|y|^{2}=v(Re(y)),\,|y|^{2}\leq\frac{\widehat{c}(x_{2})}{\widehat{a}(x_{2})}.
\end{array}
\end{displaymath}
Set $\eta_{0}:=\sqrt{\frac{\widehat{c}(x_{2})}{\widehat{a}(x_{2})}}$, $\eta_{1}=-\sqrt{\frac{\widehat{c}(x_{1})}{\widehat{a}(x_{1})}}$ the extreme right and left points of $\mathcal{L}$, respectively.
\end{enumerate}
\end{lemma}
\begin{proof}
We will prove the part related to $\mathcal{L}$. Similarly, we can also prove part 2. For $x\in(x_{1},x_{2})$, $D_{x}(x)=\widehat{b}^{2}(x)-4\widehat{a}(x)\widehat{c}(x)$ is negative, so $X_{0}(y)$ and $X_{1}(y)$ are complex conjugates. Thus, $|Y(x)|^{2}=\frac{\widehat{c}(x)}{\widehat{a}(x)}=k(x)$, and together with,
\begin{equation*}
\begin{array}{c}
\frac{d}{dx}k(x)=\frac{x^{2}(p_{0,1}p_{1,-1}-p_{1,1}p_{0,-1})+2p_{1,-1}p_{-1,1}x+p_{-1,1}p_{0,-1}}{\widehat{a}(x)^{2}},
\end{array}
\end{equation*}
being a non-negative function for $x\in(0,\infty)$ implies that $k(x)\leq k(x_{2})$. We can further solve $|Y(x)|^2 = \widehat{c}(x)/\widehat{a}(x)$ as a function of $x$, and denote the solution that lies within $[x_1,x_2]$ by
\begin{equation*}
\begin{array}{l}
\tilde{x}(y)=\frac{p_{0,-1}-p_{0,1}|y|^{2}-\sqrt{(p_{0,1}|y|^{2}-p_{0,-1})^{2}-4p_{-1,1}|y|^{2}(p_{1,1}|y|^{2}-p_{1,-1})}}{2(p_{1,1}|y|^{2}-p_{1,-1})}.
\end{array}
\label{cxz}
\end{equation*}
So $\tilde{x}(y)$ is in fact the one-valued inverse function of $Y(x)$. For each $y\in\mathcal{L}$ it also follows that
\begin{equation}
\begin{array}{l}
Re(Y(x))=\frac{-\widehat{b}(\tilde{x}(y))}{2\widehat{a}(\tilde{x}(y))}.
\end{array}
\label{rd1}
\end{equation} 
Solving (\ref{rd1}) as a function of $|Y(x)|^{2}$ gives an expression for $|Y(x)|^{2}$ in terms of $Re(y)$. \hfill$\square$
\end{proof}\vspace{-0.25in}
\subsection{Formulation and solution of a Riemann-Hilbert boundary value problem}\label{rhp}
We are ready to proceed with the solution of \eqref{fun}. Note that $g(x, y)$ is well-defined for $(x, y) = (X_{0}(y), y)$ with $|y| = 1$, since $(i)$ $g(x,y)$ is well-defined for $|x|\leq 1$, $|y|\leq 1$, $(ii)$ $X_{0}(y)$ is
continuous for $|y| =1$ (note that from Lemma \ref{lemm1} it is known that $X_{0}(y)$ is analytic in $\mathbb{C}_{y} - [y_{1}, y_{2}]$ and that $0 < y_{1} < y_{2} < 1$ so that $X_{0}(y)$ is continuous for $|y| = 1$), $(iii)$ $|X_{0}(y)| \leq 1$ for $|y|=1$. Thus, the left hand side of \eqref{fun} must vanish for all pairs $(X_{0}(y), y)$:
\begin{equation}
\begin{array}{c}
A(X_{0}(y),y)g_{0}(X_{0}(y))+B(X_{0}(y),y)h_{0}(y)+C(X_{0}(y),y)=0,\,|y|=1.
\end{array}
\label{con}
\end{equation}
Let $y\in \mathcal{D}_{y}=\{y\in\mathbb{C}_{y}:|y|\leq1,|X_{0}(y)|\leq1\}$. For $y\in \mathcal{D}_{y}-[y_{1},y_{2}]$ both $g(X_{0}(y))$, $h_{0}(y)$ are analytic and the left-hand side of \eqref{con} can be analytically continued (Appendix \ref{defi}) up to the slit $[y_{1},y_{2}]$ (and similarly the functions $g_{n_{2}}(\cdot)$, $n_{2}=1,\ldots,N_{2}$, and $h_{n_{1}}(\cdot)$, $n_{1}=1,\ldots,N_{1}$). For convenience, assume that $B(X_{0}(y),y)\neq 0$ for any $y\in [y_{1},y_{2}]$. Therefore, we can divide with $B(X_{0}(y),y)$ and \eqref{con} is written as
\begin{equation}
\begin{array}{c}
\frac{A(X_{0}(y),y)}{B(X_{0}(y),y)}g_{0}(X_{0}(y))=-h_{0}(y)-\frac{C(X_{0}(y),y)}{B(X_{0}(y),y)},\,y\in[y_{1},y_{2}].
\end{array}
\label{cons}
\end{equation}
For $y\in[y_{1},y_{2}]$, $X_{0}(y)=x\in\mathcal{M}$, so that $Y_{0}(X_{0}(y))=y$. Thus, \eqref{cons} can be written as
\begin{equation}
\begin{array}{c}
\frac{A(x,Y_{0}(x))}{B(x,Y_{0}(x))}g_{0}(x)=-h_{0}(Y_{0}(x))-\frac{C(x,Y_{0}(x))}{B(x,Y_{0}(x))},\,x\in\mathcal{M}.
\end{array}
\label{con1}
\end{equation}
Note that $g_{0}(x)$ is holomorphic in $D_{x}=\{x\in\mathbb{C}:|x|<1\}$, and continuous in $\bar{D}_{x}=\{x\in\mathbb{C}:|x|\leq1\}$. However, $g_{0}(x)$ may have poles in $S_{x}=\mathcal{M}^{+}\cap\bar{D}_{x}^{c}$, where $\bar{D}_{x}^{c}=\{x\in\mathbb{C}:|x|>1\}$. These poles (if they exist) coincide with the zeros of $A(x,Y_{0}(x))$ in $S_{x}$. 

Taking into account the (possible) poles of $g_{0}(x)$ (say, $\xi_{1}$,...,$\xi_{k}$, and noticing that $h_{0}(Y_{0}(x))$ is real for $x\in\mathcal{M}$ we conclude in,
\begin{equation}
\begin{array}{c}
Re(iU(x)f(x))=w(x),\,x\in\mathcal{M},
\end{array}
\label{bou}
\end{equation}
where
\begin{displaymath}
\begin{array}{lcr}
U(x)=\frac{A(x,Y_{0}(x))}{\prod_{i=1}^{k}(x-\xi_{i})B(x,Y_{0}(x))},&f(x)=\prod_{i=1}^{k}(x-\xi_{i})g_{0}(x),&w(x)=Im(\frac{C(x,Y_{0}(x))}{B(x,Y_{0}(x))}).
\end{array}
\end{displaymath}

In order to solve (\ref{bou}), we must first conformally transform it from $\mathcal{M}$ to the unit circle $\mathcal{C}$. Let the mapping, $z=\gamma(x):\mathcal{M}^{+}\to \mathcal{C}^{+}$, and its inverse $x=\gamma_{0}(z):\mathcal{C}^{+}\to \mathcal{M}^{+}$. Then, we have the following problem: Find a function $\tilde{T}(z)=f(\gamma_{0}(z))$ regular for $z\in \mathcal{C}^{+}$, and continuous for $z\in\mathcal{C}\cup \mathcal{C}^{+}$ such that, 
\begin{equation}
Re(iU(\gamma_{0}(z))\tilde{T}(z))=w(\gamma_{0}(z)),\,z\in\mathcal{C}.
\end{equation}

To obtain the conformal mappings, we need to represent $\mathcal{M}$ in polar coordinates, i.e., $\mathcal{M}=\{x:x=\rho(\phi)\exp(i\phi),\phi\in[0,2\pi]\}.$ This procedure is described in detail in \cite{coh1}. We briefly summarized the basic steps: Since $0\in \mathcal{M}^{+}$, for each $x\in\mathcal{M}$, a relation between its absolute value and its real part is given by $|x|^{2}=m(Re(x))$ (see Lemma \ref{lemp1}). Given the angle $\phi$ of some point on $\mathcal{M}$, the real part of this point, say $\delta(\phi)$, is the solution of $\delta-\cos(\phi)\sqrt{m(\delta)}$, $\phi\in[0,2\pi].$ Since $\mathcal{M}$ is a smooth, egg-shaped contour, the solution is unique. Clearly, $\rho(\phi)=\frac{\delta(\phi)}{\cos(\phi)}$, and the parametrization of $\mathcal{M}$ in polar coordinates is fully specified. Then, the mapping from $z\in \mathcal{C}^{+}$ to $x\in \mathcal{M}^{+}$, where $z = e^{i\phi}$ and $x= \rho(\tilde{\psi}(\phi))e^{i\tilde{\psi}(\phi)}$, satisfying $\gamma_{0}(0)=0$ and $\gamma_{0}(z)=\overline{\gamma_{0}(\overline{z})}$ is uniquely determined by (see \cite{coh1}, Section I.4.4),
\begin{equation}
\begin{array}{rl}
\gamma_{0}(z)=&z\exp[\frac{1}{2\pi}\int_{0}^{2\pi}\log\{\rho(\tilde{\psi}(\omega))\}\frac{e^{i\omega}+z}{e^{i\omega}-z}d\omega],\,|z|<1,\\
\tilde{\psi}(\phi)=&\phi-\int_{0}^{2\pi}\log\{\rho(\tilde{\psi}(\omega))\}\cot(\frac{\omega-\phi}{2})d\omega,\,0\leq\phi\leq 2\pi,
\end{array}
\label{zx}
\end{equation}
i.e., $\tilde{\psi}(.)$ is uniquely determined as the solution of a Theodorsen's integral equation with $\tilde{\psi}(\phi)=2\pi-\tilde{\psi}(2\pi-\phi)$. Due to the correspondence-boundaries theorem (see Appendix \ref{defi}), $\gamma_{0}(z)$ is continuous in $\mathcal{C}\cup \mathcal{C}^{+}$. 

The solution of the boundary value problem depends on its index $\chi=\frac{-1}{\pi}[arg\{U(x)\}]_{x\in \mathcal{M}}$ (see Appendix \ref{defi}). If $\chi\leq 0$, our problem has at most one linearly independent solution. The solution of the problem defined in (\ref{bou}) is given for $z\in \mathcal{C}_{x}^{+}$ by,
\begin{equation}
\begin{array}{l}
g_{0}(\gamma_{0}(z))=\prod_{i=1}^{k}(\gamma_{0}(z)-\xi_{i})^{-1}e^{i\sigma(z)}z^{\chi}[iK+\frac{1}{2\pi i}\int_{|t|=1}e^{\omega_{1}(t)}\delta(t)\frac{t+z}{t-z}\frac{dt}{t}],
\end{array}
\label{sool1}
\end{equation}
where $K$ is a constant to be determined, $\omega_{1}(z)=Im(\sigma(z))$, $\delta(z)=\frac{w(\gamma_{0}(z))}{|U(\gamma_{0}(z))|}$, $U(\gamma_{0}(z))=b_{1}(z)+ia_{1}(z)$ and
\begin{displaymath}
\begin{array}{rl}
\sigma(z)=&\frac{1}{2\pi i}\int_{|t|=1}(\arctan\frac{b_{1}(t)}{a_{1}(t)}-\chi\arg t)\frac{t+z}{t-z}\frac{dt}{t}.
\end{array}
\end{displaymath}

Note that $g_{0}(x)=g_{0}(\gamma_{0}(\gamma(x)))$. When $\chi=0$, $K$ can be determined from the solution to $g_{0}(0)$. If $\chi<0$, then $K=0$ and a solution exists if \cite{ga}
\begin{equation}
\frac{1}{2\pi i}\int_{|t|=1}e^{\omega_{1}(t)}\delta(t)t^{-k-1}dt=0,
\label{index1}
\end{equation}
for $k=0,1,...,-\chi-1.$ Similarly, we can obtain $h_{0}(y)$, by solving another Riemann-Hilbert problem on $\mathcal{L}$. Having obtained $g_{0}(x)$ and $h_{0}(y)$, we can obtain $g(x,y)$ from \eqref{fun} in terms of the unknown probabilities $\pi(N_{1},n_{2})$, $n_{2}=0,\ldots,N_{2}-1$, $\pi(n_{1},N_{2})$, $n_{1}=0,\ldots,N_{1}-1$, $\pi(N_{1},N_{2})$, as we did in Section \ref{solution}. 

The following steps summarize the way we fully determine the stationary distribution: 
\begin{enumerate}
\item The $N_{1}\times N_{2}$ equations for $S_{a}$ involve $(N_{1}+1)\times(N_{2}+1)$ unknowns: $\pi(n_{1},n_{2})$ for $n_{1}=0,1,...,N_{1}$, $n_{2}=0,1,...,N_{2}$, excluding $\pi(N_{1},N_{2})$. Thus, we further need $N_{1}+N_{2}$ equations that involve the unknowns $\pi(N_{1},n_{2})$, $n_{2}=0,1,...,N_{2}-1$, and $\pi(n_{1},N_{2})$, $n_{1}=0,1,...,N_{1}-1$.
\item As in Section \ref{pre}, we have expressed the generating functions of equilibrium probabilities that are associated with the states of subregions $S_{b}$, $S_{c}$ in terms of $g_{0}(x)$, $h_{0}(y)$, which are obtained in terms of the solution of a Riemann-Hilbert boundary value problem. That solution obtains $g_{0}(x)$, $h_{0}(y)$ in terms of $A(x,y)$, $B(x,y)$ and $C(x,y)$. The first two are known, and the third one contains the $N_{1}+N_{2}+1$ unknown probabilities. The $N_{1}+N_{2}$ probabilities mentioned in point 1 and $\pi(N_{1},N_{2})$. These additional equations are derived from the (integral) expressions of $g_{0}(x)$, $h_{0}(y)$ as follows at steps 3 and 4.
\item The solution of \eqref{sys1}, \eqref{sys2} is similar to those in Section \ref{pre}, but now the coefficients are:
\begin{displaymath}
\begin{array}{cccc}
e_{n_{2}}(x)=\frac{q_{n_{2}}(x)}{x^{n_{2}}},&t_{n_{2}}(x)=\frac{l_{n_{2}}(x)}{x^{n_{2}}},&
\tilde{e}_{n_{1}}(y)=\frac{\tilde{q}_{n_{1}}(y)}{y^{n_{1}}},&\tilde{t}_{n_{1}}(y)=\frac{\tilde{l}_{n_{1}}(y)}{y^{n_{1}}},
\end{array}
\end{displaymath}
where $q_{n_{2}}(x)$, $l_{n_{2}}(x)$, $\tilde{q}_{n_{1}}(y)$, $\tilde{l}_{n_{1}}(y)$ are polynomials. Therefore, to derive additional equations for the unknown probabilities in point 1, we use (\ref{r2}), (\ref{r3})  and take the derivatives of $g_{0}(x)$, $h_{0}(y)$ at point 0, i.e.,
\begin{equation}
\begin{array}{rl}
n_{2}!\pi(N_{1},n_{2})=&\frac{d^{n_{2}}}{dx^{n_{2}}}[e_{n_{2}}(x)g_{0}(x)+t_{n_{2}}(x)]|_{x=0},\,n_{2}=1,...,N_{2},\vspace{2mm}\\
n_{1}!\pi(n_{1},N_{2})=&\frac{d^{n_{1}}}{dy^{n_{1}}}[\tilde{e}_{n_{1}}(y)h_{0}(y)+\tilde{t}_{n_{1}}(y)]|_{y=0},\,n_{1}=1,...,N_{1}.\end{array}\label{fc}
\end{equation}
This procedure will provide the $N_{1}+N_{2}$ equations referred at step $1$. 
\item The last additional equation for the determination of the last unknown $\pi(N_{1},N_{2})$ is done by the use of the normalization equation:
\begin{equation}
\begin{array}{c}
1=\sum_{n_{1}=0}^{N_{1}-1}\sum_{n_{2}=0}^{N_{2}-1}\pi(n_{1},n_{2})+\sum_{n_{1}=0}^{N_{1}-1}h_{n_{1}}(1)+\sum_{n_{2}=0}^{N_{2}-1}g_{n_{2}}(1)+g(1,1).
\end{array}
\label{nor}
\end{equation}
\end{enumerate}

Having obtained the equilibrium probabilities associated with subregion $S_{a}$, and the pgfs reffering to the equilibrium probabilities in subregions $S_{i}$, $i=b,c,d$, we can obtain useful performance metrics; i.e., the expression \eqref{ex1} in subsection \ref{perm} remains valid in this section. The only difference relies on the way we solved \eqref{fun}, by following an analysis in \cite{fay1}.\vspace{-0.25in}
 \section{Overview of definitions and theorems used in this paper}\label{defi}
\begin{definition}
(\cite{marku}) A function $f:\mathbb{C}\to\mathbb{C}$ is called regular at a point $z_{0}\in\mathbb{C}$ if there exists a neighbourhood of $z_{0}$ in every point of which $f$ is complex differentiable; $f$ is called regular (equivalently analytic of holomorphic) in a domain $D$ if it is regular at every point of $D$.  
\end{definition}
\begin{definition}
(\cite{marku}) A complex function $z=f(t)$ of a real variable, defined on a closed interval $a\leq t\leq b$ is said to define a \textit{(continuous) curve}. If the same point $z$ corresponds to more than one parameter value in the half-open interval $a\leq t<b$, we say that $z$ is a \textit{multiple point} of the curve $z=f(t)$, $a\leq t\leq b$. A curve with no multiple points is called a \textit{Jordan curve}. A closed curve is called a \textit{contour}. A continuous curve $\mathcal{L}$ is called a \textit{smooth curve}, if among its various parametric representations there is at least one, i.e., $z=f(t):=\gamma(t)+i\delta(t)$, such that $f(t)$ has a continuous non-vanishing derivative $f^{\prime}(t)$ at every point of the interval $[a,b]$.
\end{definition}
\begin{definition}
(\cite{ga}) Consider a function $G$, defined on a contour $\mathcal{L}$. The increment of the argument of $G(t)$ as $t$ traverses $\mathcal{L}$ once in the positive direction, divided by $2\pi$, is called the \textit{index} of $G(.)$ on $\mathcal{L}$. 
\end{definition}
\begin{definition}
(\cite{neh})
A continuous function $f$ is said to be univalent in the domain $D$ if $z_{1}\neq z_{2}$ implies $f(z_{1})\neq f(z_{2})$ for $z_{1},z_{2}\in D$. A regular function that is univalent is called a \textit{conformal mapping}. For more information see \cite{neh}, \cite[Section I.8]{marku}.
\end{definition}
\begin{definition}
Let $D$ a domain, $E\subset D$, and $f(z)$ a function defined on $E$. A function $F(z)$, which is regular in the domain $D$ and coincides with $f(z)$ on the set $E$ is called an analytic continuation of $f(z)$ into the domain $D$. Analytic continuation deals
with the problem of properly redefining an analytic function so as to extend
its domain of analyticity.
\end{definition}
\begin{theorem}
(\textit{The corresponding boundaries theorem} \cite{neh}) Denote by $\mathcal{L}^{+}$ the interior of a contour $\mathcal{L}$. If $\mathcal{L}_{1}^{+}$, $\mathcal{L}_{2}^{+}$ are two domains bounded by smooth contours, then, the conformal mapping $\omega: \mathcal{L}_{1}^{+}\to\mathcal{L}_{2}^{+}$ is continuous in $\mathcal{L}_{1}^{+}\cup\mathcal{L}_{1}$ and establishes a one-to-one correspondence among the points of $\mathcal{L}_{1}$ and $\mathcal{L}_{2}$.
\end{theorem}
\begin{theorem}
(\textit{The principle of corresponding boundaries} \cite[Section I.4.2, p. 67]{coh1}) Let $\mathcal{L}_{1}^{+}$, $\mathcal{L}_{2}^{+}$ be two domains bounded by piecewise smooth contours $\mathcal{L}_{1}$, $\mathcal{L}_{2}$. If $f(z)$ is regular in $\mathcal{L}_{1}^{+}$ and continuous in $\mathcal{L}_{1}^{+}\cup\mathcal{L}_{1}$, and maps $\mathcal{L}_{1}$ one-to-one onto $\mathcal{L}_{2}$, then $f(z)$ is univalent in $\mathcal{L}_{1}^{+}\cup\mathcal{L}_{1}$. If $f(z)$ preserves the positive directions on $\mathcal{L}_{1}$ and $\mathcal{L}_{2}$, then $f(z)$ maps $\mathcal{L}_{1}^{+}$ conformally onto $\mathcal{L}_{2}^{+}$, otherwise onto $\mathcal{L}_{2}^{-}$.
\end{theorem}\vspace{-0.3in}
\section{Proof of Theorem \ref{th1}}\label{nnrw}
It is readily seen that $\Psi(gs,gs^{-1})$ is for every fixed $|s|=1$ regular in $|g|<1$, continuous in $|g|\leq1$, and for $|g|=1$:
\begin{displaymath}
|\Psi(gs,gs^{-1})|\leq 1=|g^{2}|,
\end{displaymath}
and the proof of the first statement is a straightforward application of Rouch\'{e}'s theorem. Moreover, for $s=1$, it is readily seen that $g(1)=1$ is a zero of multiplicity one provided that $E_{x}+E_{y}<0$. Similarly, for $s=-1$, $g(-1)=-1$ is also a simple zero if $E_{x}+E_{y}<0$. The second statement follows from \eqref{fty} when we first sum in the denominator the terms for the indices $i$, $j$ for which $i+j$ and $i-j$ are both even, and then when both are odd; for more details see also \cite[Lemma 2.1, II.3.2, p. 153]{coh1}, and \cite[Theorem II.2.2.1, p. 65]{rw}.\hfill$\square$\vspace{-0.2in}
\section{Proof of Lemma \ref{lemaw}}\label{ape} For assertion 1, it is readily seen that $R(gs,gs^{-1})=0$ is written as
\begin{equation}
\begin{array}{c}
g^{2}=g^{2}p_{0,0}+g^{3}sp_{1,0}+gs^{-1}p_{-1,0}+g^{4}p_{1,1}+g^{3}s^{-1}p_{0,1}+g^{2}s^{-2}p_{-1,1}+gsp_{0,-1}+g^{2}s^{2}p_{1,-1}.
\end{array}
\label{hel}
\end{equation} 
It is clear that $g:=g(s)=0$ is a zero of $R(gs,gs^{-1})$. For $|s|=1$, $0<|g|\leq 1$, \eqref{hel} is reduced to $m_{1}(g)=m_{2}(g,s)$, where $m_{1}(g)=g$, $m_{2}(g,s)=gp_{0,0}+g^{2}sp_{1,0}+s^{-1}p_{-1,0}+g^{3}p_{1,1}+g^{2}s^{-1}p_{0,1}+gs^{-2}p_{-1,1}+sp_{0,-1}+gs^{2}p_{1,-1}.$ We now show that if $E_{x}+E_{y}<0$ and for fixed $s$ with $|s|=1$, $s\neq\pm 1$, $m_{1}(g)=m_{2}(g,s)$ has a unique root in $|g|\leq 1$. 

Clearly, $m_{1}(g)$ has a single zero in the complex unit disk, and we wish to establish that $|m_{1}(g)|>|m_{2}(g,s)|$ for $|g|=1$, $|s|=1$, $s\neq\pm1$. Note that $|m_{1}(g)|=|g|=m_{1}(|g|)$, and $|m_{2}(g,s)|\leq m_{2}(|g|,1)$. Thus, we require $m_{1}(|g|)>m_{2}(|g|,1)$ for $|g|=1$. However, for $|g|=1$, $m_{1}(|g|)=m_{1}(1)=1=m_{2}(1,1)=m_{2}(|g|,1)$. To resolve this issue, we evaluate $m_{1}(|g|)$, $m_{2}(|g|,1)$ on the circle $|g|=1+\epsilon$ with $\epsilon$ small and positive. To accomplish
this task, we use the Taylor expansion $m_{2}(1+\epsilon,1)=m_{2}(1,1)+\epsilon m_{2}^{\prime}(1,1)+o(\epsilon)$ (where $m_{2}^{\prime}(1,1)=\frac{d}{d|g|}m_{2}(|g|,1)|_{|g|=1}$), and similarly for $m_{1}(1+\epsilon)$. So we need to show that $m_{1}(1+\epsilon)>m_{2}(1+\epsilon,1)$. Since, $m_{1}(1)=m_{2}(1,1)$ we are left to show that $m_{1}^{\prime}(1)>m_{2}^{\prime}(1,1)$, or equivalently 
\begin{eqnarray}
p_{-1,0}+p_{0,-1}>p_{1,0}+2p_{1,1}+p_{0,1}.\label{hel2}
\end{eqnarray}
It is easy to see that \eqref{hel2} is related to the mean drifts in region $S_{d}$, and it is always true as $E_{x}+E_{y}<0$. So, for sufficiently small $\epsilon>0$ we
have that $m_{1}(|g|)>m_{2}(|g|,1)$ for $|g|\in(1,1+\epsilon]$, which proves assertion 1. If $s=1$, then $g(1)=1$ is the only root of $m_{1}(g)=m_{2}(g,1)$ in $|g|\leq 1$, and that $m_{2}^{\prime}(1,1)<1$, i.e., that $E_{x}+E_{y}<0$, implies that this is a simple root. Similarly, for $s=-1$, $g(-1)=-1$ is the unique root of $m_{1}(g)=m_{2}(g,1)$ in $|g|\leq 1$ with multiplicity 1 if $E_{x}+E_{y}<0$. Equation \eqref{cas} follows directly from \eqref{fty}. Assertion 2 also follows from \eqref{cas}; see also \cite[Lemma 10.1, II.3.10]{coh1}.\hfill$\square$\vspace{-0.2in}
\section{On the induced Markov chains in subsection \ref{appl1}}\label{apc}
For $Q_{1,n}>N_{1}$, $Q_{2,n}$ evolves as a one-dimensional RW with one step transition probabilities $w^{(2)}_{j}(n_{2})=P(Q_{2,n+1}=n_{2}+j|Q_{2,n}=n_{2})$, $j=0,\pm1$, for $n_{2}=0,1,\ldots,$ given by
\begin{displaymath}
\begin{array}{rl}
w^{(2)}_{1}(n_{2})=&p_{0,1}(N_{1},n_{2})+p_{1,1}(N_{1},n_{2})+p_{-1,1}(N_{1},n_{2})=\lambda_{2}(N_{1},n_{2})[1-\bar{a}_{1}(N_{1},n_{2})a_{2}(N_{1},n_{2})],\\
w^{(2)}_{-1}(n_{2})=&p_{1,-1}(N_{1},n_{2})+p_{0,-1}(N_{1},n_{2})=\bar{\lambda}_{2}(N_{1},n_{2})\bar{a}_{1}(N_{1},n_{2})a_{2}(N_{1},n_{2}),\\
w^{(2)}_{0}(n_{2})=&p_{1,0}(N_{1},n_{2})+p_{-1,1}(N_{1},n_{2})+p_{0,0}(N_{1},n_{2}).
\end{array}
\end{displaymath}
Note that for $n_{2}\geq N_{2}$, $w^{(2)}_{j}(n_{2}):=w^{(2)}_{j}$, $j=0,\pm1$. Recall $\psi:=(\psi_{0},\psi_{1},\ldots)$ its stationary distribution. Then, simple calculations yield
\begin{displaymath}
\begin{array}{rl}
\psi_{n_{2}}=&\psi_{0}\prod_{j=0}^{n_{2}-1}\frac{w^{(2)}_{1}(j)}{w^{(2)}_{-1}(j+1)},\,1\leq j\leq N_{2},\\
\psi_{n_{2}}=&\psi_{N_{2}}\left(\frac{w^{(2)}_{1}}{w^{(2)}_{-1}}\right)^{n_{2}-N_{2}},\,j\geq N_{2}+1,\\
\psi_{0}=&[1+\sum_{n_{2}=1}^{N_{2}-1}\prod_{j=0}^{n_{2}-1}\frac{w^{(2)}_{1}(j)}{w^{(2)}_{-1}(j+1)}+\frac{w^{(2)}_{-1}}{w^{(2)}_{-1}-w^{(2)}(1)}\prod_{j=0}^{N_{2}-1}\frac{w^{(2)}_{1}(j)}{w^{(2)}_{-1}(j+1)}]^{-1},
\end{array}
\end{displaymath} 
where $w^{(2)}_{-1}-w^{(2)}(1)=\bar{a}_{1}a_{2}-\lambda_{2}<0$. Similar argumentation follows for the component $Q_{1,n}$, which evolves as a one-dimensional RW for $Q_{2,n}>N_{2}$.\vspace{-0.2in}
\section{One-step transition probabilities for the model in subsection \ref{appl2}}\label{onestep}
The one-step transition probabilities for the model in subsection \ref{appl2} are given below:
\begin{displaymath}
\begin{array}{rl}
p_{-2,0}(\underline{n})=&\theta_{1}(n_{1})\bar{\theta}_{2}(n_{2})\beta_{1}(\underline{n}-\underline{1}_{1})\mu_{1}(n_{1}-1)d_{0,0}(\underline{n})\\&+\theta_{1}(n_{1})\theta_{2}(n_{2})\beta_{1}(\underline{n}-\underline{1}_{1}-\underline{1}_{2})\mu_{1}(n_{1}-1)d_{0,1}(\underline{n}),\vspace{2mm}\\
p_{0,-2}(\underline{n})=&\theta_{2}(n_{2})\bar{\theta}_{1}(n_{1})\beta_{2}(\underline{n}-\underline{1}_{2})\mu_{2}(n_{2}-1)d_{0,0}(\underline{n})\\&+\theta_{1}(n_{1})\theta_{2}(n_{2})\beta_{2}(\underline{n}-\underline{1}_{1}-\underline{1}_{2})\mu_{2}(n_{2}-1)d_{1,0}(\underline{n}),\vspace{2mm}\\
p_{-2,1}(\underline{n})=&\theta_{1}(n_{1})\bar{\theta}_{2}(n_{2})\beta_{1}(\underline{n}-\underline{1}_{1})\mu_{1}(n_{1}-1)d_{0,1}(\underline{n}),\vspace{2mm}\\
p_{1,-2}(\underline{n})=&\theta_{2}(n_{2})\bar{\theta}_{1}(n_{1})\beta_{2}(\underline{n}-\underline{1}_{2})\mu_{2}(n_{2}-1)d_{1,0}(\underline{n}),\vspace{2mm}\\
p_{-2,-1}(\underline{n})=&\theta_{1}(n_{1})\theta_{2}(n_{2})\beta_{1}(\underline{n}-\underline{1}_{1}-\underline{1}_{2})\mu_{1}(n_{1}-1)d_{0,0}(\underline{n}),\vspace{2mm}\\
p_{-1,-2}(\underline{n})=&\theta_{1}(n_{1})\theta_{2}(n_{2})\beta_{2}(\underline{n}-\underline{1}_{1}-\underline{1}_{2})\mu_{2}(n_{2}-1)d_{0,0}(\underline{n}),\vspace{2mm}\\
p_{1,-1}(\underline{n})=&\theta_{2}(n_{2})\bar{\theta}_{1}(n_{1})[\beta_{2}(\underline{n}-\underline{1}_{2})\mu_{2}(n_{2}-1)d_{1,1}(\underline{n})+(\beta_{1}(\underline{n}-\underline{1}_{2})\bar{\mu_{1}}(n_{1})\\&+\beta_{2}(\underline{n}-\underline{1}_{2})\bar{\mu}_{2}(n_{2}-1))d_{1,0}(\underline{n})]
+\bar{\theta}_{1}(n_{1})\bar{\theta}_{2}(n_{2})\beta_{2}(\underline{n})\mu_{2}(n_{2})d_{1,0}(\underline{n}),\vspace{2mm}\\
p_{1,1}(\underline{n})=&\bar{\theta}_{1}(n_{1})\bar{\theta}_{2}(n_{2})(\beta_{1}(\underline{n})\bar{\mu_{1}}(n_{1})+\beta_{2}(\underline{n})\bar{\mu}_{2}(n_{2}))d_{1,1}(\underline{n}),\vspace{2mm}\\
p_{-1,-1}(\underline{n})=&\theta_{2}(n_{2})\theta_{1}(n_{1})\left[\beta_{1}(\underline{n}-\underline{1}_{1}-\underline{1}_{2})\mu_{1}(n_{1}-1)d_{1,0}(\underline{n})+\beta_{2}(\underline{n}-\underline{1}_{1}-\underline{1}_{2})\mu_{2}(n_{2}-1)d_{0,1}(\underline{n})\right.\\
&\left.+(\beta_{1}(\underline{n}-\underline{1}_{1}-\underline{1}_{2})\bar{\mu_{1}}(n_{1}-1)+\beta_{2}(\underline{n}-\underline{1}_{1}-\underline{1}_{2})\bar{\mu}_{2}(n_{2})-1)d_{0,0}(\underline{n})\right]\\
&+\theta_{1}(n_{1})\bar{\theta}_{2}(n_{2})\beta_{2}(\underline{n}-\underline{1}_{1})\mu_{2}(n_{2})d_{0,0}(\underline{n})+\theta_{2}(n_{2})\bar{\theta}_{1}(n_{1})\beta_{1}(\underline{n}-\underline{1}_{2})\mu_{1}(n_{1})d_{0,0}(\underline{n}),\vspace{2mm}\\
p_{-1,0}(\underline{n})=&\theta_{1}(n_{1})\bar{\theta}_{2}(n_{2})\left[\beta_{1}(\underline{n}-\underline{1}_{1})\mu_{1}(n_{1}-1)d_{1,0}(\underline{n})+\beta_{2}(\underline{n}-\underline{1}_{1})\mu_{2}(n_{2})d_{0,1}(\underline{n})\right.\\
&\left.+(\beta_{1}(\underline{n}-\underline{1}_{1})\bar{\mu_{1}}(n_{1}-1)+\beta_{2}(\underline{n}-\underline{1}_{1})\mu_{2}(n_{2}))d_{0,0}(\underline{n})\right]\\
&+\bar{\theta_{1}}(n_{1})\bar{\theta}_{2}(n_{1})\beta_{1}(\underline{n})\mu_{1}(n_{1})d_{0,0}(\underline{n})+\theta_{2}(n_{2})\bar{\theta}_{1}(n_{1})\beta_{1}(\underline{n})\mu_{1}(n_{1})d_{0,1}(\underline{n})\\
&+\theta_{1}(n_{1})\theta_{2}(n_{2})[\beta_{1}(\underline{n}-\underline{1}_{1}-\underline{1}_{2})(\mu_{1}(n_{1}-1)d_{1,1}(\underline{n})+\bar{\mu}_{1}(n_{1}-1)d_{0,1}(\underline{n}))\\&+\beta_{2}(\underline{n}-\underline{1}_{1}-\underline{1}_{2})\bar{\mu}_{2}(n_{2}-1)d_{0,1}(\underline{n})],\vspace{2mm}
\end{array}
\end{displaymath}
\begin{displaymath}
\begin{array}{rl}
p_{1,0}(\underline{n})=&\theta_{2}(n_{2})\bar{\theta}_{1}(n_{1})(\beta_{1}(\underline{n}-\underline{1}_{2})\bar{\mu_{1}}(n_{1})+\beta_{2}(\underline{n}-\underline{1}_{2})\bar{\mu}_{2}(n_{2}-1))d_{1,1}(\underline{n})\\
&+\bar{\theta}_{1}(n_{1})\bar{\theta}_{2}(n_{2})[\beta_{2}(\underline{n})\mu_{2}(n_{2})d_{1,1}(\underline{n})+(\beta_{1}(\underline{n})\bar{\mu_{1}}(n_{1})+\beta_{2}(\underline{n})\bar{\mu}_{2}(n_{2}))d_{1,0}(\underline{n})],\vspace{2mm}\\
p_{-1,1}(\underline{n})=&\theta_{1}(n_{1})\bar{\theta}_{2}(n_{2})[\beta_{1}(\underline{n}-\underline{1}_{1})\mu_{1}(n_{1}-1)d_{1,1}(\underline{n})+(\beta_{1}(\underline{n}-\underline{1}_{1})\bar{\mu_{1}}(n_{1}-1)\\&+\beta_{2}(\underline{n}-\underline{1}_{1})\bar{\mu}_{2}(n_{2}))d_{0,1}(\underline{n})]+\bar{\theta}_{1}(n_{1})\bar{\theta}_{2}(n_{2})\beta_{1}(\underline{n})\mu_{1}(n_{1})d_{0,1}(\underline{n}),\vspace{2mm}\\\\
p_{0,-1}(\underline{n})=&\theta_{2}(n_{2})\bar{\theta}_{1}(n_{1})\left[\beta_{1}(\underline{n}-\underline{1}_{2})\mu_{1}(n_{1})d_{1,0}(\underline{n})+\beta_{2}(\underline{n}-\underline{1}_{2})\mu_{2}(n_{2}-1)d_{0,1}(\underline{n})\right.\\
&\left.+(\beta_{1}(\underline{n}-\underline{1}_{2})\bar{\mu_{1}}(n_{1})+\beta_{2}(\underline{n}-\underline{1}_{2})\mu_{2}(n_{2}-1))d_{0,0}(\underline{n})\right]\\
&+\bar{\theta_{1}}(n_{1})\bar{\theta}_{2}(n_{2})\beta_{2}(\underline{n})\mu_{2}(n_{2})d_{0,0}(\underline{n})+\theta_{1}(n_{1})\bar{\theta}_{2}(n_{2})\beta_{2}(\underline{n}-\underline{1}_{1})\mu_{2}(n_{2})d_{1,0}(\underline{n})\\
&+\theta_{1}(n_{1})\theta_{2}(n_{2})[\beta_{2}(\underline{n}-\underline{1}_{1}-\underline{1}_{2})(\mu_{2}(n_{2}-1)d_{1,1}(\underline{n})+\bar{\mu}_{2}(n_{2}-1)d_{1,0}(\underline{n}))\\&+\beta_{1}(\underline{n}-\underline{1}_{1}-\underline{1}_{2})\bar{\mu}_{1}(n_{1}-1)d_{1,0}(\underline{n})],\vspace{2mm}\\
p_{0,1}(\underline{n})=&\theta_{1}(n_{1})\bar{\theta}_{2}(n_{2})(\beta_{1}(\underline{n}-\underline{1}_{1})\bar{\mu_{1}}(n_{1}-1)+\beta_{2}(\underline{n}-\underline{1}_{1})\bar{\mu}_{2}(n_{2}))d_{1,1}(\underline{n})\\
&+\bar{\theta}_{1}(n_{1})\bar{\theta}_{2}(n_{2})[\beta_{1}(\underline{n})\mu_{1}(n_{1})d_{1,1}(\underline{n})+(\beta_{1}(\underline{n})\bar{\mu_{1}}(n_{1})+\beta_{2}(\underline{n})\bar{\mu}_{2}(n_{2}))d_{0,1}(\underline{n})],\vspace{2mm}\\
p_{0,0}(\underline{n})=&\theta_{1}(n_{1})\theta_{2}(n_{2})(\beta_{1}(\underline{n}-\underline{1}_{1}-\underline{1}_{2})\bar{\mu_{1}}(n_{1}-1)+\beta_{2}(\underline{n}-\underline{1}_{1}-\underline{1}_{2})\bar{\mu}_{2}(n_{2}-1))d_{1,1}(\underline{n})\\
&+\theta_{1}(n_{1})\bar{\theta}_{2}(n_{2})[\beta_{2}(\underline{n}-\underline{1}_{1})\mu_{2}(n_{2})d_{1,1}(\underline{n})+(\beta_{1}(\underline{n}-\underline{1}_{1})\bar{\mu_{1}}(n_{1}-1)\\&+\beta_{2}(\underline{n}-\underline{1}_{1})\bar{\mu}_{2}(n_{2}))d_{1,0}(\underline{n})]\\
&+\theta_{2}(n_{2})\bar{\theta}_{2}(n_{2})[\beta_{1}(\underline{n}-\underline{1}_{2})\mu_{1}(n_{1})d_{1,1}(\underline{n})+(\beta_{1}(\underline{n}-\underline{1}_{2})\bar{\mu_{1}}(n_{1})\\&+\beta_{2}(\underline{n}-\underline{1}_{2})\bar{\mu}_{2}(n_{2}-1))d_{0,1}(\underline{n})]\\
&+\bar{\theta}_{1}(n_{1})\bar{\theta}_{2}(n_{2})[\beta_{1}(\underline{n})\mu_{1}(n_{1})d_{1,0}(\underline{n})+\beta_{2}(\underline{n})\mu_{2}(n_{2})d_{0,1}(\underline{n})\\&+(\beta_{1}(\underline{n})\bar{\mu_{1}}(n_{1})+\beta_{2}(\underline{n})\bar{\mu}_{2}(n_{2}))d_{0,0}(\underline{n})],
\end{array}
\end{displaymath}
where $d_{i,j}(\underline{n})$, $i,j=0,1$ are as in \eqref{hjj}. \vspace{-0.25in}
\bibliographystyle{spmpsci}      
\bibliography{paper_arxiv_version3}

\begin{thebibliography}{10}
\providecommand{\url}[1]{{#1}}
\providecommand{\urlprefix}{URL }
\expandafter\ifx\csname urlstyle\endcsname\relax
  \providecommand{\doi}[1]{DOI~\discretionary{}{}{}#1}\else
  \providecommand{\doi}{DOI~\discretionary{}{}{}\begingroup
  \urlstyle{rm}\Url}\fi

\bibitem{Adan2016}
Adan, I.J.B.F., Boxma, O.J., Kapodistria, S., Kulkarni, V.G.: The shorter queue
  polling model.
\newblock Annals of Operations Research \textbf{241}(1), 167--200 (2016)

\bibitem{ad1}
Adan, I.J.B.F., Kapodistria, S., van Leeuwaarden, J.S.H.: Erlang arrivals
  joining the shorter queue.
\newblock Queueing Systems \textbf{74}(2-3), 273--302 (2013)

\bibitem{adRL}
Adan, I.J.B.F., van Leeuwaarden, J.S.H., Raschel, K.: The compensation approach
  for walks with small steps in the quarter plane.
\newblock Combinatorics, Probability and Computing \textbf{22}(2), 161–183
  (2013)

\bibitem{ad2}
Adan, I.J.B.F., Wessels, J., Zijm, W.H.M.: Analysis of the asymmetric shortest
  queue problem.
\newblock Queueing Systems \textbf{8}(1), 1--58 (1991)

\bibitem{ad3}
Adan, I.J.B.F., Wessels, J., Zijm, W.H.M.: A compensation approach for
  two-dimentional {Markov} processes.
\newblock Advances in Applied Probability \textbf{25}(4), 783--817 (1993)

\bibitem{ave}
Aveklouris, A., Vlasiou, M., Zhang, J., Zwart, A.: Heavy-traffic approximations
  for a layered network with limited resources.
\newblock Probability and Mathematical Statistics \textbf{37}(2), 497--532
  (2017)

\bibitem{bonald}
Bonald, T., Borst, S.C., Hegde, N., Jonckheere, M., Prouti{\`{e}}re, A.:
  Flow-level performance and capacity of wireless networks with user mobility.
\newblock Queueing Syst. \textbf{63}(1-4), 131--164 (2009)

\bibitem{bonald2}
Bonald, T., Prouti\`{e}re, A.: Wireless downlink data channels: User
  performance and cell dimensioning.
\newblock In: Proceedings of the 9th Annual International Conference on Mobile
  Computing and Networking, MobiCom ’03, p. 339–352. ACM, NY, USA (2003)

\bibitem{boro}
Borovkov, A.: Ergodicity and Stability of Stochastic Processes.
\newblock Wiley Series in Probability and Statistics. Wiley (1998)

\bibitem{bor2}
Borst, S.: User-level performance of channel-aware scheduling algorithms in
  wireless data networks.
\newblock IEEE/ACM Trans. Netw. \textbf{13}(3), 636–647 (2005)

\bibitem{bor1}
{Borst}, S., {Hegde}, N., {Proutiere}, A.: Mobility-driven scheduling in
  wireless networks.
\newblock In: IEEE INFOCOM 2009, pp. 1260--1268 (2009)

\bibitem{borst2018}
Borst, S., den Hollander, F., Nardi, F.R., Sfragara, M.: Transition time
  asymptotics of queue-based activation protocols in random-access networks
  (2018)

\bibitem{bouman1}
Bouman, N., Borst, S., {Leeuwaarden, van}, J.: Delay performance in
  random-access networks.
\newblock Queueing Syst. \textbf{77}(2), 211--242 (2014)

\bibitem{BM1}
Bousquet-Mélou, M., Mishna, M.: Walks with small steps in the quarter plane.
\newblock Contemporary mathematics, American Mathematical Society \textbf{520},
  1--40 (2010)

\bibitem{boxmadown}
Boxma, O., Down, D.: Dynamic server assignment in a two-queue model.
\newblock European Journal of Operational Research \textbf{103}(3), 595 -- 609
  (1997)

\bibitem{boyer}
Boyer, J., Guillemin, F., Robert, P., Zwart, B.: Heavy tailed {M/G/1-PS} queues
  with impatience and admission control in packet networks.
\newblock In: Proceedings {IEEE} {INFOCOM} 2003, San Franciso, CA, USA, March
  30 - April 3, 2003, pp. 186--195. {IEEE} Computer Society (2003)

\bibitem{carofaglioetal}
{Carofiglio}, G., {Muscariello}, L.: On the impact of tcp and per-flow
  scheduling on internet performance.
\newblock In: 2010 Proceedings IEEE INFOCOM, pp. 1--9 (2010)

\bibitem{castiel}
Castiel, E., Borst, S., Miclo, L., Simatos, F., Whiting, P.: Induced idleness
  leads to deterministic heavy traffic limits for queue-based random-access
  algorithms (2019)

\bibitem{choikimlee}
Choi, D.I., Kim, T.S., Lee, S.: Analysis of a queueing system with a general
  service scheduling function, with applications to telecommunication network
  traffic control.
\newblock European Journal of Operational Research \textbf{178}(2), 463 -- 471
  (2007)

\bibitem{choilee}
{Choi}, D.I., {Lee}, Y.: Performance analysis of a dynamic priority queue for
  traffic control of bursty traffic in atm networks.
\newblock IEE Proceedings - Communications \textbf{148}(3), 181--187 (2001)

\bibitem{coh2}
Cohen, J.: Boundary value problems in queueing theory.
\newblock Queueing Syst. \textbf{3}, 97--128 (1988)

\bibitem{c3}
Cohen, J.: Analysis of the asymmetrical shortest two-server queueing model.
\newblock Journal of Applied Mathematics and Stochastic Analysis
  \textbf{11}(2), 115--162 (1998)

\bibitem{coh1}
Cohen, J., Boxma, O.: Boundary value problems in queueing systems analysis.
\newblock North Holland Publishing Company, Amsterdam, Netherlands (1983)

\bibitem{cohenlong}
Cohen, J.W.: A two-queue, one-server model with priority for the longer queue.
\newblock Queueing Syst. Theory Appl. \textbf{2}(3), 261–283 (1987)

\bibitem{rw}
Cohen, J.W.: Analysis of random walks.
\newblock IOS Press (Amsterdam) (1992)

\bibitem{c2}
Cohen, J.W.: On a class of two-dimensional nearest-neighbour random walks.
\newblock Journal of Applied Probability \textbf{31}(A), 207–237 (1994)

\bibitem{c1}
Cohen, J.W.: On the asymmetric clocked buffered switch.
\newblock Queueing Syst. \textbf{30}(3/4), 385--404 (1998)

\bibitem{cont}
Cont, R., de~Larrard, A.: Price dynamics in a markovian limit order market.
\newblock SIAM Journal on Financial Mathematics \textbf{4}(1), 1--25 (2013)

\bibitem{dimpaptwc}
Dimitriou, I., Pappas, N.: Stable throughput and delay analysis of a random
  access network with queue-aware transmission.
\newblock IEEE Transactions on Wireless Communications \textbf{17}(5),
  3170--3184 (2018)

\bibitem{ephr}
{Ephremides}, A., {Hajek}, B.: Information theory and communication networks:
  an unconsummated union.
\newblock IEEE Transactions on Information Theory \textbf{44}(6), 2416--2434
  (1998)

\bibitem{faysta}
Fayolle, G.: On random walks arising in queueing systems: ergodicity and
  transience via quadratic forms as {L}yapounov functions--{P}art {I}.
\newblock Queueing Syst. \textbf{5}(1), 167--183 (1989)

\bibitem{fay2}
Fayolle, G., Iasnogorodski, R.: Two coupled processors: The reduction to a
  {R}iemann-{H}ilbert problem.
\newblock Zeitschrift f{\"u}r Wahrscheinlichkeitstheorie und Verwandte Gebiete
  \textbf{47}(3), 325--351 (1979)

\bibitem{fay1}
Fayolle, G., Iasnogorodski, R., Malyshev, V.: Random walks in the
  quarter-plane: {A}lgebraic methods, boundary value problems, applications to
  queueing systems and analytic combinatorics.
\newblock Springer-Verlag, Berlin (2017)

\bibitem{fayo}
Fayolle, G., King, P.J.B., Mitrani, I.: The solution of certain two-dimensional
  markov models.
\newblock Advances in Applied Probability \textbf{14}(2), 295--308 (1982)

\bibitem{fay}
Fayolle, G., Malyshev, V.A., Menshikov, M.: Topics in the constructive theory
  of countable {M}arkov chains.
\newblock Cambridge university press (1995)

\bibitem{feng}
Feng, W., Kowada, M., Adachi, K.: A two-queue model with bernoulli service
  schedule and switching times.
\newblock Queueing Syst. \textbf{30}(3-4), 405--434 (1998)

\bibitem{flat}
Flatto, L., McKean, H.P.: Two queues in parallel.
\newblock Communications on Pure and Applied Mathematics \textbf{30}(2),
  255--263 (1977)

\bibitem{ga}
Gakhov, F.: Boundary value problems.
\newblock Pergamon Press, Oxford (1966)

\bibitem{gha}
Ghaderi, J., Borst, S., Whiting, P.: Queue-based random-access algorithms:
  Fluid limits and stability issues.
\newblock Stoch. Syst. \textbf{4}(1), 81--156 (2014)

\bibitem{gromoll}
Gromoll, H.C., Robert, P., Zwart, B.: Fluid limits for processor-sharing queues
  with impatience.
\newblock Mathematics of Operations Research \textbf{33}(2), 375--402 (2008)

\bibitem{guu}
{Guillemin}, F., {Elayoubi}, S.E., {Robert}, P., {Fricker}, C., {Sericola}, B.:
  Controlling impatience in cellular networks using qoe-aware radio resource
  allocation.
\newblock In: 2015 27th International Teletraffic Congress, pp. 159--167 (2015)

\bibitem{guileu}
Guillemin, F., van Leeuwaarden, J.S.H.: Rare event asymptotics for a random
  walk in the quarter plane.
\newblock Queueing Systems \textbf{67}(1), 1--32 (2011)

\bibitem{guilzwart}
Guillemin, F., Robert, P., Zwart, B.: Tail asymptotics for processor-sharing
  queues.
\newblock Advances in Applied Probability \textbf{36}(2), 525–543 (2004)

\bibitem{guillemin1}
Guillemin, F., Simonian, A.: Stationary analysis of the shortest queue first
  service policy.
\newblock Queueing Syst. \textbf{77}(4), 393–426 (2014)

\bibitem{gu}
{Gupta}, P., {Stolyar}, A.L.: Optimal throughput allocation in general
  random-access networks.
\newblock In: 2006 40th Annual Conference on Information Sciences and Systems,
  pp. 1254--1259 (2006)

\bibitem{guptamor}
Gupta, V., Harchol-Balter, M.: Self-adaptive admission control policies for
  resource-sharing systems.
\newblock SIGMETRICS Perform. Eval. Rev. \textbf{37}(1), 311–322 (2009)

\bibitem{king}
Kingman, J.F.C.: Two similar queues in parallel.
\newblock Ann. Math. Statist. \textbf{32}(4), 1314--1323 (1961)

\bibitem{knessl}
Knessl, C., Tier, C., Il~Choi, D.: A dynamic priority queue model for
  simultaneous service of two traffic types.
\newblock SIAM Journal on Applied Mathematics \textbf{63}(2), 398--422 (2003)

\bibitem{dslee}
Lee, D.S.: Analysis of a two-queue model with bernoulli schedules.
\newblock Journal of Applied Probability \textbf{34}(1), 176–191 (1997)

\bibitem{leepatent}
Lee, D.S.: Generalized longest queue first: An adaptive scheduling discipline
  for atm networks.
\newblock In: Proceedings of the INFOCOM ’97., p. 318. IEEE Computer Society,
  USA (1997)

\bibitem{lizha1}
Li, H., Zhao, Y.Q.: Exact tail asymptotics in a priority
  queue---characterizations of the preemptive model.
\newblock Queueing Systems \textbf{63}(1), 355 (2009)

\bibitem{lizha2}
Li, H., Zhao, Y.Q.: Exact tail asymptotics in a priority
  queue---characterizations of the non-preemptive model.
\newblock Queueing Syst. \textbf{68}(2), 165--192 (2011)

\bibitem{lizha3}
Li, H., Zhao, Y.Q.: Tail asymptotics for a generalized two-demand queueing
  model---a kernel method.
\newblock Queueing Syst. \textbf{69}(1), 77--100 (2011)

\bibitem{LuoAE1999}
Luo, W., Ephremides, A.: Stability of {N} interacting queues in random-access
  systems \textbf{45}(5), 1579--1587 (1999)

\bibitem{maly1}
Malyshev, V.A.: The classification of two-dimensional positive random walks and
  almost linear semi-martingales.
\newblock Dokl. Akad. Nauk SSSR \textbf{202}, 526--528 (1972)

\bibitem{maly3}
{Malyshev}, V.A.: {An analytical method in the theory of two-dimensional
  positive random walks.}
\newblock {Sib. Math. J.} \textbf{13}, 917--929 (1973)

\bibitem{malmen}
Malyshev, V.A., Menshikov, M.V.: Ergodicity, continuity and analyticity of
  countable {M}arkov chains.
\newblock Trudy Moskov. Mat. Obshch. \textbf{39}, 3--48, 235 (1979)

\bibitem{marku}
Markushevich, A.: Theory of Functions of a Complex Variable, Vol. I.
\newblock Prentice-Hall, New York (1965)

\bibitem{miya}
Miyazawa, M.: Light tail asymptotics in multidimensional reflecting processes
  for queueing networks.
\newblock TOP \textbf{19}(2), 233--299 (2011)

\bibitem{neh}
Nehari, Z.: Conformal mapping.
\newblock McGraw-Hill, New York (1952)

\bibitem{metis}
{Osseiran}, A., {Boccardi}, F., {Braun}, V., {Kusume}, K., {Marsch}, P.,
  {Maternia}, M., {Queseth}, O., {Schellmann}, M., {Schotten}, H., {Taoka}, H.,
  {Tullberg}, H., {Uusitalo}, M.A., {Timus}, B., {Fallgren}, M.: Scenarios for
  5g mobile and wireless communications: the vision of the metis project.
\newblock IEEE Communications Magazine \textbf{52}(5), 26--35 (2014)

\bibitem{oza}
Ozawa, T.: Asymptotics for the stationary distribution in a discrete-time
  two-dimensional quasi-birth-and-death process.
\newblock Queueing Syst. \textbf{74}(2), 109--149 (2013)

\bibitem{PappasTWC2015}
Pappas, N., Kountouris, M., Ephremides, A., Traganitis, A.: Relay-assisted
  multiple access with full-duplex multi-packet reception.
\newblock IEEE Transactions on Wireless Communications \textbf{14}(7),
  3544--3558 (2015)

\bibitem{Rao_TIT1988}
Rao, R., Ephremides, A.: On the stability of interacting queues in a
  multiple-access system.
\newblock IEEE Transactions on Information Theory \textbf{34}(5), 918--930
  (1988)

\bibitem{rasc}
Raschel, K.: Counting walks in a quadrant: a unified approach via boundary
  value problems.
\newblock J. Eur. Math. Soc. \textbf{14}, 749--777 (2012)

\bibitem{sauta}
Sauter, M.: From GSM to LTE-Advanced Pro and 5G: An Introduction to Mobile
  Networks and Mobile Broadband.
\newblock Wiley \& Sons, Chichester, UK (2017)

\bibitem{shn2}
Shneer, S., Stolyar, A.: {Stability conditions for a decentralised medium
  access algorithm: single- and multi-hop networks}.
\newblock Queueing Syst. \textbf{94}(1), 109--128 (2020)

\bibitem{sima}
Simatos, F., Simonian, A.: Mobility can drastically improve the heavy traffic
  performance from {$1/(1-\rho)$} to {$\log(1/(1-\rho))$}.
\newblock Queueing Syst. \textbf{95}(1), 1--28 (2020)

\bibitem{stol2}
Stolyar, A.L.: Dynamic distributed scheduling in random access networks.
\newblock Journal of Applied Probability \textbf{45}(2), 297--313 (2008)

\bibitem{takagi}
Takagi, H.: Queueing analysis : Vacation and Priority Systems, vol. 1.
\newblock North-Holland, Amsterdam (1991)

\bibitem{telek}
Telek, M., Van~Houdt, B.: Response time distribution of a class of limited
  processor sharing queues.
\newblock SIGMETRICS Perform. Eval. Rev. \textbf{45}(3), 143–155 (2018)

\bibitem{vis}
Viswanath, P., Tse, D.N., Laroia, R.: Opportunistic beamforming using dumb
  antennas.
\newblock IEEE Trans. Inf. Theor. \textbf{48}(6), 1277–1294 (2006)

\bibitem{walr}
Walraevens, J., Leeuwaarden, J.S., Boxma, O.J.: Power series approximations for
  two-class generalized processor sharing systems.
\newblock Queueing Syst. \textbf{66}(2), 107–130 (2010)

\bibitem{za}
Zachary, S.: On two-dimensional {M}arkov chains in the positive quadrant with
  partial spatial homogeneity.
\newblock Markov Process. Relat. Fields \textbf{1}(1), 267--280 (1995)

\bibitem{zwart1}
Zhang, J., Dai, J.G., Zwart, B.: Diffusion limits of limited processor sharing
  queues.
\newblock Ann. Appl. Probab. \textbf{21}(2), 745--799 (2011)

\bibitem{zwart}
Zhang, J., Zwart, B.: Steady state approximations of limited processor sharing
  queues in heavy traffic.
\newblock Queueing Syst. \textbf{60}(3-4), 227--246 (2008)

\end{thebibliography}
%
%

\end{document}